\documentclass[pdflatex,sn-mathphys-num]{sn-jnl}% Math and Physical Sciences Numbered Reference Style
\usepackage{graphicx}%
\usepackage{multirow}%
\usepackage{amsmath,amssymb,amsfonts}%
\usepackage{amsthm}%
\usepackage{mathrsfs}%
\usepackage[title]{appendix}%
\usepackage{xcolor}%
\usepackage{textcomp}%
\usepackage{manyfoot}%
\usepackage{booktabs}%
\usepackage{algorithm}%
\usepackage{algorithmicx}%
\usepackage{algpseudocode}%
\usepackage{listings}%
\usepackage[utf8]{inputenc}
\usepackage{amsthm,amscd,cite,mathrsfs,graphicx,url,mathtools}
\usepackage{hyperref}
\usepackage{comment}
\usepackage{xparse,multirow,enumitem,array,pst-node,placeins,xspace}
\usepackage{tikz-cd,quiver}

\def\nameit#1{\textrm{#1}~}
\def\thex{\nameit{Theorem}}
\def\prox{\nameit{Proposition}}
\def\corx{\nameit{Corollary}}
\def\lemx{\nameit{Lemma}}
\def\defx{\nameit{Definition}}
\def\remx{\nameit{Remark}}

\def\conx{\nameit{Construction}}
\def\dfn#1{{\itshape #1}}
\newcommand{\refs}[1]{\textup{(}\ref{#1}\textup{)}}

\tikzset{iso/.style={draw=none,every to/.append style={edge node={node [sloped, allow upside down, auto=false]{$\cong$}}}}}
\tikzset{proarrowS/.style={draw=none,every to/.append style={edge node={node [sloped, allow upside down, auto=false]{\raisebox{1.6pt}{\small$\shortmid$}}}}}}
\tikzset{proarrow/.style={proarrowS,preaction={draw,->}}}
\newcolumntype{F}{>{$}c<{\hspace{-0.9ex}$}}
\newcolumntype{:}{>{$}m{0.8ex}<{$}}
\newcolumntype{R}{>{$}r<{$}}
\newcolumntype{C}{>{$}c<{$}}
\newcolumntype{L}{>{$}l<{$}}
\newcolumntype{N}{@{}>{$}l<{$}}
\NewDocumentEnvironment{cd}{s O{6} O{6} b}{%
	\IfBooleanF{#1}{\begin{equation*}}\begin{tikzcd}[row sep=#2ex,column sep=#3ex,ampersand replacement=\&]
			#4
		\end{tikzcd}\IfBooleanF{#1}{\end{equation*}}\ignorespacesafterend}{}

\newenvironment{enumT}{\begin{enumerate}[label=$($\hspace{-0.1ex}\roman*\hspace{0.13ex}$)$]}{\end{enumerate}}
\newenvironment{fun}{\[\begin{tabular}{F:RCL}}{\end{tabular}\]\ignorespacesafterend}
\newenvironment{eqD}[1]{\begin{equation}\label{#1}}{\end{equation}\ignorespacesafterend}
\newenvironment{eqD*}{\begin{equation*}}{\end{equation*}\ignorespacesafterend}

\def\:{\colon}
\def\c{\circ}
\newcommand{\iso}{\cong}
\def\phi{\varphi}
\DeclareMathOperator{\ps}{ps}
\DeclareMathOperator{\nps}{ps}
\DeclareMathOperator{\lax}{lax}
\DeclareMathOperator{\nlax}{lax}
\DeclareMathOperator{\noplax}{opl}
\DeclareMathOperator{\oplax}{opl}
\newcommand{\sqm}{(-)^{\to}}

\DeclareFontFamily{OT1}{pzc}{}
\DeclareFontShape{OT1}{pzc}{m}{it}{<->s*[1.19]pzcmi7t}{}
\DeclareMathAlphabet{\mathpzc}{OT1}{pzc}{m}{it}
\newcommand{\catfont}[1]{\mathpzc{#1}\xspace}

\newcommand{\B}{\mathcal{B}}
\newcommand{\C}{\mathcal{C}}
\newcommand{\D}{\mathcal{D}}
\newcommand{\E}{\mathcal{E}}
\newcommand{\F}{\mathcal{F}}
\newcommand{\cL}{\mathcal{L}}
\newcommand{\cR}{\mathcal{R}}
\newcommand{\dB}{\mathbb{B}}
\newcommand{\dC}{\mathbb{C}}
\newcommand{\dD}{\mathbb{D}}
\newcommand{\dE}{\mathbb{E}}
\newcommand{\dL}{\mathbb{L}}
\newcommand{\dR}{\mathbb{R}}
\newcommand{\K}{\catfont{K}}
\renewcommand{\S}{\catfont{S}}
\newcommand{\T}{\catfont{T}}

\newcommand{\Set}{\catfont{Set}}
\newcommand{\Cat}{\catfont{Cat}}
\newcommand{\VCat}[1]{{#1}\mbox{-}\Cat}
\newcommand{\Fact}{\catfont{Fact}}
\newcommand{\Factlax}{\catfont{Fact}_{\lax}}

\newcommand{\Factcolax}{\catfont{Fact}_{\operatorname{colax}}}
\newcommand{\Factstr}{\catfont{Fact}_{\operatorname{str}}}
\newcommand{\Factps}{\catfont{Fact}_{\operatorname{ps}}}
\newcommand{\Fib}{\catfont{Fib}}
\newcommand{\DblCat}{\catfont{DblCat}}
\newcommand{\DblCatnps}{\DblCat_{\nps}}
\newcommand{\DblCatnlax}{\DblCat_{\nlax}}
\newcommand{\DblCatnoplax}{\DblCat_{\noplax}}
\newcommand{\DblCatnpsnoplax}{\DblCat_{\nps/\noplax}}
\newcommand{\Prof}{\mathbb{P}\mathrm{rof}}

\newcommand{\Span}{\mathbb{S}\mathrm{pan}}
\newcommand{\Cospan}{\mathbb{C}\mathrm{ospan}}
\newcommand{\Rel}{\mathbb{R}\mathrm{el}}
\newcommand{\dEl}[1]{\mathbb{E}\mathrm{l}(#1)}
\newcommand{\dpath}{\mathbb{P}\mathrm{ath}}
\NewDocumentCommand{\Catint}{t+ t' O{n} m}{
\catfont{\IfBooleanT{#1}{Ps}Cat}
\ifx#3l{_{\lax}}\else{\ifx#3p{_{\ps}}\else{\ifx#3o{_{\oplax}}\fi}\fi}\fi\IfBooleanT{#2}{^{\operatorname{s}}}(#4)
}
\NewDocumentCommand{\Alg}{t+ O{n} m}{
#3\mbox{-}\catfont{\IfBooleanT{#1}{Ps}Alg}\ifx#2l{_{\lax}}\else{\ifx#2p{_{\ps}}\else{\ifx#2o{_{\oplax}}\fi}\fi}\fi}

\newcommand{\src}{\operatorname{src}}
\newcommand{\tgt}{\operatorname{tgt}}
\newcommand{\im}{\operatorname{Im}}
\newcommand{\HomC}[3]{{#1}\left({#2},\hspace{0.1ex}{#3}\right)}

\newcommand{\op}{^{\operatorname{op}}}
\newcommand{\tight}{_\tau}
\newcommand{\loose}{_\lambda}
\newcommand{\id}[1]{\operatorname{id}_{#1}}
\newcommand{\Id}[1]{\operatorname{Id}_{#1}}

\newcommand{\dofs}[2]{(({#1}_0, {#2}_0),({#1}_1, {#2}_1))}

\newcommand{\too}{\longrightarrow}
\newcommand{\mto}{\mapsto}

\newcommand{\ar}[2][]{\xrightarrow[#1]{#2}}
\makeatletter
\newcommand{\aR}[2][]{%
	\ext@arrow 0055{\Rightarrowfill@}{#1}{#2}}
\def\xmapstofill@{\arrowfill@{\mapstochar\relbar}\relbar\rightarrow}
\newcommand{\am}[2][]{%
	\ext@arrow 0395\xmapstofill@{#1}{#2}}
\def\aitofill@{\arrowfill@{\lhook\joinrel\relbar}\relbar\rightarrow}
\newcommand{\aito}[2][]{%
	\ext@arrow 3095\aitofill@{#1}{#2}}
\def\atoefill@{\arrowfill@\relbar\relbar\twoheadrightarrow}
\newcommand{\atoe}[2][]{%
	\ext@arrow 3095\atoefill@{#1}{#2}}
\def\xprorightarrowfill@{\arrowfill@{\relbar\joinrel\raisebox{0.6pt}{\small$\shortmid$}\joinrel{\relbar}}\relbar\rightarrow}
\newcommand{\aproarrow}[2][]{%
	\ext@arrow 0099\xprorightarrowfill@{#1}{#2}}
\newcommand{\proarrow}{\aproarrow{}}
\makeatother
\newcommand{\PB}[1]{\arrow[#1,phantom,"\scalebox{1.6}{\color{black}$\lrcorner$}",very near start]}

\newcommand{\dom}{\operatorname{dom}}
\newcommand{\cod}{\operatorname{cod}}

\NewDocumentCommand{\sq}{s O{n} O{6} O{6} O{} O{2.7} O{2.2} O{0.5} O{n}}{%
	\def\foosq##1##2##3##4##5##6##7##8{%
		\IfBooleanTF{#1}{\begin{cd}*}{\begin{cd}}[#3][#4]
				{##1}\ifx#2p{\PB{rd}}\fi\arrow[r,"{##5}"]\ifx#9l{\arrow[d,equal,"{##6}"']}\else{\ifx#2d{\arrow[d,"{##6}\hspace{0.4ex}"',proarrow]}\else{\arrow[d,"{##6}"']}\fi}\fi\&{##2}\ifx#9r{\arrow[d,equal,"{##7}"]}\else{\ifx#2d{\arrow[d,"\hspace{0.4ex}{##7}",proarrow]}\else{\arrow[d,"{##7}"]}\fi}\fi\ifx#2l{\arrow[ld,Rightarrow,shorten <=#6ex,shorten >=#7ex,"{#5}"{pos=#8}]}\fi\\
				{##3}\ifx#9d{\arrow[r,equal,"{##8}"']}\else{\arrow[r,"{##8}"']}\fi\ifx#2o{\arrow[ur,Rightarrow,shorten <=#6ex,shorten >=#7ex,"{#5}"{pos=#8}]}\fi\ifx#2d{\arrow[ur,phantom,"{#5}"{pos=#8}]}\fi\&{##4}
		\end{cd}}%
		\foosq}
\NewDocumentCommand{\sqleft}{s O{n} O{6} O{6} O{} O{2.7} O{2.2} O{0.5} O{n}}{%
	\def\foosqleft##1##2##3##4##5##6##7##8{%
		\IfBooleanTF{#1}{\begin{cd}*}{\begin{cd}}[#3][#4]
				{##1}\ifx#2p{\PB{rd}}\fi\arrow[r,"{##5}"]\ifx#9l{\arrow[d,equal,"{##6}"']}\else{\ifx#2d{\arrow[d,twoheadrightarrow,"{##6}\hspace{0.4ex}"',proarrow]}\else{\arrow[d,twoheadrightarrow,"{##6}"']}\fi}\fi\&{##2}\ifx#9r{\arrow[d,equal,"{##7}"]}\else{\ifx#2d{\arrow[d,twoheadrightarrow,"\hspace{0.4ex}{##7}",proarrow]}\else{\arrow[d,twoheadrightarrow,"{##7}"]}\fi}\fi\ifx#2l{\arrow[ld,Rightarrow,shorten <=#6ex,shorten >=#7ex,"{#5}"{pos=#8}]}\fi\\
				{##3}\ifx#9d{\arrow[r,equal,"{##8}"']}\else{\arrow[r,"{##8}"']}\fi\ifx#2o{\arrow[ur,Rightarrow,shorten <=#6ex,shorten >=#7ex,"{#5}"{pos=#8}]}\fi\ifx#2d{\arrow[ur,phantom,"{#5}"{pos=#8}]}\fi\&{##4}
		\end{cd}}%
		\foosqleft}
\NewDocumentCommand{\sqright}{s O{n} O{6} O{6} O{} O{2.7} O{2.2} O{0.5} O{n}}{%
	\def\foosqright##1##2##3##4##5##6##7##8{%
		\IfBooleanTF{#1}{\begin{cd}*}{\begin{cd}}[#3][#4]
				{##1}\ifx#2p{\PB{rd}}\fi\arrow[r,"{##5}"]\ifx#9l{\arrow[d,equal,"{##6}"']}\else{\ifx#2d{\arrow[d,hookrightarrow,"{##6}\hspace{0.4ex}"',proarrow]}\else{\arrow[d,hookrightarrow,"{##6}"']}\fi}\fi\&{##2}\ifx#9r{\arrow[d,equal,"{##7}"]}\else{\ifx#2d{\arrow[d,hookrightarrow,"\hspace{0.4ex}{##7}",proarrow]}\else{\arrow[d,hookrightarrow,"{##7}"]}\fi}\fi\ifx#2l{\arrow[ld,Rightarrow,shorten <=#6ex,shorten >=#7ex,"{#5}"{pos=#8}]}\fi\\
				{##3}\ifx#9d{\arrow[r,equal,"{##8}"']}\else{\arrow[r,"{##8}"']}\fi\ifx#2o{\arrow[ur,Rightarrow,shorten <=#6ex,shorten >=#7ex,"{#5}"{pos=#8}]}\fi\ifx#2d{\arrow[ur,phantom,"{#5}"{pos=#8}]}\fi\&{##4}
		\end{cd}}%
		\foosqright}
\NewDocumentCommand{\fib}{O{n} O{2.3} mmm}{%
	\begin{cd}*[#2][5]
		{#3}\ifx#1n{\arrow[d,"{\,\scaleu{#4}}"]}\else{\ifx#1i{\arrow[d,hookrightarrow,"{\,\scaleu{#4}}"]}\else{\ifx#1e{\arrow[d,equal,"{\,\scaleu{#4}}"]}\else{\ifx#1R{\arrow[d,Rightarrow,"{\,\scaleu{#4}}"]}\fi}\fi}\fi}\fi\\
		{#5}\ifx#1o{\arrow[u,"{\,\scaleu{#4}}"']}\fi
	\end{cd}\xspace
}
\theoremstyle{thmstyleone}%
\newtheorem{theorem}{Theorem}[section]
\newtheorem{proposition}[theorem]{Proposition}
\newtheorem{corollary}[theorem]{Corollary}
\newtheorem{lemma}[theorem]{Lemma}
\newtheorem{thm}[theorem]{Theorem}
\newtheorem*{thm*}{Theorem}

\theoremstyle{thmstyletwo}%

\theoremstyle{thmstylethree}%
\newtheorem{definition}[theorem]{Definition}
\newtheorem{remark}[theorem]{Remark}
\newtheorem{rmk}[theorem]{Remark}
\newtheorem{example}[theorem]{Example}
\newtheorem{construction}[theorem]{Construction}
\newtheorem{notation}[theorem]{Notation}

\newtheorem{observation}[theorem]{Observation}

\begin{document}

\title[Double Orthogonal Factorization Systems]{Double Orthogonal Factorization Systems\\{
\large
 {\it In honour of Robert Par\'e on the occasion of his 80th birthday.}}}


\author[1]{\fnm{C. B.} \sur{Aberl\'e}}\email{caberle@andrew.cmu.edu}

\author[2]{\fnm{Elena} \sur{Caviglia}}\email{elena.caviglia@outlook.com}

\author*[3]{\fnm{Matthew} \sur{Kukla}}\email{matthew.kukla@gwmail.gwu.edu}

\author[4]{\fnm{Ruben} \sur{Maldonado}}\email{rubasmh@matmor.unam.mx}

\author[5]{\fnm{Luca} \sur{Mesiti}}\email{luca.mesiti@outlook.com}

\author[6]{\fnm{Dorette} \sur{Pronk}}\email{dorette.pronk@dal.ca}

\author[7]{\fnm{Tanjona} \sur{Ralaivaosaona}}\email{t.f.r.ralaivaosaona@liacs.leidenuniv.nl}

\affil[1]{\orgname{Carnegie-Mellon University}, \orgaddress{ \city{Pittsburgh,} \country{USA}}}

\affil[2]{\orgname{National Institute for Theoretical and Computational Sciences}, \orgaddress{\country{South Africa}}}

\affil[3]{\orgname{The George Washington University}, \orgaddress{\city{District of Columbia,} \country{USA}}}

\affil[4]{ \orgname{Universidad Nacional Aut\'onoma de M\'exico}, \orgaddress{\country{M\'exico}}}

\affil[5]{\orgname{Stellenbosch University, South Africa}, \orgaddress{\country{South Africa}}}

\affil[6]{ \orgname{Dalhousie University}, \orgaddress{\city{Halifax}, \country{Canada}}}

\affil[7]{ \orgname{Leiden Institute of Advanced Computer Science}, \orgaddress{\city{Leiden University}, \country{The Netherlands}}}


\abstract{We define strict and lax orthogonal factorization systems on double categories. These consist of an orthogonal factorization system on arrows and one on double cells that are compatible with each other. Our definitions are motivated by several explicit examples, including factorization systems on double categories of spans, relations and bimodules. We then prove monadicity results for orthogonal factorization systems on double categories in order to justify our definitions. For fibrant double categories we discuss the structure of the double orthogonal factorization systems that have a given orthogonal factorization system on the arrows in common. Finally, we study the interaction of orthogonal factorization systems on double categories with double fibrations.}

\keywords{double categories, double fibrations, factorization systems, monads}


\pacs[MSC Classification]{18N10, 18A32, 18B10, 18C15}

\maketitle

\section{Introduction}\label{sec1}
Orthogonal factorization systems (OFS) for ordinary categories were introduced by Isbell \citep{isbell1957}, Freyd and Kelly \citep{freyd1972categories} and Bousfield \citep{bousfield1977constructions}. An orthogonal factorization system provides a category with a notion of \emph{images} along its arrows, generalizing the usual image factorization on categories such as $\mathbf{Set}$. For instance, various factorization systems on the category of categories give rise to the following images for functors: the replete images, the essential image and the full image.
This concept of image is integral to the foundational concepts of homological algebra.
One of the most well-known applications of factorization systems (both of orthogonal ones and of weak algebraic ones) is in homotopy theory, as part of Quillen model structures.  More generally, orthogonal factorization  systems reveal important properties of the structure of a category, and may assist in defining a functor on a category. The orthogonal factorization system can be viewed as a type of distributive law for the interaction between the two classes of arrows: if one can compose the arrows in each class and has a way of re-factorizing a composition of an arrow in the left class after an arrow in the right class, one has now a composition law for the whole category.

A related aspect of orthogonal factorization systems is their interaction with \emph{fibrations} between categories. Each fibration gives rise to what has been called  a Cartesian factorization system on its domain: the left class consists of the vertical arrows in the fibration (and satisfies the 3-for-2 property) and the right class consists of the Cartesian arrows. This result can be generalized to show that fibrations have the property that one can lift any factorization system in the base to a factorization system in the total category in a canonical way.

In this paper we will introduce a notion of orthogonal factorization system for double categories that generalizes that for ordinary categories. We will call these {\em double orthogonal factorization systems} (DOFS). We will show how a number of well-known orthogonal factorization systems generalizes to double orthogonal factorization systems.

With the recent introduction of the notion of a fibration between double categories, we believe it is time to develop an analogous notion of orthogonal factorization system for double categories that interacts well with the double fibrations. 

Motivated by examples of factorization systems on the double categories of quartets and spans in an arbitrary category $\C$ namely $\mathbb{S}\mathrm{q} (\mathcal{C})$ and $\Span(\C)$ respectively, as well as the double category $\Rel(\C)$ of relations in a regular category $\C$, we see how the orthogonal factorization systems on $\C$ can be lifted to factorization systems on the arrows and on the double cells of the double categories, in a way that makes the two factorization systems suitably compatible. 

This leads us to our internal definition of a double orthogonal factorization system (DOFS): a pseudo category object in a suitable category of categories equipped with an orthogonal factorization system and a choice of factorizations. We require that the source, target and unit functors of such category objects preserve the factorizations, but the composition functor in the category diagram can be more `loose'. If the composition functor preserves both classes, we call this a DOFS (our default is `pseudo'), if it preserves the factorizations we will call it a strict DOFS, if it preserves only the right class we call it a lax DOFS, and if it preserves only the left class we call it a colax DOFS.

Note the similarity with the notion of double fibration \citep{df}. For double fibrations one takes internal pseudo categories in the category of fibrations with a chosen cleavage but whose arrows only preserve cartesian arrows; however, one requires that the source, target and unit functors in a double fibration diagram preserve the chosen cleavages. The reason for this restriction is the same as it is for us: we require the existence of the pullback of the source along the target as an object in our category. 

Classically, a good way to capture and understand orthogonal factorization systems and morphisms between them is in terms of Eilenberg-Moore algebras for a particular 2-monad. The first result in this direction was published by Coppey \citep{coppey-thesis} showing that strict factorization systems can be viewed as strict algebras for the squaring 2-monad. The full result characterizing orthogonal factorization systems with a chosen factorization as the pseudo-algebras for this 2-monad was established by Korostenski and Tholen in \citep{kt93}. This result leads us to four natural notions of morphism between categories with an OFS: the strict algebra maps preserve the chosen factorizations, the pseudo maps preserve both classes of morphisms, the lax maps preserve the right class and the oplax maps preserve the left class. This provides us with four 2-categories in which to take internal pseudo-categories in order to define a corresponding notion of DOFS.
We illustrate this with several examples in each category of algebras.

However, there is a different way to approach the generalization from OFSs to DOFSs. In Theorem \ref{teorgenmonadicity} we establish a monadicity result that intuitively say that suitable 2-monads on a 2-category give rise to a 2-monad on the 2-category of internal pseudo-categories with internal (pseudo or oplax) functors as arrows, and  levelwise (also called, internal) transformations as 2-cells, with the property that the 2-category of  algebras for the 2-monad on the internal categories is isomorphic to the 2-category of internal categories in the algebras of the original 2-monad.
To be precise,

\begin{theorem}[{Theorem \ref{teorgenmonadicity}}] 
    Let $\K$ be a 2-category and let $(T\: \K\to \K, \eta\:\Id{}\aR{}T, \mu\:T^2\aR{} T)$ be a 2-monad whose underlying 2-functor preserves pullbacks. Then $T$ induces 2-monads
    $$\overline{T}_{\oplax}\:\Catint+[o]{\K}\to \Catint+[o]{\K}$$
    $$\overline{T}_{\ps}\:\Catint+[p]{\K}\to \Catint+[p]{\K}$$
    such that
    $$\Alg+[l]{\overline{T}_{\oplax}}\iso \Catint+'[o]{\Alg+[l]{T}}$$
    $$\Alg+[p]{\overline{T}_{\ps}}\iso \Catint+'[p]{\Alg+[p]{T}}$$
    where $\Catint+'[o]{\Alg+[l]{T}}$ (respectively, $\Catint+'[p]{\Alg+[p]{T}}$) is the full sub-2-category of $\Catint+[o]{\Alg+[l]{T}}$ (respectively, $\Catint+[p]{\Alg+[p]{T}}$) on those pseudo-categories whose source, target and identity assignment are strict morphisms between normal pseudo-algebras.
\end{theorem}

The intuitive idea behind this theorem is that, interestingly, objects, arrows and 2-cells on the two sides of the isomorphisms are presented by the same data, grouped with each other in two different ways. When applied to the 2-category of categories, this general monadicity theorem allows us to view our various flavours of double categories with a DOFS as pseudo algebras for 2-monads defined by sending a double category $\dD$ to its cotensor with a single arrow double category, $\dD^\to$.

This lets us establish the correct notions of arrows between double categories with a DOFS.
The natural notion of 2-cell coming from the algebra perspective is that of an internal, or levelwise, transformation. These are the transformations with components that are arrows in the strict direction of the double category. We simply take all such transformations between the appropriate (lax, colax or pseudo) morphisms. However, we can also add transformations in the proarrow direction: these correspond to suitable functors from the object category of the domain to the arrow category of the codomain, in a manner compatible with the OFS structures on the underlying categories.

We end the paper with a large variety of examples. We show how extensions and restrictions in fibrant double categories allow us to lift each orthogonal factorizatiom system on its category of arrows to two double ortogonal factorization systems. We also show that these are initial and terminal among all possible liftings, and illustrate this in the double categories of spans, relations and profunctors. We discuss the relationship between double orthogonal factorization systems and double fibrations and show how a double factorization system on the base can be lifted to one on the total double category.  In particular for the double category of spans we have four distinct double orthogonal factorization systems extending the epi-mon factorization system on sets.

\subsection*{Outline of the paper}

Section \ref{sectioncategoriesoffs} recalls the definition of orthogonal factorization systems for ordinary categories, as well as their monadicity result. We then apply such monadicity result to discuss strict, pseudo, lax and oplax arrows and 2-cells between categories equipped with an orthogonal factorization systems. This provides us with various 2-categories of categories with an OFS. We also establish a useful result on pullbacks of categories with an OFS.

In Section \ref{sectionDOFS}, after presenting our motivating examples (the double categories of quartets, relations and spans), we give the internal definition of a double orthogonal factorization system (as a suitable internal category in one of the 2-categories of categories with an OFS). We then present an equivalent characterization in terms of cells, and conclude the section with examples related to the double category of quintets.

Section \ref{sectionmonadicity} establishes the monadicity of double categories equipped with a double orthogonal factorization systems as algebras for an $\F$-monad. This sheds light on the various flavours (strict, pseudo, lax and oplax) of double orthogonal factorization systems, naturally forming various 2-categories of DOFS. However, note that in the larger world of double categories there are more arrows and 2-cells available, placing all of these 2-categories inside a larger double category.

In Section \ref{appsandex} we first prove some general results about DOFS on fibrant double categories. We then explore a list of examples (modules, profunctors, quantale-valued relations, and spans) of fibrant double categories with multiple DOFS lifting a common OFS on the arrows.

Section \ref{doublefibs} discusses the lifting of double orthogonal factorization systems along double fibrations. We show many examples, including explicit descriptions of (pseudo-vertical, cartesian) factorization systems and the way they interact with the double factorization systems introduced in previous sections.

\subsection*{Notation}

Given a 2-category $\K$, we will write $\Catint+[p]{\K}$ (respectively, $\Catint+[l]{\K}$, $\Catint+[o]{\K}$) for the 2-category of normal (i.e, unitary) pseudo-categories internal to $\K$, with normal (i.e, unitary) internal pseudo (respectively, normal lax, normal oplax) functors and internal, also called levelwise, natural transformations; we will review these concepts in detail in  \defx\ref{recallinternalcats}.

Given a monad $T\colon\C\to \C$ on a category $\C$, we will write $\Alg{T}$ for the category of Eilenberg-Moore $T$-algebras and algebra morphisms.

Given a 2-monad $T\colon\K\to \K$ on a 2-category $\K$, we will denote by $\Alg+[s]{T}$ (respectively, $\Alg+[p]{T}$, $\Alg+[l]{T}$) the 2-category of normal pseudo-algebras with strict (respectively, pseudo, lax) morphisms and 2-cells (i.e, $T$-transformations).

\section{Categories of factorization systems}\label{sectioncategoriesoffs}

In this section, we first recall ordinary orthogonal factorization systems and their monadicity over the 2-category $\Cat$ of categories. 
We then extract from this monadicity result three flavours of morphisms between categories with an orthogonal factorization system. 
Finally, we study the corresponding 2-categories of categories equipped with an orthogonal factorization system.

\begin{definition}\label{defnofs}
    Let $\catfont{C}$ be a category.  An \dfn{orthogonal factorization system} on $\catfont{C}$ consists of two classes of morphisms 
    $\catfont{L}$ and  $\catfont{R}$ such that:
    \begin{itemize}
        \item The classes $\catfont{L}$ and $\catfont{R}$ are closed under composition, and contain all
        isomorphisms.
        \item Every morphism $f\colon A \to B$ in $\catfont{C}$ can be written as $f=r_f \c \ell_f$ where $r_f \in \catfont{R}$ and $ \ell_f \in \catfont{L}$, giving commutativity of the following diagram:
        \[\begin{tikzcd}
	A && B \\
	& {\text{Im}(f)}
	\arrow["f", from=1-1, to=1-3]
	\arrow["\ell_f"', from=1-1, to=2-2]
	\arrow["r_f"', from=2-2, to=1-3]
\end{tikzcd}\]
        \item For any commutative square
        \[\begin{tikzcd}
	\bullet & \bullet \\
	\bullet & \bullet
	\arrow["f", from=1-1, to=1-2]
	\arrow["u"', from=1-1, to=2-1]
	\arrow["v", from=1-2, to=2-2]
	\arrow["\exists ! \hspace{0.1ex} e"', dashed, from=2-1, to=1-2]
	\arrow["g"', from=2-1, to=2-2]
\end{tikzcd}\]
where $u \in \catfont{L}, v \in \catfont{R}$, there exists a unique $e$ such that $g=v\c e$ and $f=e\c u$.
    \end{itemize}
\end{definition}

The target object of $\ell_f$ (equivalently, the source of $r_f$) is often called the \emph{image} of $f$, with respect to $(\cL,\cR)$.
Note that the image is unique up to unique isomorphism. When we choose a specific factorization for each arrow, the object $\im(f)$ is uniquely determined.

\begin{definition}
    A \emph{functorial factorization system} on $\mathcal{C}$ is an assignment, to each morphism $f \in \text{Mor}(\mathcal{C})$, 
    of a pair of morphisms $\ell_f, r_f$ such that $f = r_f \circ \ell_f$, that extends to a functor 
    $Q\colon\mathcal{C}^{\rightarrow} \to \mathcal{C}$ (where $\mathcal{C}^{\rightarrow}$ is the arrow category of $\mathcal{C}$), 
    mapping every $f$ to its image $Q(f)$.

\[\begin{tikzcd}[ampersand replacement=\&]
	\bullet \& {Q(f)} \& \bullet \\
	\bullet \& {Q(g)} \& \bullet
	\arrow["{\ell_f}", from=1-1, to=1-2]
	\arrow["f", curve={height=-24pt}, from=1-1, to=1-3]
	\arrow["u"', from=1-1, to=2-1]
	\arrow["{r_f}", from=1-2, to=1-3]
	\arrow["\exists !"',"{Q(u,v)}", dashed, from=1-2, to=2-2]
	\arrow["v", from=1-3, to=2-3]
	\arrow["{\ell_g}"', from=2-1, to=2-2]
	\arrow["g"', curve={height=24pt}, from=2-1, to=2-3]
	\arrow["{r_g}"', from=2-2, to=2-3]
\end{tikzcd}\]

Note that every OFS with a choice of a factorization for each arrow is a functorial factorization system.
\end{definition}

\begin{rmk}\label{monadic-dfn}
 In this paper, when we say that a category has an OFS, we will always mean that it is also equipped with a choice of factorizations $f=r_f\circ \ell_f$ for every arrow $f$ in the category. (Note that this is analogous to choosing a cleavage for a fibration.) We will require that for each identity arrow $\id{A}$ we have $\ell_{\id{A}}=\id{A}=r_{\id{A}}$.
\end{rmk}

In \citep{kt93}, Korostenski and Tholen proved a 2-monadicity result for categories with an orthogonal factorization system. The 2-monad they considered is the squaring 2-monad $\sqm$ on $\Cat$, sending any category to its arrow category (also known as its category of commutative squares). More precisely, they proved the following.

\begin{theorem}[\hspace{1sp}\citep{kt93}]\label{theorkorostenskitholen}
    Orthogonal factorization systems, with a choice of factorizations for the arrows as in Remark \ref{monadic-dfn}, are the normal pseudo (Eilenberg-Moore) algebras for the squaring 2-monad $\sqm\:\Cat\to \Cat$.
\end{theorem}

\begin{remark}\label{remdetailsmonadicity}
    Explicitly, the normal pseudo-algebra structure map associated to an orthogonal factorization system $(\cL,\cR)$ on a category $\C$ is the functor $Q\:\C^{\to}\to \C$ that sends every $(A\ar{f}B)\in \C^{\to}$ to the middle object of its chosen factorization $A\atoe{\ell_f} Q(f)\aito{r_f} B$, and sends every commutative square as on the left below to the unique morphism $Q(u,v)$ on the right below induced by the functoriality of the orthogonal factorization system:
    \begin{eqD*}
        \sq*{A}{C}{B}{D}{u}{f}{g}{v} \qquad\qquad \begin{cd}*[3][6]
	A \arrow[d,twoheadrightarrow,"{\ell_f}"']\arrow[r,"{u}"] \& C \arrow[d,twoheadrightarrow,"{\ell_g}"]\\
	Q(f) \arrow[d,hookrightarrow,"{r_f}"']\arrow[r,dashed,"{Q(u,v)}"] \& Q(g) \arrow[d,hookrightarrow,"{r_g}"] \\
	B \arrow[r,"{v}"'] \& D
\end{cd}
    \end{eqD*}
Interestingly, the functor $Q$ contains all the data of the orthogonal factorization system $(\cL,\cR)$. The normality condition precisely imposes that $Q$ is a (strict) retraction of the diagonal $\C\to \C^{\to}$, which sends $A$ to $\id{A}$ and $f\:A\to B$ to $(f,f)$. Considering the factorization
\begin{cd}[5.5][5.5]
	A \arrow[r,"{f}"] \arrow[d,equal,"{}"] \& B \arrow[d,equal,"{}"{name=B}] \& A \arrow[d,equal,"{}"'{name=A}] \arrow[r,equal,"{}"] \& A \arrow[r,"{f}"] \arrow[d,"{f}"]\& B\arrow[d,equal,"{}"] \\
	A \arrow[r,"{f}"'] \& B \& A \arrow[r,"{f}"'] \& B \arrow[r,equal,"{}"]\& B
	\arrow[from=B, to=A,equal,shorten <=3ex, shorten >=3ex]
\end{cd}
in $\C^{\to}$, one then finds that
$$Q(\id{A},f)=\ell_f, \qquad Q(f,\id{B})=r_f, \qquad f=Q(f,f)=r_f\c \ell_f$$
provided that the choice of factorization for the identity arrows is $\ell_{\id{A}}=\id{A}=r_{\id{A}}$. We will always assume that in this paper.
The morphisms in $\cL$ are precisely those $f$ for which $Q(f,\id{})$ is an isomorphism, and the morphisms in $\cR$ are precisely 
those $f$ for which $Q(\id{},f)$ is an isomorphism.

As was noted in \citep{Grandis200217}, $\C^\to$ is the free category with an orthogonal factorization system on a given category $\C$. 
Its factorization system is given as follows: the left class consists of the morphisms of the form $(i,u)$ with $i$ iso and $u$ any morphism in $\C$, 
and the right class consists of the morphisms of the form $(v,j)$ with $v$ any morphism in $\C$ and $j$ iso. A general morphism $(u,v)$ then 
factorizes as $(u,\id{})\c (\id{},v)$. This is indeed the free algebra on $\C$ with respect to the squaring 2-monad $\sqm$. One can in fact check 
that the multiplication $\mu_{\C}\:{(\C^\to)}^\to\to \C^\to$ of the monad precisely translates to the left and right classes described above.
\end{remark}

\begin{remark}\label{remLR}
Orthogonal factorization systems are in particular algebraic weak factorization systems, which were introduced in \citep{grandistholen} with the name ``natural weak factorization systems". As a consequence, we obtain the further monadicity results shown in \prox\ref{propfurthermonadicity}.

We can extract from $Q\:\C^{\to}\to \C$ two other functors $L,R\:\C^{\to}\to\C^{\to}$, sending 
\begin{eqD*}
    \sq*[5][5]{A}{C}{B}{D}{u}{f}{g}{v} \hspace{0.4ex} \am{L}\hspace{0.3ex}
    \sqleft*[5][5]{A}{C}{Q(f)}{Q(g)}{u}{\ell_f}{\ell_g}{Q(u,v)}\qquad\hspace{0.5ex}
    \sq*[5][5]{A}{C}{B}{D}{u}{f}{g}{v} \hspace{0.4ex} \am{R}\hspace{0.3ex}
    \sqright*[5][5]{Q(f)}{Q(g)}{B}{D}{Q(u,v)}{r_f}{r_g}{v}
\end{eqD*}
\end{remark}

\begin{proposition}[\hspace{1sp}\citep{grandistholen}]\label{propfurthermonadicity}
    Let $(\cL,\cR)$ be an orthogonal factorization system on a category $\C$. Its associated functors $L,R\:\C^{\to}\to\C^{\to}$ of \remx\ref{remLR} extend respectively to an idempotent comonad and to an idempotent monad. The left class $\cL$, viewed as a full subcategory of $\C^{\to}$, is precisely the category of Eilenberg-Moore coalgebras for $L$; as a consequence, it has all colimits, created by the projection to $\C^{\to}$. The right class $\cR$, viewed as a full subcategory of $\C^{\to}$, is precisely the category of Eilenberg-Moore algebras for $R$; as a consequence, it has all limits, created by the projection to $\C^{\to}$.
\end{proposition}

We can now extract from \thex\ref{theorkorostenskitholen} four flavours of morphisms between categories with an orthogonal factorization system (see also Remark \ref{monadic-dfn}), corresponding to   pseudo, lax, colax, and strict morphisms between normal pseudo-algebras.  Each flavour of factorization yields a 2-category of categories equipped with an OFS, whose 1-arrows are given by morphisms of the corresponding type, and 2-cells are natural transformations.

\begin{definition}\label{defmorfact}
    A {\em morphism} of categories with an orthogonal factorization system from $(\C,(\cL_{\C}, \cR_{\C}))$ to $(\D,(\cL_{\D}, \cR_{\D}))$ is a functor $F\:\C\to \D$ that preserves both the left and the right classes; i.e, that sends morphisms in $\cL_{\C}$ to morphisms in $\cL_{\D}$ and morphisms in $\cR_{\C}$ to morphisms in $\cR_{\D}$.
    
    A {\em lax morphism} of categories with an orthogonal factorization system from $(\C,(\cL_{\C}, \cR_{\C}))$ to $(\D,(\cL_{\D}, \cR_{\D}))$ 
    is a functor $F\:\C\to \D$ that preserves the right class: $F(\cR_{\C})\subseteq\cR_{\D}$.  

    A {\em colax morphism} of categories with an orthogonal factorization system from $(\C,(\cL_{\C}, \cR_{\C}))$ to $(\D,(\cL_{\D}, \cR_{\D}))$ is 
    a functor $F\:\C\to \D$ that preserves the left class: $F(\cL_{\C})\subseteq\cL_{\D}$.

    A {\em strict morphism} of categories with an orthogonal factorization system (with a choice of factorizations) from $(\C,(\cL_{\C}, \cR_{\C}))$ to $(\D,(\cL_{\D}, \cR_{\D}))$ is a functor $F\:\C\to \D$ that preserves the chosen factorizations on the nose: $\ell_{F(f)}=F(\ell_f)$ and $r_{F(f)}=F(r_f)$ for all $f$ in $\C_1$.
\end{definition}

\begin{notation}\label{notationFact}
   We obtain the following 2-categories of categories equipped with an OFS, from the flavours of morphisms above:
\begin{itemize}
    \item $\Fact$, whose 1-arrows are morphisms of categories with an orthogonal factorization system;
    \item $\Factlax$, whose 1-arrows are lax morphisms;
    \item $\Factcolax$, whose 1-arrows are colax morphisms;
    \item $\Factstr$, whose 1-arrows are strict morphisms.
\end{itemize}
In all of them we use all natural transformations as 2-cells.
\end{notation}

\begin{remark}\label{rempreservefactorizationslaxly}
    Let $F\:(\C,(\cL_{\C}, \cR_{\C})) \to (\D,(\cL_{\D}, \cR_{\D}))$ be a lax morphism of categories with an orthogonal factorization system. Call $Q_{\C}$ and $Q_{\D}$ the normal pseudo-algebra maps corresponding to $(\cL_{\C}, \cR_{\C})$ and $(\cL_{\D}, \cR_{\D})$ and their chosen factorizations respectively. Then $F$ preserves the factorizations of morphisms in $\C$ in a lax way. More precisely, given a morphism $f\:A\to B$ in $\C$ and considering its factorization
    \begin{cd}[2.5][6]
        A \arrow[rr,"{f}"] \arrow[rd,twoheadrightarrow,"{\ell_f}"']\& \& B \\
        \& Q_{\C}(f) \arrow[ru,hookrightarrow,"{r_f}"']
    \end{cd}
    with $\ell_f\in \cL$ and $r_f\in \cR$, there exists a unique morphism $\rho_f\:Q_{\D}(F(f))\to F(Q_{\C}(f))$ which makes the following diagram 
    commute:
    \begin{cd}[4.5][6]
    	{F(A)} \& Q_{\D}(F(f)) \& {F(B)} \\
    	\& {F(Q_{\C}(f))}
    	\arrow[twoheadrightarrow,"{\ell_{F(f)}}", from=1-1, to=1-2]
    	\arrow["{F(\ell_{f})}"', from=1-1, to=2-2]
    	\arrow[hookrightarrow,"{r_{F(f)}}", from=1-2, to=1-3]
    	\arrow[dashed,"{\rho_f}", from=1-2, to=2-2]
    	\arrow[hookrightarrow,"{F(r_{f})}"', from=2-2, to=1-3]
    \end{cd}
    Indeed $F(r_f)\in \cR_{\D}$, as $F$ preserves the right class, and we obtain the unique lifting $\rho_f$ by definition of orthogonal 
    factorization system. By uniqueness the morphisms $\rho_f$ form the components of a natural transformation $\rho\:Q_{\D}F^{\rightarrow}\Rightarrow FQ_{\C}$.

    Analogously, a colax morphism $F\:(\C,(\cL_{\C}, \cR_{\C})) \to (\D,(\cL_{\D}, \cR_{\D}))$ of categories with with an orthogonal factorization system gives 
    rise to  a natural transformation $\lambda\:FQ_{\C}\Rightarrow Q_{\D}F^{\rightarrow}$.

    If $F$ is a morphism of categories with an orthogonal factorization system, then for every $f\:A\to B$ in $\C$ the morphism $\rho_f$ 
    is an isomorphism. So in this case $F$ preserves the factorizations up to isomorphism.
    
    Notice then that a strict morphism $F$ of categories with an orthogonal factorization system preserves in particular both the left and the right classes. Indeed, a morphism is in the left class if and only if the right part of its factorization is an isomorphism, and $F$ preserves the factorizations.
\end{remark}

\begin{proposition}\label{propmorphismsofalgebrassqmonad}
    Let $(\C,(\cL_{\C}, \cR_{\C}))$ and $(\D,(\cL_{\D}, \cR_{\D}))$ be categories with an orthogonal factorization system and corresponding normal pseudo-algebra maps $Q_\C$ and $Q_\D$. 
    For a functor $F\:\C\to \D$, the following are equivalent:
    \begin{enumT}
        \item $F$ is a lax morphism (respectively, morphism, strict morphism) of categories with an orthogonal factorization system;
        \item $F$ extends to a lax morphism (respectively, pseudo morphism, strict morphism) between normal pseudo-algebras for the squaring 2-monad $\sqm$ on $\Cat$; 
        i.e, there exists a natural transformation (respectively, a natural isomorphism, an identity)
        \begin{eqD}{diagramrho}
            \sq*[l][5.5][5.5][\rho]{\C^{\to}}{\D^{\to}}{\C}{\D}{F^{\to}}{Q_\C}{Q_\D}{F}
        \end{eqD}
        such that for every $A\in \C$ we have $\rho_{\id{A}}=\id{F(A)}$ and, for every commutative square in $\C$ as on the left below, the pentagon on the right below commutes:
        \begin{eqD*}
    	\sq*[n][5][5]{A}{C}{B}{D}{u}{f}{g}{v} \quad 
    	\begin{cd}*[2.2][1.2]
    		Q_\D(F(g\c u)) \arrow[rr,"{\rho_{g\c u}}"] \arrow[d,iso,"{}"] \&\& F(Q_\C(g\c u))\arrow[d,iso,"{}"] \\
    		Q_\D(Q_\D(F(u),F(v)))\arrow[rd,"{Q_\D(\rho_f,\rho_g)}"'{pos=0.67,inner sep=0.3ex}] \&\& F(Q_\C(Q_\C(u,v))) \\[-1.7ex]
    		\& Q_\D(F(Q_\C(u,v))) \arrow[ru,"{\rho_{Q_\C(u,v)}}"'{pos=0.4,inner sep=0.3ex},shorten <=-0.33ex]
    	\end{cd}    
        \end{eqD*}
        where the isomorphisms are given by the normal pseudo-algebra structures;
        \item (based on \citep[Section I.6.2]{riehl_phdthesis}) there exists a natural transformation (respectively, a natural isomorphism, an identity) 
        $\rho$ as in \refs{diagramrho} such that $(1,\rho)\:L_\D\c F\to F\c L_\C$ is a lax morphism of comonads and $(\rho,1)\:R_\D\c F\to F\c R_\C$ 
        is a lax morphism of monads, where $L_\C,R_\C,L_\D,R_\D$ are the functors of \prox\ref{propfurthermonadicity}, for $\C$ and $\D$ respectively; 
        i.e, for every $f\: A\to B$ in $\C$ the following three diagrams commute:
\begin{cd}[5][5]
	F(A) \arrow[r,"{F(\ell_f)}"]\arrow[d,twoheadrightarrow,"{\ell_{F(f)}}"'] \& F(Q_\C(f)) \arrow[d,"{F(r_f)}"]\\
	Q_\D(F(f)) \arrow[ru,"{\rho_f}"{inner sep=0.3ex}]\arrow[r,hookrightarrow,"{r_{F(f)}}"'] \& F(B)
\end{cd}

\begin{eqD*}
\begin{cd}*[2.2][1.2]
	Q_\D(F(f)) \arrow[d,"{\delta_{F(f)}}"']\arrow[rr,"{\rho_f}"]\&\& F(Q_\C(f)) \arrow[d,"{F(\delta_f)}"] \\
	Q_\D(\ell_{F(f)}) \arrow[rd,"{Q_\D(\id{},\rho_f)}"'{inner sep=0.3ex},shorten <=-0.3ex, shorten >=-0.3ex]\&\& F(Q_\C(\ell_f)) \\[-1.7ex]
	\& Q_\D(F(\ell_f)) \arrow[ru,"{\rho_{\ell_f}}"'{inner sep=0.3ex},shorten <=-0.45ex, shorten >=-0.3ex]
\end{cd}\quad
\begin{cd}*[2.2][1.2]
	Q_\D(F(f)) \arrow[rr,"{\rho_f}"]\&\& F(Q_\C(f)) \\
	Q_\D(r_{F(f)}) \arrow[u,"{\mu_{F(f)}}"]\arrow[rd,"{Q_\D(\rho_f,\id{})}"'{inner sep=0.3ex},shorten <=-0.3ex, shorten >=-0.3ex]\&\& F(Q_\C(r_f)) \arrow[u,"{F(\mu_f)}"']\\[-1.7ex]
	\& Q_\D(F(r_f)) \arrow[ru,"{\rho_{r_f}}"'{inner sep=0.3ex},shorten <=-0.45ex, shorten >=-0.3ex]
\end{cd}
\end{eqD*}
where $\delta_h$ and $\mu_h$ are given by
\begin{eqD*}
\begin{cd}*[5][5]
	A \arrow[d,twoheadrightarrow,"{\ell_h}"']\arrow[r,twoheadrightarrow,"{\ell_{\ell_h}}"] \& Q(\ell_h) \arrow[d,hookrightarrow,"{r_{\ell_h}}"] \\
	Q(h) \arrow[ru,dashed,"{\exists ! \hspace{0.1ex} \delta_h}"{inner sep=0.3ex}] \arrow[r,equal,"{}"] \& Q(h)
\end{cd}
\qquad\quad
\begin{cd}*[5][5]
	Q(h)\arrow[d,twoheadrightarrow,"{\ell_{r_f}}"'] \arrow[r,equal,"{}"] \& Q(h) \arrow[d,hookrightarrow,"{r_h}"]\\
	Q(r_h) \arrow[ru,dashed,"{\exists ! \hspace{0.1ex} \mu_h}"{inner sep=0.3ex}] \arrow[r,hookrightarrow,"{r_{r_h}}"'] \& B
\end{cd}
\end{eqD*}
    \end{enumT}
Moreover, the 2-cells between (lax) morphisms of normal pseudo-algebras for $\sqm$ are precisely the natural transformations between the underlying functors.
\end{proposition}

\begin{proof}
   We prove $(i)\aR{}(ii)$. Given $f\:A\to B$ in $\C$, we can produce the component $\rho_f$ of $\rho$ on $f$ as in \remx\ref{rempreservefactorizationslaxly}. Then $\rho$ is natural by the uniqueness of the liftings required by definition of orthogonal factorization system. And clearly $\rho_{\id{}}=\id{}$, because the choice of factorization for the identity is always by identities. It is straightforward to show that 
   the pentagon commutes, using again the uniqueness of the liftings.

   We now prove $(ii)\aR{}(iii)$. We can take the needed $\rho$ to be the natural transformation $\rho$ given by $(ii)$. By \remx\ref{remdetailsmonadicity}, naturality of $\rho$ and $\rho_{\id{}}=\id{}$, we obtain that the two triangles in the first diagram in $(iii)$ commute. Indeed, considering the morphism in $\C^{\to}$ on the left below, naturality of $\rho$ gives the square on the right below
   \begin{eqD*}
       \begin{cd}*[5][5]
       A \arrow[d,equal]\arrow[r,equal] \& A \arrow[d,"f"] \\
       A \arrow[r,"f"'] \& B
   \end{cd}\qquad
   \begin{cd}*[4.5][4.5]
	{Q_\D(F(\id{A}))} \arrow[d,"{\ell_{F(f)}}"']\arrow[r,equal, "{\rho_{\id{A}}}"] \& F(Q_\C(\id{A}))\arrow[d,"{F(\ell_f)}"] \\
	Q_\D(F(f)) \arrow[r, "{\rho_{f}}"'] \& F(Q_\C(f))
\end{cd}
   \end{eqD*}
   and considering the morphism $(f,\id{})$ we obtain the other triangle. Evaluating the pentagon of $(ii)$ for the morphism $(\id{},f)$ (respectively, $(f,\id{})$) in $\C^{\to}$, we obtain the pentagon on the left (respectively, on the right) of $(iii)$. It is indeed straightforward to show that the isomorphisms of the pentagon of $(ii)$ become the appropriate morphisms $\delta$ and $\mu$.

   Finally, we prove $(iii)\aR{}(i)$. The lax morphism of monads $(\rho,1)$ induces a functor $R_\C\mbox{-}\catfont{Alg}\to R_\D\mbox{-}\catfont{Alg}$. 
   So by \prox\ref{propfurthermonadicity}, $F$ preserves the right class. If $\rho$ is a natural isomorphism, the two triangles in the first diagram in $(iii)$ show that $F$ preserves both the left and the right classes. If $\rho$ is the identity, the same triangles show that $F$ preserves the factorizations.

It is then straightforward to show that the 2-cells between (lax) morphisms of normal pseudo-algebras for $\sqm$ are just the natural transformations, by uniqueness of liftings.
\end{proof}

\begin{corollary}
    $\Fact$ and $\Factlax$ are precisely the 2-categories $\Alg+[p]{\sqm}$ (respectively, $\Alg+[l]{\sqm}$) of normal pseudo-algebras for $\sqm$ on $\Cat$ and pseudo morphisms (respectively, lax morphisms) between them.
\end{corollary}

\begin{proof}
    This follows immediately from \prox\ref{propmorphismsofalgebrassqmonad}.
\end{proof}

The following result will be needed to define double orthogonal factorization systems.

\begin{proposition}\label{propfacthaspullbacks}
    The 2-categories $\Fact, \Factlax$, and $\Factcolax$ have all pullbacks of strict morphisms.
\end{proposition}
\begin{proof}
    Let $\catfont{A,B,C}$ be categories equipped with orthogonal factorization systems, and strict morphisms $F\:\catfont{A} \to \catfont{C}$ and $G:\catfont{B} \to \catfont{C}$.  We show that the pullback category $\catfont{A} \times_{\catfont{C}}\catfont{B}$ 
    admits an orthogonal factorization system, and that the associated projections and the unique arrows determined by its universal property are factorization-preserving.  Additionally, when working in $\Factlax$ or $\Factps$, we show that the induced pullback maps are respectively strict or lax.
    Denote the factorization systems on $\catfont{A}$ and $\catfont{B}$ by $(\catfont{L}_1, \catfont{R}_1)$ and $(\catfont{L_2}, \catfont{R}_2)$ respectively.  We form a factorization system  on $\catfont{A} \times_{\catfont{C}} \catfont{B}$ whose classes are given by:
    \[(\catfont{L}^\star, \catfont{R}^\star) = (\catfont{L}_1 \times \catfont{L}_2, \catfont{R}_1 \times \catfont{R}_2)\]

    For a  morphism $(f,g)\: (A,B) \to (A', B')$ in $\catfont{A} \times_\catfont{C} \catfont{B}$ we choose the factorization:
    \[\begin{tikzcd}
	{(A, B)} && {(A', B')} \\
	& {(\operatorname{Im}(f), \operatorname{Im}(g))}
	\arrow["{(f,g)}", from=1-1, to=1-3]
	\arrow["{(\ell_f, \ell_g)}"', from=1-1, to=2-2]
	\arrow["{(r_f, r_g)}"', from=2-2, to=1-3]
\end{tikzcd} \label{pbfact}\]
where $f=r_f\circ\ell_f$ and $g=r_g\circ\ell_g$ are the factorizations of $f$, $g$ in $\catfont{A}$, and  $\catfont{B}$ respectively. 
Note that this is well-defined because $F$ and $G$ are strict morphisms between categories with an OFS.
Because $\catfont{L}_1$ and $\catfont{L}_2$ are closed under composition, so is $\catfont{L^\star}$.
Likewise, by the fact that $\catfont{L}_1$ and $\catfont{L}_2$ contain all isomorphisms, $\catfont{L}^\star$ does as well.  An identical argument holds for $\catfont{R}^\star$, therefore, $(\catfont{L}^\star, \catfont{R}^\star)$ is a factorization system.  We now show orthogonality.  Given a diagram in $\catfont{A} \times _\catfont{C} \catfont{B}$ of the form:

\[\begin{tikzcd}
	{(A, B)} & {(A', B')} \\
	{(C, D)} & {(C', D')}
	\arrow["{(f, f')}", from=1-1, to=1-2]
	\arrow["{(u, u')}"', from=1-1, to=2-1]
	\arrow["{(v, v')}", from=1-2, to=2-2]
	\arrow["{(g,g')}"', from=2-1, to=2-2]
\end{tikzcd}\]

 with $(u, u') \in \catfont{L}^\star$ and $(v, v') \in \catfont{R}^\star$, consider the following square in $\catfont{A}$:
\[\begin{tikzcd}
	A & {A'} \\
	C & {C'}
	\arrow["f", from=1-1, to=1-2]
	\arrow["u"', from=1-1, to=2-1]
	\arrow["v", from=1-2, to=2-2]
	\arrow["g"', from=2-1, to=2-2]
\end{tikzcd}\]

We have $u \in \catfont{L}_1, v \in 
\catfont{R}_1$, so by orthogonality there exists a unique morphism $e\:C \to A'$ making the following commute:

\[\begin{tikzcd}
	A & {A'} \\
	C & {C'}
	\arrow["f", from=1-1, to=1-2]
	\arrow["u"', from=1-1, to=2-1]
	\arrow["v", from=1-2, to=2-2]
	\arrow["e"', dashed, from=2-1, to=1-2]
	\arrow["g"', from=2-1, to=2-2]
\end{tikzcd}\]

    Applying an identical argument to the square likewise constructed in $\catfont{B}$ yields $e'\:D \to B'$ satisfying a similar condition.  
    Furthermore, $F(e)$ and $G(e')$ are both the unique diagonal arrow in the following diagram in $\C$,
\[\begin{tikzcd}[ampersand replacement=\&]
	FA \& {FA'} \&\& GB \& {GB'} \\
	FC \& {FC'} \&\& GD \& {GD'}
        \arrow[dashed, from=2-1, to=1-2]
        \arrow[dashed, from=2-4, to=1-5]
	\arrow["Ff", from=1-1, to=1-2]
	\arrow["Fu"', from=1-1, to=2-1]
	\arrow[""{name=0, anchor=center, inner sep=0}, "Fv", from=1-2, to=2-2]
	\arrow["{Gf'}", from=1-4, to=1-5]
	\arrow[""{name=1, anchor=center, inner sep=0}, "{Gu'}"', from=1-4, to=2-4]
	\arrow["{Gv'}", from=1-5, to=2-5]
	\arrow["Fg"', from=2-1, to=2-2]
	\arrow["{Gg'}"', from=2-4, to=2-5]
	\arrow["{=}"{description}, draw=none, from=0, to=1]
\end{tikzcd}\]
    
    Hence, $F(e)=G(e')$  and $(e,e')$ is a well-defined arrow in the pullback category.
    Now, $(e, e')$ is the unique arrow making the following diagram commute:

\[\begin{tikzcd}
	{(A, B)} & {(A', B')} \\
	{(C, D)} & {(C', D')}
	\arrow["{{(f, f')}}", from=1-1, to=1-2]
	\arrow["{{(u, u')}}"', from=1-1, to=2-1]
	\arrow["{{(v, v')}}", from=1-2, to=2-2]
	\arrow["{(e, e')}", dashed, from=2-1, to=1-2]
	\arrow["{{(g,g')}}"', from=2-1, to=2-2]
\end{tikzcd}\]
This yields orthogonality of the factorization system $(\catfont{L}^\star, \catfont{R}^\star)$ on $\catfont{A} \times_\catfont{C} \catfont{B}$.  With this established, we now show that the induced pullback maps satisfy the relevant properties for $\Fact, \Factlax, \Factcolax$.

Consider the projection map $P_1\: \catfont{A} \times_\catfont{C} \catfont{B} \to \catfont{A}$.  For $(\ell_f, \ell_g) \in \catfont{L}^\star$,  we have $P_1(\ell_f, \ell_g) = \ell_f$, hence $P_1$ preserves the left class, and likewise for the right class.  Therefore, $P_1$ is a morphism of categories equipped with a factorization system, as in Definition \ref{defmorfact}.  An analogous argument applies to $P_2$.

Suppose that $F'\: \catfont{D} \to \catfont{A}$, and $G'\: \catfont{D} \to \catfont{B}$ are lax morphisms such that $F\circ F'=G\circ G'$.  By the universal property of pullbacks, there exists a unique $I\: \catfont{D} \to \catfont{A} \times_\catfont{C} \catfont{B}$ such that $F' = P_1\circ I$ and $G' = P_2\circ I$.  Explicitly, $I$ maps an object $D \in \catfont{D}$ to $(F'(D), G'(D))$, and a morphism $g\: D \to D'$ to $(F'(f), G'(f))$.  

If $F'$ and $G'$ both preserve right classes, and $u \in \catfont{R}$, we have $F'(u) \in \catfont{R}_1, G'(u) \in \catfont{R}_2$, hence $(F'(u), G'(u)) \in \catfont{R}^\star$.  Thus,  $I$ is a lax morphism of categories with an OFS. If $F'$ and $G'$ both preserve the left class as well, an identical argument shows that $I$ preserves the left class, so $I$ is a morphism of categories with an OFS.
Hence, the pullback morphisms will be strict, lax, or colax as required.

When $F'$ and $G'$ are strict, let $u$ be any morphism $u$ in $\catfont{D}$, with chosen factorization $u=r_u\circ\ell_u$. Since $P_1$ and $P_2$ are strict morphisms of categories with an OFS, $I(u)$ in $\catfont{A} \times_\catfont{C} \catfont{B}$ factors as $(F'(u), G'(u)) = (F'(r_u), G'(r_u)) \circ (F'(\ell_u), G'(\ell_u))$; therefore, $I$ preserves factorizations strictly.

Because $\catfont{Cat}$ admits 2-pullbacks, and 2-cells in $\Fact$ and $\Factlax$ are ordinary
natural transformations, the relevant 2-dimensional universal properties are satisfied.

We have shown that $\catfont{A} \times_\catfont{C} \catfont{B}$ admits an orthogonal factorization system.  Furthermore, all morphisms induced by the pullback are appropriately strict, lax, or colax. Therefore, pullbacks of strict morphisms exist in $\Fact, \Factlax$, and $\Factcolax$.
\end{proof}

\section{Double factorization systems}\label{sectionDOFS}

In this section, we propose our notion of double orthogonal factorization system. We first present some motivating examples that led us to 
this notion. We observe that some examples are more lax than others, and then focus primarily on two cases of double orthogonal factorization systems: the lax and the pseudo case. As an extension to what has been discussed in the previous section, one may also develop strict and colax versions of our theory.

We define double orthogonal factorization systems with an approach of internal category theory. We then explicitly describe them in terms of compatible orthogonal factorization systems on the arrows and the double cells of a double category. In Section~\ref{sectionmonadicity}, we will prove monadicity results for double orthogonal factorization systems, justifying our notions.

\subsection{Motivating Examples}
\begin{example}[The quartet double category on $\C$, $\mathbb{S}\mathrm{q} (\mathcal{C})$]\label{ex_quartet}

For any category $\C$, the double category $\mathbb{S}\mathrm{q} (\mathcal{C})$, introduced by Ehresmann in his original work on double categories \citep{Ehres}, has the same objects as $\C$, both the arrows and the proarrows of $\mathbb{S}\mathrm{q} (\mathcal{C})$ are the arrows of $\C$, and the double cells are the commutative squares in $\C$.
When $\C$ is a category with an OFS $(\mathcal{L},\mathcal{R})$, we can lift the factorization structure on the arrows of $\mathbb{S}\mathrm{q} (\mathcal{C})$, here drawn horizontally, to a factorization structure on the double cells of $\mathbb{S}\mathrm{q} (\mathcal{C})$ as follows.

     Given any cell  $\alpha$ of $\mathbb{S}\mathrm{q} (\mathcal{C})$ (which is completely determined by its boundary arrows), we factor the (horizontal) arrows in its boundary as $f=r_f\circ\ell_f$ and $g=r_g\circ\ell_g$ according to the factorization system on $\C$. By the functoriality of the factorization system there is a unique arrow $q$ as in the following diagram.
     
\begin{equation}\label{chosenfact-squares}
    \begin{tikzcd}
	A & B && A & {\text{Im}(f)} & B \\
	C & D && C & {\text{Im}(g)} & D
	\arrow["f", from=1-1, to=1-2]
	\arrow["u"', from=1-1, to=2-1]
	\arrow[""{name=0, anchor=center, inner sep=0}, "v", from=1-2, to=2-2]
	\arrow["\ell_f", two heads, from=1-4, to=1-5]
	\arrow[""{name=1, anchor=center, inner sep=0}, "u"', from=1-4, to=2-4]
	\arrow["r_f",hook, from=1-5, to=1-6]
	\arrow["q", dashed, from=1-5, to=2-5]
	\arrow["v", from=1-6, to=2-6]
	\arrow["g"', from=2-1, to=2-2]
	\arrow["\ell_g"', two heads, from=2-4, to=2-5]
	\arrow["r_g"', hook, from=2-5, to=2-6]
	\arrow["{{=}}"{description}, draw=none, from=0, to=1]
\end{tikzcd}\end{equation}

This shows that $\alpha$ factors into a horizontal (arrow-direction) composition of a cell $l_{\alpha}$ with (horizontal) arrows in $\mathcal{L}$, followed by a cell $r_{\alpha}$ with (horizontal) arrows from $\mathcal{R}$.
Hence, $\mathbb{S}\mathrm{q} (\mathcal{C})$ has a factorization system $(\mathcal{L}', \mathcal{R}')$ on its double cells where the left class $\mathcal{L}'$ consists of cells with morphisms from $\mathcal{L}$ as arrows and the right class $\mathcal{R}'$ consists of cells with morphisms from $\mathcal{R}$ as arrows. Both classes of cells are closed under horizontal composition.
It is not hard to show that the factorization system on the double cells is orthogonal as factorization system on the category with vertical arrows as objects and double cells as arrows (using horizontal composition). 

Thus, we have an OFS $(\mathcal{L},\mathcal{R})$ on the category of arrows $\mathbb{S}\mathrm{q} (\mathcal{C})_0$ and an OFS $(\mathcal{L}', \mathcal{R}')$ on the category of double cells $\mathbb{S}\mathrm{q} (\mathcal{C})_1$. We will now investigate how these factorization systems interact with the structure of the double category.
First note that the source, target and unit morphisms preserve the factorizations, so they are strict morphisms of categories with an OFS.
%We will now investigate how proarrow (vertical) composition interacts with the factorizations.

The vertical composition of two cells factors as follows:
\[\begin{tikzcd}
	A & B & A & {\text{Im}(f)} & B \\
	C & D & C && D \\
	E & F & E & {\text{Im}(g)} & F
	\arrow["f", from=1-1, to=1-2]
	\arrow[""{name=0, anchor=center, inner sep=0}, from=1-1, to=2-1]
	\arrow[""{name=1, anchor=center, inner sep=0}, from=1-2, to=2-2]
	\arrow[two heads, from=1-3, to=1-4]
	\arrow[from=1-3, to=2-3]
	\arrow[hook, from=1-4, to=1-5]
	\arrow[from=1-4, to=3-4]
	\arrow[from=1-5, to=2-5]
	\arrow["g", from=2-1, to=2-2]
	\arrow[""{name=2, anchor=center, inner sep=0}, from=2-1, to=3-1]
	\arrow["{=}"{description}, draw=none, from=2-2, to=2-3]
	\arrow[""{name=3, anchor=center, inner sep=0}, from=2-2, to=3-2]
	\arrow[from=2-3, to=3-3]
	\arrow[from=2-5, to=3-5]
	\arrow["h"', from=3-1, to=3-2]
	\arrow[two heads, from=3-3, to=3-4]
	\arrow[hook, from=3-4, to=3-5]
	\arrow["\alpha"{description}, draw=none, from=0, to=1]
	\arrow["\beta"{description}, draw=none, from=2, to=3]
        \arrow["\ell_{\beta\otimes\alpha}"{description}, draw=none, from=1-3, to=3-4]
        \arrow["r_{\beta\otimes\alpha}"{description}, draw=none, from=1-4, to=3-5]
\end{tikzcd}\]
where the center arrow exists by functoriality of the OFS on $\catfont{C}$.
On the other hand, the composition of the factorizations of both cells is given by:
\[\begin{tikzcd}
	A & B & A & {\text{Im}(f)} & B \\
	C & D & C & {\text{Im}(g)} & D \\
	E & F & E & {\text{Im}(h)} & F
	\arrow["f", from=1-1, to=1-2]
	\arrow[from=1-1, to=2-1]
	\arrow[from=1-2, to=2-2]
	\arrow[two heads, from=1-3, to=1-4]
	\arrow[""{name=0, anchor=center, inner sep=0}, from=1-3, to=2-3]
	\arrow[hook, from=1-4, to=1-5]
	\arrow[""{name=1, anchor=center, inner sep=0}, from=1-4, to=2-4]
	\arrow[""{name=2, anchor=center, inner sep=0}, from=1-5, to=2-5]
	\arrow["g", from=2-1, to=2-2]
	\arrow[from=2-1, to=3-1]
	\arrow["{{=}}"{description}, draw=none, from=2-2, to=2-3]
	\arrow[from=2-2, to=3-2]
	\arrow[two heads, from=2-3, to=2-4]
	\arrow[""{name=3, anchor=center, inner sep=0}, from=2-3, to=3-3]
	\arrow[hook, from=2-4, to=2-5]
	\arrow[""{name=4, anchor=center, inner sep=0}, from=2-4, to=3-4]
	\arrow[""{name=5, anchor=center, inner sep=0}, from=2-5, to=3-5]
	\arrow["h"', from=3-1, to=3-2]
	\arrow[two heads, from=3-3, to=3-4]
	\arrow[hook, from=3-4, to=3-5]
	\arrow["{{l_{\alpha}}}"{description}, draw=none, from=0, to=1]
	\arrow["{{r_{\alpha}}}"{description}, draw=none, from=1, to=2]
	\arrow["{{l_{\beta}}}"{description}, draw=none, from=3, to=4]
	\arrow["{{r_{\beta}}}"{description}, draw=none, from=4, to=5]
        \arrow["\alpha"{description},draw=none, from=1-1, to=2-2]
        \arrow["\beta"{description},draw=none, from=2-1, to=3-2]
\end{tikzcd}\]

The equalities $l_{\beta \otimes \alpha} = l_{\beta} \otimes e_{\alpha}$ and $r_{\beta \otimes \alpha} = r_{\beta} \otimes r_{\alpha}$ hold by uniqueness of the middle arrow, which is a consequence of  functoriality of the OFS. It follows that the composition of the codomains of $l_{\beta}$ and $l_{\alpha}$ is equal to the the codomain of the composition of $l_{\beta \otimes \alpha}$.

Hence, the double category $\mathbb{S}\mathrm{q} (\mathcal{C})$ has a factorization system $(\mathcal{L}', \mathcal{R}')$ on its double cells, such that $\mathcal{L}'$ and $\mathcal{R}'$ are closed under vertical composition and furthermore, with the chosen factorization for cells indicated in diagram (\ref{chosenfact-squares}), all structure arrows are strict morphisms between categories with an OFS.
\end{example}

\begin{rmk}
In our first example, all structure maps preserve the factorization systems strictly. The next  example presents a factorization system on $\mathbb{S}\mathrm{pan}(\mathcal{C})$ where vertical composition only preserves the right class. However, source, target and unit arrows are strict.
\end{rmk}

\begin{example}[The span double category on $\C$, $\mathbb{S}\mathrm{pan}(\mathcal{C})$]\label{examplespan} 

        Let $\mathcal{C}$ be a category with all pullbacks and an OFS $(\mathcal{L}, \mathcal{R})$ with chosen factorizations,
 
\[\begin{tikzcd}[ampersand replacement=\&]
	{(A} \& {B)} \& {(A} \& {\im(f)} \& {B)}
	\arrow["f", from=1-1, to=1-2]
	\arrow["{=}"{description}, draw=none, from=1-2, to=1-3]
	\arrow["{\ell_f}", two heads, from=1-3, to=1-4]
	\arrow["{r_f}", hook, from=1-4, to=1-5]
\end{tikzcd}\]
        The double category $\mathbb{S}\mathrm{pan}(\mathcal{C})$ has morphisms in $\mathcal{C}$ as arrows and spans of $\mathcal{C}$ as proarrows (proarrow composition is through chosen pullbacks); further, a double cell is defined by an arrow from the apex of its domain span to the apex of its codomain making both squares commute as in the following diagram.
\[\begin{tikzcd}\label{span-dbl-cell}
	& A && C \\
	X && Y \\
	& B && D
	\arrow["f", from=1-2, to=1-4]
	\arrow[from=2-1, to=1-2]
	\arrow["\alpha", from=2-1, to=2-3]
	\arrow[from=2-1, to=3-2]
	\arrow[from=2-3, to=1-4]
	\arrow[from=2-3, to=3-4]
	\arrow["g", from=3-2, to=3-4]
\end{tikzcd}\]

 Since $\mathcal{C}$ has an OFS, any cell in $\mathbb{S}\mathrm{pan}(\mathcal{C})$ can be factorized in the following manner. 

\[\begin{tikzcd}
	& A && C && A & {\text{Im}(f)} & C \\
	X && Y && X & {\text{Im}(\alpha)} & Y \\
	& B && D && B & {\text{Im}(g)} & D
	\arrow["f", from=1-2, to=1-4]
	\arrow[tail reversed, no head, from=1-4, to=2-3]
	\arrow["{\ell_f}", two heads, from=1-6, to=1-7]
	\arrow["{r_f}", hook, from=1-7, to=1-8]
	\arrow[from=2-1, to=1-2]
	\arrow["\alpha", from=2-1, to=2-3]
	\arrow[from=2-1, to=3-2]
	\arrow["{{{{=}}}}"{description, pos=0.6}, draw=none, from=2-3, to=2-5]
	\arrow[from=2-3, to=3-4]
	\arrow[from=2-5, to=1-6]
	\arrow["{\ell_\alpha}", two heads, from=2-5, to=2-6]
	\arrow[from=2-5, to=3-6]
	\arrow[from=2-6, to=1-7]
	\arrow["{r_\alpha}", hook, from=2-6, to=2-7]
	\arrow[from=2-6, to=3-7]
	\arrow[from=2-7, to=1-8]
	\arrow[from=2-7, to=3-8]
	\arrow["g"', from=3-2, to=3-4]
	\arrow["{\ell_g}"', two heads, from=3-6, to=3-7]
	\arrow["{r_g}"', hook, from=3-7, to=3-8]
\end{tikzcd}\]
giving rise to the factorization system on $\mathbb{S}\mathrm{pan}(\mathcal{C})_1$ where cells are in the left class $\mathcal{L}'$ when all three horizontal arrows $f$, $g$, and $\alpha$ are in $\mathcal{L}$, and cells are in the right class $\mathcal{R}'$ if all three of these arrows are in $\mathcal{R}$.
It is clear that both classes of this system are closed under horizontal composition. It is also clear that the source, target and units functors preserve the chosen factorizations. 

We now show that the right class $\mathcal{R}'$ is  closed under vertical composition.  Let $\theta$ and $\psi$ be vertically composable double cells in the right class. Composing them yields the following.

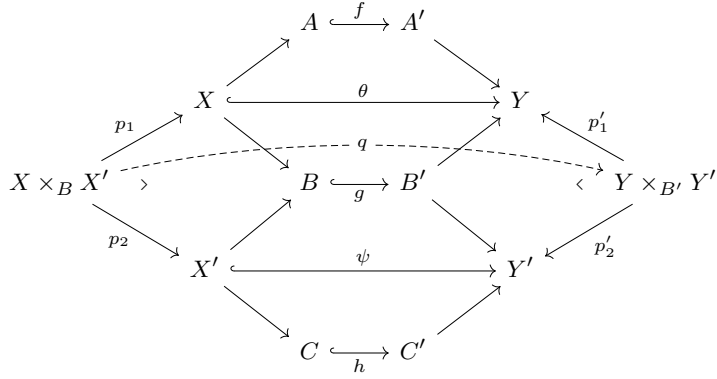
\[\begin{tikzcd}
	&& A & {A'} \\
	& X &&& Y \\
	{X \times_BX'} && B & {B'} && {Y \times_{B'} Y'} \\
	& {X'} &&& {Y'} \\
	&& C & {C'}
	\arrow["f", hook, from=1-3, to=1-4]
	\arrow[from=1-4, to=2-5]
	\arrow[from=2-2, to=1-3]
	\arrow["\theta", hook, from=2-2, to=2-5]
	\arrow[from=2-2, to=3-3]
	\arrow["{p_1}", from=3-1, to=2-2]
	\arrow["\lrcorner"{anchor=center, pos=0.125, rotate=45}, draw=none, from=3-1, to=3-3]
	\arrow["q"{description}, curve={height=-18pt}, dashed, from=3-1, to=3-6]
	\arrow["{p_2}"', from=3-1, to=4-2]
	\arrow["g"', hook, from=3-3, to=3-4]
	\arrow[from=3-4, to=2-5]
	\arrow[from=3-4, to=4-5]
	\arrow["{p'_1}"', from=3-6, to=2-5]
	\arrow["\lrcorner"{anchor=center, pos=0.125, rotate=-135}, draw=none, from=3-6, to=3-4]
	\arrow["{p'_2}", from=3-6, to=4-5]
	\arrow[from=4-2, to=3-3]
	\arrow["\psi", hook, from=4-2, to=4-5]
	\arrow[from=4-2, to=5-3]
	\arrow["h"', hook, from=5-3, to=5-4]
	\arrow[from=5-4, to=4-5]
\end{tikzcd}\]

We need to show that the center arrow $q\:X\times_BX'\to Y\times_{B'}Y'$ induced by pullback is also in $\mathcal{R}$.  To this end, we will show that it satisfies the left lifting property with respect to any morphism in $\mathcal{L}$.  Suppose we have any morphism  $e\: R \twoheadrightarrow S$ in $\mathcal{L}$ which is part of a commutative square on the left, as 
in the following diagram:

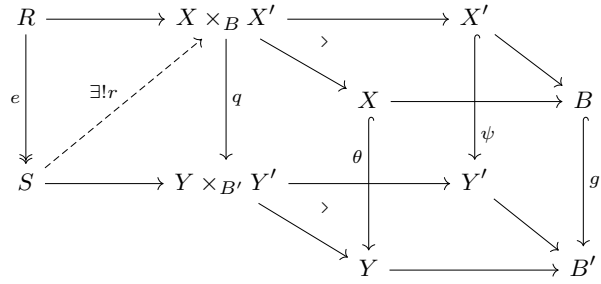
\[\begin{tikzcd}
	R && {X \times_BX'} && {X'} \\
	&&& X && B \\
	S && {Y \times_{B'} Y'} && {Y'} \\
	&&& Y && {B'}
	\arrow[from=1-1, to=1-3]
	\arrow["e"', two heads, from=1-1, to=3-1]
	\arrow[from=1-3, to=1-5]
	\arrow[from=1-3, to=2-4]
	\arrow["\lrcorner"{anchor=center, pos=0.125, rotate=45}, draw=none, from=1-3, to=2-6]
	\arrow["q", from=1-3, to=3-3]
	\arrow[from=1-5, to=2-6]
	\arrow["\psi"{pos=0.8}, hook, from=1-5, to=3-5]
	\arrow[from=2-4, to=2-6]
	\arrow["\theta"'{pos=0.3}, hook, from=2-4, to=4-4]
	\arrow["g", hook, from=2-6, to=4-6]
	\arrow["{\exists! r}", dashed, from=3-1, to=1-3]
	\arrow[from=3-1, to=3-3]
	\arrow[from=3-3, to=3-5]
	\arrow[from=3-3, to=4-4]
	\arrow["\lrcorner"{anchor=center, pos=0.125, rotate=45}, draw=none, from=3-3, to=4-6]
	\arrow[from=3-5, to=4-6]
	\arrow[from=4-4, to=4-6]
\end{tikzcd}\]

We wish to show that there exists a unique $r\: S \to X \times_B X'$ giving us two commutative triangles to make up the square.  By lifting $e$ against $\theta, g$, and $\psi$ (with squares formed by composing upper segments of the diagram appropriately), we obtain unique morphisms $f_1, f_2$, and $f_3$ as in the following diagram,

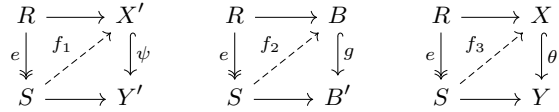
\if{false}
\[\begin{tikzcd}
	R && {X \times_BX'} && {X'} \\
	&&& X && B \\
	S && {Y \times_{B'} Y'} && {Y'} \\
	&&& Y && {B'}
	\arrow[from=1-1, to=1-3]
	\arrow["e"', two heads, from=1-1, to=3-1]
	\arrow[from=1-3, to=1-5]
	\arrow["\lrcorner"{anchor=center, pos=0.125, rotate=45}, draw=none, from=1-3, to=2-6]
	\arrow[from=1-5, to=2-6]
	\arrow["\psi"{pos=0.8}, hook, from=1-5, to=3-5]
	\arrow[from=2-4, to=2-6]
	\arrow["\theta"'{pos=0.3}, hook, from=2-4, to=4-4]
	\arrow["g", hook, from=2-6, to=4-6]
	\arrow["{f_2}", curve={height=-18pt}, dashed, from=3-1, to=1-5]
	\arrow["{f_4}"', curve={height=-6pt}, dashed, from=3-1, to=2-4]
	\arrow["{f_3}"', curve={height=-30pt}, dashed, from=3-1, to=2-6]
	\arrow[from=3-1, to=3-3]
	\arrow[from=3-3, to=3-5]
	\arrow[from=3-3, to=4-4]
	\arrow["\lrcorner"{anchor=center, pos=0.125, rotate=45}, draw=none, from=3-3, to=4-6]
	\arrow[from=3-5, to=4-6]
	\arrow[from=4-4, to=4-6]
\end{tikzcd}\]
\fi

\[\begin{tikzcd}
	R & {X'} & R & B & R & X \\
	S & {Y'} & S & {B'} & S & Y
	\arrow[from=1-1, to=1-2]
	\arrow["e"', two heads, from=1-1, to=2-1]
	\arrow["\psi", hook, from=1-2, to=2-2]
	\arrow[from=1-3, to=1-4]
	\arrow["e"', two heads, from=1-3, to=2-3]
	\arrow["g", hook, from=1-4, to=2-4]
	\arrow[from=1-5, to=1-6]
	\arrow["e"', two heads, from=1-5, to=2-5]
	\arrow["\theta", hook, from=1-6, to=2-6]
	\arrow["{f_1}", dashed, from=2-1, to=1-2]
	\arrow[from=2-1, to=2-2]
	\arrow["{f_2}", dashed, from=2-3, to=1-4]
	\arrow[from=2-3, to=2-4]
	\arrow["{f_3}", dashed, from=2-5, to=1-6]
	\arrow[from=2-5, to=2-6]
\end{tikzcd}\]
Uniqueness of each lifting, along with the pullback's universal property, guarantees a unique $r$ filling the left square as desired. 

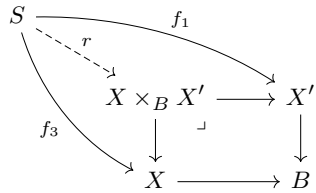
\[\begin{tikzcd}
	S \\
	& {X \times_{B}X'} & {X'} \\
	& X & B
	\arrow["r", dashed, from=1-1, to=2-2]
	\arrow["{f_1}", curve={height=-12pt}, from=1-1, to=2-3]
	\arrow["{f_3}"', curve={height=12pt}, from=1-1, to=3-2]
	\arrow[from=2-2, to=2-3]
	\arrow[from=2-2, to=3-2]
	\arrow["\lrcorner"{anchor=center, pos=0.125}, draw=none, from=2-2, to=3-3]
	\arrow[from=2-3, to=3-3]
	\arrow[from=3-2, to=3-3]
\end{tikzcd}\] Thus, $q$ is right-orthogonal to any map in the left class $\mathcal{L}$, so $q$ is in the right class $\mathcal{R}$.

We show in Example \ref{counter-ex} that vertical composition does not preserve the left class of cells, $\mathcal{L}'$ in this case.
\end{example}

\begin{example}\label{ex:rel}[The double category $\mathbb{R}\mathrm{el}(\mathcal{C})$ of relations on any regular category $\C$] \label{examplerel} 
     
     Let $\catfont{C}$ be a regular category. Then the canonical factorization system on $\C$ has regular epimorphisms as its left class and monomorphisms as its right class.
     We assume again that a factorization $f=r_f\circ \ell_f$ has been chosen for each arrow $f$ in $\C$.
     The double category of relations on $\C$ has morphisms of $\C$ as arrows and relations in $\C$ as proarrows. 
     This means that a proarrow $R$ from $X$ to $Y$ is given by a monic arrow $R\hookrightarrow X\times Y$ in $\C$.
     A double cell 
\[\begin{tikzcd}
	X && {X'} \\
	\\
	Y && {Y'}
	\arrow["{e_0}", from=1-1, to=1-3]
	\arrow[""{name=0, anchor=center, inner sep=0}, "R"', "\shortmid"{marking}, from=1-1, to=3-1]
	\arrow[""{name=1, anchor=center, inner sep=0}, "{R'}", "\shortmid"{marking}, from=1-3, to=3-3]
	\arrow["{e_1}"', from=3-1, to=3-3]
	\arrow["\theta"{description}, draw=none, from=0, to=1]
\end{tikzcd}\]
in $\mathbb{R}\mathrm{el}(\mathcal{C})$ is given by an arrow $\theta\: R\to R'$ in $\C$ such that the following diagram commutes in $\C$:
\[\begin{tikzcd}
	R && {X\times Y} \\
	\\
	{R'} && {X'\times Y'}
	\arrow[hook, from=1-1, to=1-3]
	\arrow["\theta"', from=1-1, to=3-1]
	\arrow["{e_0\times e_1}", from=1-3, to=3-3]
	\arrow[hook, from=3-1, to=3-3]
\end{tikzcd}\]
     We introduce a factorization system for cells in $\mathbb{R}\mathrm{el}(\mathcal{C})$ as follows:
\begin{itemize}
    \item the left class consists of cells with regular epimorphisms as arrow domain and codomain and a regular epimorphism for the arrow $\theta$ between the relations.
    
\[\begin{tikzcd}
	X && {X'} & &&R && {X\times Y} \\
	\\
	Y && {Y'} &&& {R'} && {X'\times Y'}
	\arrow[""{name=0, anchor=center, inner sep=0}, "{e_0}", two heads, from=1-1, to=1-3]
	\arrow["R"', "\shortmid"{marking}, from=1-1, to=3-1]
	\arrow[""{name=1, anchor=center, inner sep=0}, "{R'}", "\shortmid"{marking}, from=1-3, to=3-3]
	\arrow[""{name=2, anchor=center, inner sep=0}, hook, from=1-6, to=1-8]
	\arrow[""{name=3, anchor=center, inner sep=0}, "\theta"', two heads, from=1-6, to=3-6]
	\arrow["{e_0\times e_1}", two heads, from=1-8, to=3-8]
	\arrow[""{name=4, anchor=center, inner sep=0}, "{e_1}"', two heads, from=3-1, to=3-3]
	\arrow[""{name=5, anchor=center, inner sep=0}, hook, from=3-6, to=3-8]
	\arrow["\theta"{description}, draw=none, from=0, to=4]
	\arrow["{\mbox{\scriptsize corresponds to}}"{description}, draw=none, from=1, to=3]
	\arrow[""{description}, draw=none, from=2, to=5]
\end{tikzcd}\]

\item the right class consists of cells with monomorphisms as arrow domain and codomain and a monomorphisms for the arrow $\theta$ between the relations.

\[\begin{tikzcd}
	X && {X'} &&& R && {X\times Y} \\
	\\
	Y && {Y'} &&& {R'} && {X'\times Y'}
	\arrow[""{name=0, anchor=center, inner sep=0}, "{m_0}", hook, from=1-1, to=1-3]
	\arrow["R"', "\shortmid"{marking}, from=1-1, to=3-1]
	\arrow[""{name=1, anchor=center, inner sep=0}, "{R'}", "\shortmid"{marking}, from=1-3, to=3-3]
	\arrow[""{name=2, anchor=center, inner sep=0}, hook, from=1-6, to=1-8]
	\arrow[""{name=3, anchor=center, inner sep=0}, "\theta"', hook, from=1-6, to=3-6]
	\arrow["{m_0\times m_1}", hook, from=1-8, to=3-8]
	\arrow[""{name=4, anchor=center, inner sep=0}, "{m_1}"', hook, from=3-1, to=3-3]
	\arrow[""{name=5, anchor=center, inner sep=0}, hook, from=3-6, to=3-8]
	\arrow["\theta"{description}, draw=none, from=0, to=4]
	\arrow["{\mbox{\scriptsize corresponds to}}"{description}, draw=none, from=1, to=3]
	\arrow[""{description}, draw=none, from=2, to=5]
\end{tikzcd}\]
\end{itemize}

  We now verify that the right class of this factorization system is closed under vertical composition.  Given cells $\theta$ and $\alpha$ in $\mathbb{R}\mathrm{el}(\mathcal{C})$, recall that their vertical composition corresponds to an arrow $\theta \otimes \alpha$ in $\C$:
\[\begin{tikzcd}
	X & {X'} \\
	&&& X & {X'} &&& {R\otimes S} & {X\times Z} \\
	Y & {Y'} \\
	&&& Z & {Z'} &&& {R'\otimes S'} & {X'\times Z'} \\
	Z & {Z'}
	\arrow[""{name=0, anchor=center, inner sep=0}, "f", from=1-1, to=1-2]
	\arrow["R"', "\shortmid"{marking}, from=1-1, to=3-1]
	\arrow["{R'}", "\shortmid"{marking}, from=1-2, to=3-2]
	\arrow[""{name=1, anchor=center, inner sep=0}, "f", from=2-4, to=2-5]
	\arrow[""{name=2, anchor=center, inner sep=0}, "{R\otimes S}"', "\shortmid"{marking}, from=2-4, to=4-4]
	\arrow[""{name=3, anchor=center, inner sep=0}, "{R'\otimes S'}", "\shortmid"{marking}, from=2-5, to=4-5]
	\arrow[""{name=4, anchor=center, inner sep=0}, hook, from=2-8, to=2-9]
	\arrow[""{name=5, anchor=center, inner sep=0}, "{\theta\otimes\alpha}"', from=2-8, to=4-8]
	\arrow["{f\times h}", from=2-9, to=4-9]
	\arrow[""{name=6, anchor=center, inner sep=0}, "g"', from=3-1, to=3-2]
	\arrow[""{name=7, anchor=center, inner sep=0}, "S"', "\shortmid"{marking}, from=3-1, to=5-1]
	\arrow[""{name=8, anchor=center, inner sep=0}, "{S'}", "\shortmid"{marking}, from=3-2, to=5-2]
	\arrow[""{name=9, anchor=center, inner sep=0}, "h"', from=4-4, to=4-5]
	\arrow[""{name=10, anchor=center, inner sep=0}, hook, from=4-8, to=4-9]
	\arrow["h"', from=5-1, to=5-2]
	\arrow["\theta"{description}, draw=none, from=0, to=6]
	\arrow["{\theta\otimes\alpha}"{description}, draw=none, from=1, to=9]
	\arrow["{\mbox{\scriptsize corresponds to}}"{description}, draw=none, from=3, to=5]
	\arrow[""{description}, draw=none, from=4, to=10]
	\arrow["\alpha"{description}, draw=none, from=7, to=8]
	\arrow["{=}"{description}, draw=none, from=3-2, to=2]
\end{tikzcd}\]
\end{example}

A morphism $t\colon R \times_Y S \to X \times Z$ is obtained by the universal property of $X \times Z$ via the diagram:

\[\begin{tikzcd}
	& R & {X \times Y} & X \\
	{R\times_YS} &&&& {X \times Z} \\
	& S & {Y \times Z} & Z
	\arrow[hook, from=1-2, to=1-3]
	\arrow[from=1-3, to=1-4]
	\arrow[from=2-1, to=1-2]
	\arrow["{\exists ! t}", dashed, from=2-1, to=2-5]
	\arrow[from=2-1, to=3-2]
	\arrow[from=2-5, to=1-4]
	\arrow[from=2-5, to=3-4]
	\arrow[hook, from=3-2, to=3-3]
	\arrow[from=3-3, to=3-4]
\end{tikzcd}\]

The image of $t$ (from the canonical OFS on $\catfont{C}$) gives a subobject $R \otimes S$ of $X \times Z$:
\[\begin{tikzcd}
	{R \times_YS} & {R \otimes S} & {X \times Z}
	\arrow["{\ell_t}", two heads, from=1-1, to=1-2]
	\arrow["t"', curve={height=18pt}, from=1-1, to=1-3]
	\arrow["{r_t}", hook, from=1-2, to=1-3]
\end{tikzcd}\]

Furthermore, by the universal property of the rightmost pullback, we have a morphism $u\: R \times_Y S \to R' \times_{Y'}S'$:

\[\begin{tikzcd}
	&& R & {R'} \\
	{X \times Z} & {R\times_Y S} &&& {R'\times_{Y'} S'} & {X' \times Z'} \\
	&& S & {S'}
	\arrow["\theta", from=1-3, to=1-4]
	\arrow[from=2-2, to=1-3]
	\arrow["t", dashed, from=2-2, to=2-1]
	\arrow[""{name=0, anchor=center, inner sep=0}, "\lrcorner"{anchor=center, pos=0.125, rotate=45}, draw=none, from=2-2, to=2-5]
	\arrow["u", curve={height=-12pt}, dashed, from=2-2, to=2-5]
	\arrow[from=2-2, to=3-3]
	\arrow[from=2-5, to=1-4]
	\arrow["{{t'}}"', dashed, from=2-5, to=2-6]
	\arrow[from=2-5, to=3-4]
	\arrow["\alpha"', from=3-3, to=3-4]
	\arrow["\lrcorner"{anchor=center, pos=0.125, rotate=-135}, draw=none, from=2-5, to=0]
\end{tikzcd}\]

\if{false}
\[\begin{tikzcd}
	&&& R & {R'} \\
	{X\times Z} & {R\cdot S} & {R\times_Y S} &&& {R'\times_{Y'}S'} & {R'.S'} & {X'\times Z'} \\
	&&& S & {S'}
	\arrow["\theta", from=1-4, to=1-5]
	\arrow["{{f\times h}}"{description}, curve={height=24pt}, from=2-1, to=2-8]
	\arrow["m"', hook', from=2-2, to=2-1]
	\arrow["{{\theta.\alpha}}", curve={height=12pt}, from=2-2, to=2-7]
	\arrow[from=2-3, to=1-4]
	\arrow["e"', two heads, from=2-3, to=2-2]
	\arrow["u"{description}, curve={height=-12pt}, dashed, from=2-3, to=2-6]
	\arrow["\lrcorner"{anchor=center, pos=0.125, rotate=45}, draw=none, from=2-3, to=2-6]
	\arrow[from=2-3, to=3-4]
	\arrow[from=2-6, to=1-5]
	\arrow["\lrcorner"{anchor=center, pos=0.125, rotate=-135}, draw=none, from=2-6, to=2-3]
	\arrow["{{e'}}", two heads, from=2-6, to=2-7]
	\arrow[from=2-6, to=3-5]
	\arrow["{{m'}}", hook, from=2-7, to=2-8]
	\arrow["\alpha"', from=3-4, to=3-5]
\end{tikzcd}\]
\fi

We now examine factorizations of cells in $\mathbb{R}\mathrm{el}(\mathcal{C})$.  We can write any double cell $\alpha$ as a morphism of spans:

\[\begin{tikzcd}
	R & {X \times Y} && X && {X'} \\
	&& R && S \\
	S & {X' \times Y'} && Y && {Y'}
	\arrow["{m_1}", hook, from=1-1, to=1-2]
	\arrow["\alpha"', from=1-1, to=3-1]
	\arrow[""{name=0, anchor=center, inner sep=0}, "{f \times g}", from=1-2, to=3-2]
	\arrow["f", from=1-4, to=1-6]
	\arrow[from=2-3, to=1-4]
	\arrow["\alpha", from=2-3, to=2-5]
	\arrow[from=2-3, to=3-4]
	\arrow[from=2-5, to=1-6]
	\arrow[from=2-5, to=3-6]
	\arrow["{m_2}"', hook, from=3-1, to=3-2]
	\arrow["g"', from=3-4, to=3-6]
	\arrow["\leadsto"{marking, allow upside down, pos=0.8}, draw=none, from=0, to=2-3]
\end{tikzcd}\]

Applying the results from \ref{examplespan}, we obtain a factorization

\[\begin{tikzcd}
	& X & {\text{Im}(f)} & {X'} \\
	R & {\text{Im}(\alpha)} & S \\
	& Y & {\text{Im}(g)} & {Y'}
	\arrow["{{l_f}}", two heads, from=1-2, to=1-3]
	\arrow["{{r_f}}", hook, from=1-3, to=1-4]
	\arrow[from=2-1, to=1-2]
	\arrow["{{l_\alpha}}"', two heads, from=2-1, to=2-2]
	\arrow[from=2-1, to=3-2]
	\arrow["{{\bar f}}", dashed, from=2-2, to=1-3]
	\arrow["{{r_\alpha}}"', hook, from=2-2, to=2-3]
	\arrow["{{\bar g}}"', dashed, from=2-2, to=3-3]
	\arrow[from=2-3, to=1-4]
	\arrow[from=2-3, to=3-4]
	\arrow["{{l_g}}"', two heads, from=3-2, to=3-3]
	\arrow["{{r_g}}"', hook, from=3-3, to=3-4]
\end{tikzcd}\]

where $\bar f , \bar g$ exist by functoriality of the OFS on $\catfont{C}$.  This yields double cells:

\[\begin{tikzcd}
	R & {X \times Y} && {\text{Im}(\alpha)} & {\text{Im}(f) \times \text{Im}(g)} \\
	{\text{Im}(\alpha)} & {\text{Im}(f) \times \text{Im}(g)} && S & {X' \times Y'}
	\arrow[hook, from=1-1, to=1-2]
	\arrow["{l_\alpha}"', two heads, from=1-1, to=2-1]
	\arrow["{l_f \times l_g}", two heads, from=1-2, to=2-2]
	\arrow[hook, from=1-4, to=1-5]
	\arrow["{r_\alpha}"', hook, from=1-4, to=2-4]
	\arrow["{r_f \times r_g}", hook, from=1-5, to=2-5]
	\arrow[hook, from=2-1, to=2-2]
	\arrow[hook, from=2-4, to=2-5]
\end{tikzcd}\]

Hence, we can write any double cell in $\mathbb{R}\mathrm{el}(\mathcal{C})$ as a cell whose arrows are in the right class of $\catfont{C}$, followed by a cell whose arrows are in the left class.  Vertical closure of the right class follows the same argument as in \ref{examplespan}.  The left class, however, is not vertically closed, as we shall demonstrate later.

Next, we examine orthogonality.  Specifically, we show that there exists $v=\theta\otimes\alpha\: R\otimes S \to R'\otimes S' $ which makes the following commute: 

\[\begin{tikzcd}
	{R\otimes S} && {X\times Z} \\
	\\
	{R'\otimes S'} && {X'\times Z'}
	\arrow["r_t", hook, from=1-1, to=1-3]
	\arrow["{v}"', dashed, from=1-1, to=3-1]
	\arrow["{f\times h}", hook, from=1-3, to=3-3]
	\arrow["{r_{t'}}"', hook, from=3-1, to=3-3]
\end{tikzcd}\]

Consider the diagram:

\[\begin{tikzcd}
	{R \times_Y S} & {R' \times_Y' S'} & {R' \otimes S'} \\
	{R  \otimes S} & {X \times Z} & {X' \times Z'}
	\arrow["u", from=1-1, to=1-2]
	\arrow["\ell_t"', two heads, from=1-1, to=2-1]
	\arrow["{{{\ell_{t'}}}}", two heads, from=1-2, to=1-3]
	\arrow["r_{t'}", hook, from=1-3, to=2-3]
	\arrow["v", dashed, from=2-1, to=1-3]
	\arrow["r_t"', hook, from=2-1, to=2-2]
	\arrow["{{{f \times h}}}"',hook, from=2-2, to=2-3]
\end{tikzcd}\]

By orthogonality of the factorization system on $\catfont{C}$, there exists a unique $v$ commuting with the upper and lower triangles as required.  Furthermore, $v$ is a monomorphism because it is postcomposed with a monomorphism.  Thus, the right class of cells we have described in $\mathbb{R}\mathrm{el}(\catfont{C})$ is closed under vertical composition.

It is not always the case that the vertical composition preserves the left class, as per the next example.

\begin{example}\label{counter-ex}
    Consider the following diagram of double cells in $\mathbb{S}\mathrm{pan}(\mathcal{C})$. 
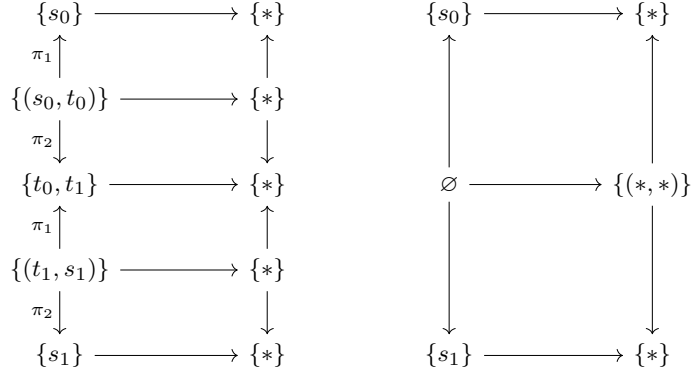
\[\begin{tikzcd}
	{\{s_0\}} && {\{*\}} && {\{s_0\}} && {\{*\}} \\
	{\{(s_0,t_0)\}} && {\{*\}} \\
	{\{t_0,t_1\}} && {\{*\}} && \varnothing && {\{(*,*)\}} \\
	{\{(t_1,s_1)\}} && {\{*\}} \\
	{\{s_1\}} && {\{*\}} && {\{s_1\}} && {\{*\}}
	\arrow[from=1-1, to=1-3]
	\arrow[from=1-5, to=1-7]
	\arrow["{\pi_1}", from=2-1, to=1-1]
	\arrow[from=2-1, to=2-3]
	\arrow["{\pi_2}"', from=2-1, to=3-1]
	\arrow[from=2-3, to=1-3]
	\arrow[from=2-3, to=3-3]
	\arrow[from=3-1, to=3-3]
	\arrow[from=3-5, to=1-5]
	\arrow[from=3-5, to=3-7]
	\arrow[from=3-5, to=5-5]
	\arrow[from=3-7, to=1-7]
	\arrow[from=3-7, to=5-7]
	\arrow["{\pi_1}", from=4-1, to=3-1]
	\arrow[from=4-1, to=4-3]
	\arrow["{\pi_2}"', from=4-1, to=5-1]
	\arrow[from=4-3, to=3-3]
	\arrow[from=4-3, to=5-3]
	\arrow[from=5-1, to=5-3]
	\arrow[from=5-5, to=5-7]
\end{tikzcd}\]
    On the left side we have two vertically  composable cells that are both in the left class $\mathcal{L}'$.
    However, their vertical composition, given on the right is clearly not in $\mathcal{L}'$: the arrow $\varnothing\rightarrow \{(*,*)\}$ is clearly not surjective.
   
\end{example}

\begin{example}{(The cospan double category on $\mathcal{C}$, $\Cospan(\mathcal{C})$).}
    Let $\mathcal{C}$ be a category with pushouts and an orthogonal factorization system $(\mathcal{L}, \mathcal{R})$.  Consider the double category $\Cospan(\mathcal{C})$, whose cells are given by commutative diagrams of the form:

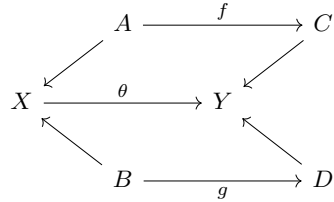
\[\begin{tikzcd}
	& A && C \\
	X && Y \\
	& B && D
	\arrow["f", from=1-2, to=1-4]
	\arrow[from=1-2, to=2-1]
	\arrow[from=1-4, to=2-3]
	\arrow["\theta", from=2-1, to=2-3]
	\arrow[from=3-2, to=2-1]
	\arrow["g"', from=3-2, to=3-4]
	\arrow[from=3-4, to=2-3]
\end{tikzcd}\]

Using the underlying OFS on $\mathcal{C}$, we can factorize any cell in the following manner, with $\bar{f}, \bar{g}$ existing by functoriality of the OFS on $\mathcal{C}$:

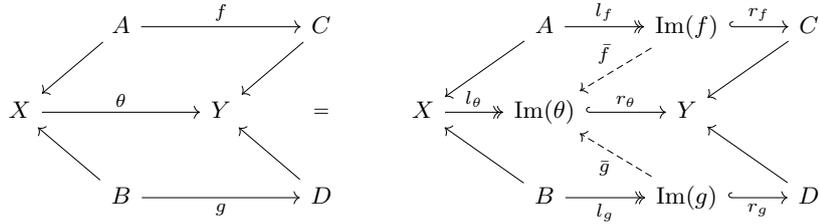
\[\begin{tikzcd}
	& A && C && A & {\im(f)} & C \\
	X && Y && X & {\im(\theta)} & Y \\
	& B && D && B & {\im(g)} & D
	\arrow["f", from=1-2, to=1-4]
	\arrow[from=1-2, to=2-1]
	\arrow[from=1-4, to=2-3]
	\arrow["{l_f}", two heads, from=1-6, to=1-7]
	\arrow[from=1-6, to=2-5]
	\arrow["{r_f}", hook, from=1-7, to=1-8]
	\arrow["{{\bar{f}}}"', dashed, from=1-7, to=2-6]
	\arrow[from=1-8, to=2-7]
	\arrow["\theta", from=2-1, to=2-3]
	\arrow["{{{=}}}"{description}, draw=none, from=2-3, to=2-5]
	\arrow["{l_\theta}", two heads, from=2-5, to=2-6]
	\arrow["{r_\theta}", hook, from=2-6, to=2-7]
	\arrow[from=3-2, to=2-1]
	\arrow["g"', from=3-2, to=3-4]
	\arrow[from=3-4, to=2-3]
	\arrow[from=3-6, to=2-5]
	\arrow["{l_g}"', two heads, from=3-6, to=3-7]
	\arrow["{{\bar{g}}}", dashed, from=3-7, to=2-6]
	\arrow["{r_g}"', hook, from=3-7, to=3-8]
	\arrow[from=3-8, to=2-7]
\end{tikzcd}\]

Using a dual argument to the case of spans, it can be shown that the left class of double cells (whose arrows are in $\mathcal{R}$) is closed under vertical composition.  However, the right class need not be closed, as the following example in $\Cospan(\Set)$ shall demonstrate.  Recall that the pushout of $A \xleftarrow{\mathit{f}} X \xrightarrow{\mathit{g}} B$ in $\operatorname{Set}$ is given by $A \sqcup B / \sim$, where $\sqcup$ is disjoint union, and \[\forall x \in X, (f(x), A) \sim (g(x),B) \] 

Consider the vertical composition of cells $\theta, \varphi$ (where $P$ and $Q$ are pushouts):
{\footnotesize
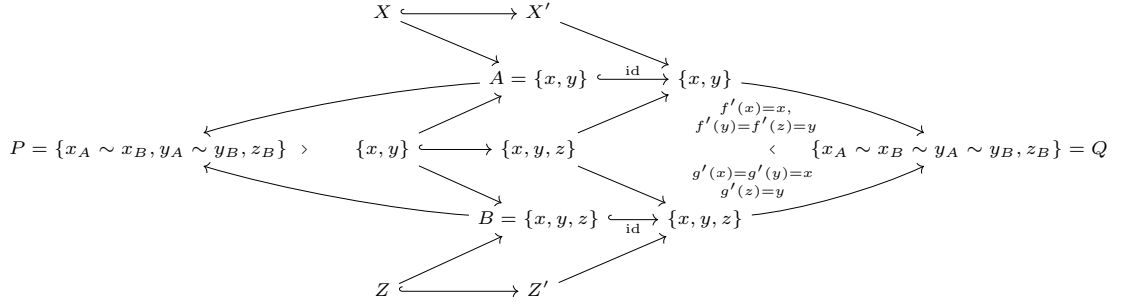
\[\begin{tikzcd}
	& X & {X'} \\
	&& {A=\{x,y\}} & {\{x,y\}} \\
	{P=\{x_A \sim x_B, y_A \sim y_B, z_B\}} & {\{x,y\}} & {\{x,y,z\}} && {\{x_A \sim x_B \sim y_A \sim y_B, z_B\}=Q} \\
	&& {B=\{x,y,z\}} & {\{x,y,z\}} \\
	& Z & {Z'}
	\arrow[hook, from=1-2, to=1-3]
	\arrow[from=1-2, to=2-3]
	\arrow[from=1-3, to=2-4]
	\arrow["{\operatorname{id}}", hook, from=2-3, to=2-4]
	\arrow[curve={height=6pt}, from=2-3, to=3-1]
	\arrow[curve={height=-6pt}, from=2-4, to=3-5]
	\arrow["\lrcorner"{anchor=center, pos=0.125, rotate=45}, draw=none, from=3-1, to=3-2]
	\arrow[from=3-2, to=2-3]
	\arrow[hook, from=3-2, to=3-3]
	\arrow[from=3-2, to=4-3]
	\arrow["\begin{array}{c} \begin{array}{c} \begin{array}{c} \begin{array}{c} \substack{f'(x)=x, \\ f'(y)=f'(z)=y} \end{array} \end{array} \end{array} \end{array}"'{pos=1}, from=3-3, to=2-4]
	\arrow["\begin{array}{c} \begin{array}{c} \begin{array}{c} \begin{array}{c} \substack{g'(x)=g'(y)=x \\ g'(z)=y} \end{array} \end{array} \end{array} \end{array}"{pos=1}, from=3-3, to=4-4]
	\arrow["\lrcorner"{anchor=center, pos=0.125, rotate=-135}, draw=none, from=3-5, to=3-3]
	\arrow[curve={height=-6pt}, from=4-3, to=3-1]
	\arrow["{\operatorname{id}}"', hook, from=4-3, to=4-4]
	\arrow[curve={height=6pt}, from=4-4, to=3-5]
	\arrow[from=5-2, to=4-3]
	\arrow[hook, from=5-2, to=5-3]
	\arrow[from=5-3, to=4-4]
\end{tikzcd}\]
}
There exists a morphism from $P$ to $Q$ induced by the pushout's universal property. However, this morphism cannot be injective, as $P$ is larger than $Q$. Therefore, the vertical composition of $\varphi$ and $\theta$ does not have a monomorphism as its center arrow, and is not in the right class.
\end{example}

\begin{example}\label{ex:PathC}[The double category of paths in $\mathcal{C}$]
    For a category $\mathcal{C}$, we define the double category $\dpath \,\C$, following the construction and notation in \citep{path-double}, such that the objects are the objects of $\C$, horizontal arrows are morphisms of $\C$, vertical arrows are paths in $\C$, i.e., composable sequences of arrows, 
\[\begin{tikzcd}[sep=scriptsize]
	{A_0} & {A_1} & {A_2} & \cdots & {A_{m-1}} & {A_m}
	\arrow["{f_1}", from=1-1, to=1-2]
	\arrow["{f_2}", from=1-2, to=1-3]
	\arrow[from=1-3, to=1-4]
	\arrow[from=1-4, to=1-5]
	\arrow["{f_m}", from=1-5, to=1-6]
\end{tikzcd}\]
We denote such a path by $\langle f_i\rangle_{1\le i \le m}$. To define the double cells in this double category, we will use the following notation for arrows in a path: 
For any $i,j \in \{1,\cdots, m\}$, such that $i\le j$, we denote the composition of the arrows in the path between $A_i$ and $A_j$ by 
\begin{equation}
f_j^i = \begin{cases}
    f_jf_{j-1}\cdots f_{i+1} & \text{if } i< j\\ \mathrm{id_{A_i}}& \text{if } i=j
\end{cases}
\end{equation}
We will also write $[m+1]$ for the set $\{0,1,\ldots,m\}$.

A double cell with boundary
\[\begin{tikzcd}
	{A_0} & {B_0} \\
	\\
	{A_m} & {B_n}
	\arrow["k", from=1-1, to=1-2]
	\arrow["{\langle f_i\rangle_{1\le i\le m}}"', from=1-1, to=3-1]
	\arrow["{\langle g_j\rangle_{1\le j\le n}}", from=1-2, to=3-2]
	\arrow["\ell"', from=3-1, to=3-2]
\end{tikzcd}\]
is then given by a pair $(\varphi, \langle h_i \rangle_{0\le i \le m})$ where
$$\varphi \: [m+1] \to [n+1]$$ is an order preserving function with $\varphi(0) = 0$ and $\varphi(m) = n$, and the $h_i \: A_i\to B_{\varphi(i)}$ are morphisms of $\C$ such that $h_0=k$, $h_m=\ell$ and for every $i \in [m+1]$, the diagram \\
\[\begin{tikzcd}[sep=scriptsize]
	{A_{i-1}} && {B_{\varphi(i-1)}} \\
	{A_i} && {B_{\varphi(i)}}
	\arrow["{h_{i-1}}", from=1-1, to=1-3]
	\arrow["{f_i}"', from=1-1, to=2-1]
	\arrow["{g^{\varphi(i-1)}_{\varphi(i)}}", from=1-3, to=2-3]
	\arrow["{=}"{marking, allow upside down}, shift left, draw=none, from=2-1, to=1-3]
	\arrow["{h_i}"', from=2-1, to=2-3]
\end{tikzcd}\]\\
commutes. These double cells are called {\em fences} in \citep{path-double}. An example of a double cell of $\dpath\,\C$ is given by,

\begin{equation}\label{ex_path_cell}
\begin{tikzcd}[sep=scriptsize]
	{A_0} & {B_0} \\
	{A_1} & {B_1} \\
	{A_2} & {B_2} \\
	{A_3} & {B_3} \\
	& {B_4}
	\arrow["{h_0}", from=1-1, to=1-2]
	\arrow["{f_1}"', from=1-1, to=2-1]
	\arrow["{=}"{marking, allow upside down}, draw=none, from=1-2, to=2-1]
	\arrow["{g_1}", from=1-2, to=2-2]
	\arrow[""{name=0, anchor=center, inner sep=0}, "{h_1}", from=2-1, to=2-2]
	\arrow["{f_2}"', from=2-1, to=3-1]
	\arrow["{g_2}", from=2-2, to=3-2]
	\arrow["{=}"{marking, allow upside down, pos=0.7}, shift left=3, draw=none, from=2-2, to=4-1]
	\arrow["{h_2}"', from=3-1, to=2-2]
	\arrow["{f_3}"', from=3-1, to=4-1]
	\arrow["{g_3}", from=3-2, to=4-2]
	\arrow["{h_3}"', from=4-1, to=5-2]
	\arrow["{g_4}", from=4-2, to=5-2]
	\arrow["{=}"{marking, allow upside down}, draw=none, from=3-1, to=0]
\end{tikzcd}\end{equation}
Vertical composition of these cells is defined by concatenation. Horizontal composition is defined by following the threads of the cross arrows and composing:
$(\psi,k_j)(\varphi,h_i)=(\psi\varphi,k_{\varphi(i)}h_i)$.
It is straightforward to check that the double category $\dpath\,\C$ is strict. 

Suppose that the category $\C$ has an OFS $(\cL,\cR)$ and $(\varphi, \langle h_i \rangle_{1\le i \le m})$ is a fence; i.e., a double cell in $\dpath\,\C$. We may than choose to take the chosen factorization in $(\cL,\cR)$ for each of the cross arrows $h_i$, $h_i=r_{h_{i}}\circ\ell_{h_i}$. We can extend this to a factorization of the fence by observing that  for all $i\in \{1, \cdots, m\}$, the commuting squares making up the fence can be factored as indicated (see example \ref{ex_quartet}):

\[\begin{tikzcd}[sep=small]
	{A_{i-1}} && {B_{\varphi(i-1)}} && {A_{i-1}} && {\mathrm{Im}(h_{i-1})} && {B_{\varphi(i-1)}} \\
	&&& {=} \\
	{A_i} && {B_{\varphi(i)}} && {A_i} && {\text{Im}(h_i)} && {B_{\varphi(i)}}
	\arrow["{h_{i-1}}", from=1-1, to=1-3]
	\arrow["{f_i}"', from=1-1, to=3-1]
	\arrow["{g^{\varphi(i-1)}_{\varphi(i)}}", from=1-3, to=3-3]
	\arrow["{l_{h_{i-1}}}", two heads, from=1-5, to=1-7]
	\arrow["{f_i}"', from=1-5, to=3-5]
	\arrow["{r_{h_{i-1}}}", hook, from=1-7, to=1-9]
	\arrow["{u_i}", dashed, from=1-7, to=3-7]
	\arrow["{g^{\varphi(i-1)}_{\varphi(i)}}", from=1-9, to=3-9]
	\arrow["{=}"{marking, allow upside down}, shift left, draw=none, from=3-1, to=1-3]
	\arrow["{h_i}"', from=3-1, to=3-3]
	\arrow["{=}"{marking, allow upside down}, draw=none, from=3-5, to=1-7]
	\arrow["{l_{h_{i}}}"', two heads, from=3-5, to=3-7]
	\arrow["{=}"{marking, allow upside down}, draw=none, from=3-7, to=1-9]
	\arrow["{r_{h_{i}}}"', hook, from=3-7, to=3-9]
\end{tikzcd}\]
The morphisms $u_i$ exist and are unique because the factorisation $(\cL, \cR)$ is orthogonal. 
Hence, every double cell $(\varphi, \langle h_i \rangle_{0\le i \le m})$ factors as the composition of a (left) double cell $(\varphi_l =\mathrm{id}_{[m+1]}, \langle l_{h_i}\rangle_{1\le i\le m})$  from $\langle f_i\rangle_{1\le i\le m}$ to $\langle u_i\rangle_{1\le i\le m}$ and a (right) double cell $(\varphi_r= \varphi, \langle r_{h_i}\rangle_{0\le i\le m})$ from $\langle u_i\rangle_{1\le i\le m}$ to $\langle g_i\rangle_{1\le i\le n}$.  For example, the double cell \ref{ex_path_cell} factorises as follows.

\[\begin{tikzcd}[sep=scriptsize]
	{A_0} && {\mathrm{Im(h_0)}} && {B_0} \\
	{A_1} && {\mathrm{Im(h_1)}} && {B_1} \\
	{A_2} && {\mathrm{Im(h_2)}} && {B_2} \\
	{A_3} && {\mathrm{Im(h_3)}} && {B_3} \\
	&&&& {B_4}
	\arrow["{l_{h_0}}", two heads, from=1-1, to=1-3]
	\arrow["{f_1}"', from=1-1, to=2-1]
	\arrow["{r_{h_0}}", hook, from=1-3, to=1-5]
	\arrow["{u_1}"', dashed, from=1-3, to=2-3]
	\arrow["{g_1}", from=1-5, to=2-5]
	\arrow["{l_{h_1}}", two heads, from=2-1, to=2-3]
	\arrow["{f_2}"', from=2-1, to=3-1]
	\arrow["{r_{h_1}}", hook, from=2-3, to=2-5]
	\arrow["{u_2}"', dashed, from=2-3, to=3-3]
	\arrow["{g_2}", from=2-5, to=3-5]
	\arrow["{l_{h_2}}"', two heads, from=3-1, to=3-3]
	\arrow["{f_3}"', from=3-1, to=4-1]
	\arrow["{r_{h_2}}"', hook, from=3-3, to=2-5]
	\arrow["{u_3}"', dashed, from=3-3, to=4-3]
	\arrow["{g_3}", from=3-5, to=4-5]
	\arrow["{l_{h_3}}"', two heads, from=4-1, to=4-3]
	\arrow["{r_{h_3}}"', hook, from=4-3, to=5-5]
	\arrow["{g_4}", from=4-5, to=5-5]
\end{tikzcd}\]
Note that double cells where the indexing map $\varphi$ is the identity are called {\em neat fences} in \citep{path-double}.

This leads us to introduce the  factorization system $(\cL',\cR')$ on the double cells of $\dpath\,\C$ where $\cL'$ 
is the class of neat fences $(\mathrm{id}_{[m+1]}, \langle l_{h_i}\rangle_{0\le i\le m})$ with cross arrows  $h_i\in\cL$ and $\cR'$ is the class of fences $(\varphi, \langle l_{h_i}\rangle_{0\le i\le m})$ where all $h_i$ are in $\cR$. In short, $\cL'$ consists of neat $\cL$-fences and $\cR'$ consists of arbitrary $\cR$-fences.
It is clear that both $\cL'$ and $\cR'$ contain all horizontally invertible double cells, so in order to prove that 
$(\cL',\cR')$ is an OFS, it is sufficient to check the unique lifting property.
So consider a horizontal commutative square of fences,

\begin{equation}\label{fence-lifting}
\begin{tikzcd}
	{\langle f_i\rangle_{1\le i\le m}} && {\langle g_i\rangle_{1\le i\le n}} \\
	{\langle f'_i\rangle_{1\le i\le m}} && {\langle g_i'\rangle_{1\le i\le n'}}
	\arrow["{(\varphi,\langle  h_i\rangle_{0\le i\le m})}", from=1-1, to=1-3]
	\arrow["{(\mathrm{id}_{[m+1]},\langle\ell_i\rangle_{0\le i\le m})}"', from=1-1, to=2-1]
	\arrow["{(\psi,\langle  r_i\rangle_{0\le i\le n})}", from=1-3, to=2-3]
	\arrow["D"{description}, dashed, from=2-1, to=1-3]
	\arrow["{(\varphi',\langle  h'_i\rangle_{0\le i\le m})}"', from=2-1, to=2-3]
\end{tikzcd}
\end{equation}
where the arrows $\ell_i$ in the left fence are in $\cL$ and the arrows $r_i$ in the right fence are in $\cR$.
We want to show that there is a unique fence $D$ providing the diagonal lifting.
First note that by the commutativity of the square, we have that for each $i\in[m+1]$, there is a commutative square in $\C$,

\begin{equation}\label{lifting-in-C}
\begin{tikzcd}
	{A_i} & {B_{\varphi(i)}} \\
	{A'_i} & {B'_{\varphi'(i)}}
	\arrow["{h_i}", from=1-1, to=1-2]
	\arrow["{\ell_i}"', from=1-1, to=2-1]
	\arrow["{r_{\varphi(i)}}", from=1-2, to=2-2]
	\arrow["{d_i}"', dashed, from=2-1, to=1-2]
	\arrow["{h_i'}"', from=2-1, to=2-2]
\end{tikzcd}
\end{equation}
Since $\ell_i\in\cL$ and $r_{\varphi(i)}\in\cR$, there is a unique arrow $d_i\in\C$ making this diagram commute. Now these arrows together define a fence $D=(\varphi,\langle d_i\rangle_{0\le i\le m})$
which forms a diagonal lifting for diagram (\ref{fence-lifting}).
To show its uniqueness, let $D'=(\theta,\langle d_i'\rangle_{0\le i\le m})$ be another lifting for this diagram.
This implies that $\theta=\varphi$, because the top triangle commutes. So each arrow $d_i'$ is of the form $d_i'\colon A'_i\to B_{\varphi(i)}$ and the way composition of fences is defined this implies that $d_i'$ is a lifting for (\ref{lifting-in-C}).
Hence, $d_i=d_i'$ and hence $D'=D$: (\ref{fence-lifting}) has a unique diagonal lifting and we conclude that $(\cL',\cR')$ is an OFS for the double cells of $\dpath\,\C$.

Since vertical composition is defined by concatenation, it is clear that it preserves both the left and the right class of cells in this case. Furthermore, it preserves the chosen factorizations, so it is a strict map of orthogonal factorization systems.
\end{example}

\begin{rmk}
We have presented examples of cell-wise factorization systems, suitably compatible with a factorization system on the arrows, on three common families of double categories.  
\begin{enumerate}
    \item Note that in all cases the source, target and unit functors are strict morphisms of categories with an OFS.  In light of Proposition \ref{propfacthaspullbacks} we will see that this is essential in both the internal category and the monadic approach toward double factorization systems. 
    \item Vertical composition can be allowed to be a strict, lax, colax or pseudo morphism of categories with an OFS. We have seen here an example of two double categories with a strict vertical composition morphisms (quartets and paths), two  with a lax vertical composition (the examples of relations and spans) and one with a  colax vertical composition (the example of copans). 
    \item These examples also lead us to ask whether there may be more than one factorization system on the double cells of a double category that `fits' with a given OFS on the arrows. We will discuss this in Proposition \ref{prop:ofsExtensions} and Theorem \ref{thm:comparaisonCartDOFS}. In particular we will see that $\mathbb{R}\mathrm{el}(\mathcal{C})$ has another DOFS with the same OFS on the arrows, and  $\mathbb{S}\mathrm{pan}(\mathcal{C})$ has at least three other factorization systems and we will discuss the ways the factorization systems are related to each other.
\end{enumerate} 
\end{rmk}

\subsection{Internal definition}
We will define the notion of a double orthogonal factorization system (DOFS) (respectively a lax DOFS) as a normal pseudo-category internal to  $\Fact$ (respectively $\Factlax$), with strict source, target and identity assignment. (Note that one can define a colax DOFS in a similar way, but we will leave this to the reader.) Recall from Section~\ref{sectioncategoriesoffs} the definition of these 2-categories of categories with an orthogonal factorization system, and the result (\prox\ref{propfacthaspullbacks}) that they have the pullbacks we need below.

We first recall the definition of pseudo-category internal to a 2-category from \citep{ferreira2006pseudo}.

\begin{definition}\label{recallinternalcats}
    Let $\K$ be a 2-category where all pullbacks depicted below exist.  A \dfn{pseudo-category $\dC$ internal to $\K$} consists of objects $C_0$ and $C_1$ of $\K$ (called the object of objects and the object of arrows, respectively) together with morphisms $\src$, $\tgt$, $i$ and $\otimes$ (source, target, identity and composition, respectively) as below
    \[\begin{tikzcd}
	{C_1\times_{C_0}C_1} & {C_1} & {C_0}
	\arrow["\otimes", from=1-1, to=1-2]
	\arrow["{\operatorname{src}}", shift left=2, from=1-2, to=1-3]
	\arrow["{\operatorname{tgt}}"', shift right=2, from=1-2, to=1-3]
	\arrow["i"{description}, from=1-3, to=1-2]
    \end{tikzcd}\]
    and invertible 2-cells $\alpha, \lambda$, and $\rho$:
\[\begin{tikzcd}
	{C_1 \times_{C_0}C_1 \times _{C_0} C_1} & {C_1 \times_{C_0}C_1} \\
	{C_1 \times_{C_0} C_1} & {C_1}
	\arrow["{{1 \times \otimes}}", from=1-1, to=1-2]
	\arrow["{{{\otimes \times 1}}}"', from=1-1, to=2-1]
	\arrow["\alpha"', Rightarrow, from=1-2, to=2-1, shorten <=3.4ex, shorten >=2.9ex]
	\arrow["\otimes", from=1-2, to=2-2]
	\arrow["\otimes"', from=2-1, to=2-2]
\end{tikzcd}\]
\[\begin{tikzcd}
	{C_1} & {C_1\times_{C_0} C_1} & {C_1} \\
	& {C_1}
	\arrow["{{{{ \left< i,1 \right>}}}}", from=1-1, to=1-2]
	\arrow[""{name=0, anchor=center, inner sep=0}, "{\operatorname{id_{C_1}}}"', curve={height=6pt}, from=1-1, to=2-2]
	\arrow["\otimes"{description}, from=1-2, to=2-2]
	\arrow["{{{{ \left< 1,i \right>}}}}"', from=1-3, to=1-2]
	\arrow[""{name=1, anchor=center, inner sep=0}, "{\operatorname{id_{C_1}}}", curve={height=-6pt}, from=1-3, to=2-2]
	\arrow["\rho", shorten >=4pt, Rightarrow, from=1-2, to=1]
	\arrow["\lambda"', shorten >=3pt, Rightarrow, from=1-2, to=0]
\end{tikzcd}\]

The above cells are further required to satisfy  the usual  coherence conditions, along with the usual equalities for arrow composition (with several weakened by isomorphisms satisfying further coherence conditions).  These requirements are fully listed in \citep{ferreira2006pseudo}.  When the 2-cells $\alpha, \lambda$, and $\rho$ are identities, we say that the pseudo-category internal to $\K$ is \dfn{strict}, or a \dfn{category internal to $\K$}.
We will assume pseudo-categories to be normal; i.e, with $\lambda$ and $\rho$ identities.
\end{definition}

In this paper we will work with a 2-category of pseudo-categories internal to a 2-category $\K$. The morphisms will be the internal lax functors as described in the next definition.

\begin{definition}{}
    Suppose that $\mathbb{C}, \mathbb{D}$ are normal pseudo-categories internal to a 2-category $\catfont{K}$.  An internal lax functor $F: \mathbb{C} \to \mathbb{D}$ consists of the following:
    \begin{itemize}
        \item Morphisms $F_0\: C_0 \to D_0$ and $F_1\: C_1 \to D_1$ in $\K$;
        \item 2-Cells $\varphi, \theta$ in $\K$ as follows:
        \[\begin{tikzcd}
	{C_1 \times_{C_0}C_1} & {C_1} && {C_0} & {C_1} \\
	{D_1 \times_{D_0}D_1} & {D_1} && {D_0} & {D_1}
	\arrow["\otimes", from=1-1, to=1-2]
	\arrow[""{name=0, anchor=center, inner sep=0}, "{F_1 \times_{F_0}F_1}"', from=1-1, to=2-1]
	\arrow[""{name=1, anchor=center, inner sep=0}, "{F_1}", from=1-2, to=2-2]
	\arrow["{i}", from=1-4, to=1-5]
	\arrow[""{name=2, anchor=center, inner sep=0}, "{F_0}"', from=1-4, to=2-4]
	\arrow[""{name=3, anchor=center, inner sep=0}, "{F_1}", from=1-5, to=2-5]
	\arrow["\otimes"', from=2-1, to=2-2]
	\arrow["{i}"', from=2-4, to=2-5]
	\arrow["\varphi", shorten <=8pt, shorten >=8pt, Rightarrow, from=0, to=1]
	\arrow["\theta", shorten <=6pt, shorten >=6pt, Rightarrow, from=2, to=3]
\end{tikzcd}\]
satisfying several coherence conditions (see \citep{ferreira2006pseudo}).
        \end{itemize}
\end{definition}
When $\theta$ is an identity, $F$ is said to be \emph{unitary} or \emph{normal}. When $\phi$ and $\theta$ are isomorphisms, $F$ is said to be a pseudofunctor. Oplax notions of internal functors are obtained by dualizing appropriately.
\begin{definition}
    Given lax internal functors $F,G\: \mathbb{C} \to \mathbb{D}$, an internal, or levelwise, natural transformation $\eta\: F \Rightarrow G$ consists of 2-cells $\eta_0\: F_0 \Rightarrow G_0, \eta_1\: F_1 \Rightarrow G_1$ satisfying the coherence conditions in Section 3.2 of \citep{ferreira2006pseudo}.
\end{definition}

\begin{notation}
    Given a 2-category $\K$, we denote by $\Catint+[p]{\K}$ (respectively, $\Catint+[l]{\K}$, $\Catint+[o]{\K}$) the 2-category of normal (i.e, unitary) pseudo-categories internal to $\K$, normal internal pseudo (respectively, normal lax, normal oplax) functors and internal natural transformations. It is straightforward to show that these are indeed 2-categories.
\end{notation}

\begin{definition}\label{def:internalDOFS}
    A strict (resp.~lax, resp.~colax) double orthogonal factorization system (DOFS) is a pseudo-category internal to the 2-category $\Fact$ (resp.~$\Factlax$; 
    resp.~$\Factcolax$), such that $\src$, $\tgt$ and $i$ are strict morphisms of categories with an orthogonal factorization system.
\end{definition}

\begin{remark}
    Thus, a (lax) DOFS $(\dL,\dR)$ on a double category $\dC$ amounts to an OFS $(\cL_0,\cR_0)$ for horizontal arrows (with chosen factorizations), along with left and right classes of double cells which form an OFS $(\cL_1,\cR_1)$ with respect to horizontal composition of cells (viewed as arrows of $\mathcal{C}_1$ in \ref{def:internalDOFS}) and chosen factorizations of the cells.  These factorization systems are compatible, in the sense that source and target functors preserve the chosen factorizations. The identity cell on a horizontal arrow of the left class (respectively, right class) is in the left class (respectively, right class) and the chosen factorizations on identity cells for horizontal arrows consist of the identity cells for the chosen factorization of the horizontal arrow. In a strict DOFS, vertical composition preserves the chosen factorizations of the cells; for a pseudo DOFS, both classes of cells are closed under vertical composition; in the lax case, only the right class is vertically closed, and in the colax case only the left class is closed.  To make this more explicit, we offer a cell-wise definition of a DOFS, and show that it is equivalent to the internal one.
\end{remark}

\begin{proposition}\label{prop:explicitDOFS}
    A DOFS
 is equivalent to a double category $\dD$ with a pair of classes of cells $(\dL,\dR)$ satisfying: 
    \begin{enumerate}
        \item $\dL$ and $\dR$ are closed under horizontal composition, and contain all the horizontally invertible cells;
        \item $\dL$ and $\dR$ are closed under vertical composition;
        \item for any cell in either class, the source and target unit cells are also contained in the respective class; e.g., if $\alpha\in \dL$, then $i(\src(\alpha))\in\dL$.
        \item For every cell $\varphi$ there are chosen cells $\ell_\varphi\in \dL$ and $r_\varphi\in\dR$ such that
        \begin{enumerate}
        \item  
        \[\begin{tikzcd}
	\bullet & \bullet & \bullet & \bullet & \bullet \\
	\bullet & \bullet & \bullet & \bullet & \bullet
	\arrow[from=1-1, to=1-2]
	\arrow[""{name=0, anchor=center, inner sep=0}, "\shortmid"{marking}, from=1-1, to=2-1]
	\arrow[""{name=1, anchor=center, inner sep=0}, "\shortmid"{marking}, from=1-2, to=2-2]
	\arrow[from=1-3, to=1-4]
	\arrow[""{name=2, anchor=center, inner sep=0}, "\shortmid"{marking}, from=1-3, to=2-3]
	\arrow[from=1-4, to=1-5]
	\arrow[""{name=3, anchor=center, inner sep=0}, "\shortmid"{marking}, from=1-4, to=2-4]
	\arrow[""{name=4, anchor=center, inner sep=0}, "\shortmid"{marking}, from=1-5, to=2-5]
	\arrow[from=2-1, to=2-2]
	\arrow[from=2-3, to=2-4]
	\arrow[from=2-4, to=2-5]
	\arrow["\varphi"{description}, draw=none, from=0, to=1]
	\arrow["{=}"{description}, draw=none, from=1, to=2]
	\arrow["{{\ell_\varphi}}"{description}, draw=none, from=2, to=3]
	\arrow["{{r_\varphi}}"{description}, draw=none, from=3, to=4]
\end{tikzcd} \]
        \item If two cells share the same source (respectively, target), then the corresponding chosen cells have the same source (respectively, target); e.g., if $\src(\alpha)=\src(\beta)$ then $\src(\ell_\alpha)=\src(\ell_\beta)$.
        \item The chosen cells of a unit cell are again unit cells.
        \end{enumerate}
        
        \item (the lifting property) for any equation \[\begin{tikzcd}
	\bullet & \bullet & \bullet & \bullet & \bullet & \bullet \\
	\bullet & \bullet & \bullet & \bullet & \bullet & \bullet
	\arrow[from=1-1, to=1-2]
	\arrow[""{name=0, anchor=center, inner sep=0}, "\shortmid"{marking}, from=1-1, to=2-1]
	\arrow[from=1-2, to=1-3]
	\arrow[""{name=1, anchor=center, inner sep=0}, "\shortmid"{marking}, from=1-2, to=2-2]
	\arrow[""{name=2, anchor=center, inner sep=0}, "\shortmid"{marking}, from=1-3, to=2-3]
	\arrow[from=1-4, to=1-5]
	\arrow[""{name=3, anchor=center, inner sep=0}, "\shortmid"{marking}, from=1-4, to=2-4]
	\arrow[from=1-5, to=1-6]
	\arrow[""{name=4, anchor=center, inner sep=0}, "\shortmid"{marking}, from=1-5, to=2-5]
	\arrow[""{name=5, anchor=center, inner sep=0}, "\shortmid"{marking}, from=1-6, to=2-6]
	\arrow[from=2-1, to=2-2]
	\arrow[from=2-2, to=2-3]
	\arrow[from=2-4, to=2-5]
	\arrow[from=2-5, to=2-6]
	\arrow["{{\ell}}"{description}, draw=none, from=0, to=1]
	\arrow["\varphi"{description}, draw=none, from=1, to=2]
	\arrow["{{=}}"{description}, draw=none, from=2, to=3]
	\arrow["\psi"{description}, draw=none, from=3, to=4]
	\arrow["{{r}}"{description}, draw=none, from=4, to=5]
\end{tikzcd}\] with $\ell \in \dL$ and $r\in\dR$, there is a unique cell $\theta$ that satisfies the two equations
\[\begin{tikzcd}
	\bullet & \bullet & \bullet & \bullet & \bullet && \bullet & \bullet & \bullet & \bullet & \bullet \\
	\bullet & \bullet & \bullet & \bullet & \bullet && \bullet & \bullet & \bullet & \bullet & \bullet
	\arrow[from=1-1, to=1-2]
	\arrow[""{name=0, anchor=center, inner sep=0}, "\shortmid"{marking}, from=1-1, to=2-1]
	\arrow[""{name=1, anchor=center, inner sep=0}, "\shortmid"{marking}, from=1-2, to=2-2]
	\arrow[from=1-3, to=1-4]
	\arrow[""{name=2, anchor=center, inner sep=0}, "\shortmid"{marking}, from=1-3, to=2-3]
	\arrow[from=1-4, to=1-5]
	\arrow[""{name=3, anchor=center, inner sep=0}, "\shortmid"{marking}, from=1-4, to=2-4]
	\arrow[""{name=4, anchor=center, inner sep=0}, "\shortmid"{marking}, from=1-5, to=2-5]
	\arrow[from=1-7, to=1-8]
	\arrow[""{name=5, anchor=center, inner sep=0}, "\shortmid"{marking}, from=1-7, to=2-7]
	\arrow[""{name=6, anchor=center, inner sep=0}, "\shortmid"{marking}, from=1-8, to=2-8]
	\arrow[from=1-9, to=1-10]
	\arrow[""{name=7, anchor=center, inner sep=0}, "\shortmid"{marking}, from=1-9, to=2-9]
	\arrow[from=1-10, to=1-11]
	\arrow[""{name=8, anchor=center, inner sep=0}, "\shortmid"{marking}, from=1-10, to=2-10]
	\arrow[""{name=9, anchor=center, inner sep=0}, "\shortmid"{marking}, from=1-11, to=2-11]
	\arrow[from=2-1, to=2-2]
	\arrow[from=2-3, to=2-4]
	\arrow[from=2-4, to=2-5]
	\arrow[from=2-7, to=2-8]
	\arrow[from=2-9, to=2-10]
	\arrow[from=2-10, to=2-11]
	\arrow["\psi"{description}, draw=none, from=0, to=1]
	\arrow["{{=}}"{description}, draw=none, from=1, to=2]
	\arrow["{{\ell}}"{description}, draw=none, from=2, to=3]
	\arrow["\theta"{description}, draw=none, from=3, to=4]
	\arrow["{{\text{and}}}"{description}, draw=none, from=4, to=5]
	\arrow["\varphi"{description}, draw=none, from=5, to=6]
	\arrow["{{=}}"{description}, draw=none, from=6, to=7]
	\arrow["\theta"{description}, draw=none, from=7, to=8]
	\arrow["{{r}}"{description}, draw=none, from=8, to=9]
\end{tikzcd}\]
In the particular case that the cells $\ell, \varphi, \psi,r$ are unit cells, then $\theta$ is also a unit cell.
    \end{enumerate}
    
    A lax DOFS is equivalent to a double category $\dD$ with a pair of classes $(\dL, \dR)$ satisfying the above conditions, but replacing (2) by:
    \begin{enumerate}
        \item[2'] $\dR$ is closed under vertical composition. 
    \end{enumerate}

    A colax DOFS is equivalent to a double category $\dD$ with a pair of classes $(\dL, \dR)$ satisfying the above conditions, but replacing (2) by: 
    \begin{enumerate}
        \item[2''] $\dL$ is closed under vertical composition. 
    \end{enumerate}
\end{proposition}

\begin{proof} We first show that a double category $\mathbb{D}$ with a DOFS $(\mathbb{L},\mathbb{R})$ satisfies the conditions of this Proposition, by taking  $\dL := \mathcal{L}_1$ and $\dR := \mathcal{R}_1$. 

Condition (1): Both of $\dL$ and $\dR$ are closed under horizontal composition and contain the horizontally invertible cells by assumption. 

Condition (2): This follows from the fact that the vertical composition is required to be a morphism of categories with orthogonal factorization systems. 

Condition (3): This is a consequence of $\src$, $\tgt$, and $i$ being morphisms of categories with orthogonal factorization systems. 

Condition (4): Here we use the chosen factorization of the cells as a cells in $\dR$ after a cell in $\dL$.
This is expressed in (4)(a). Condition (4)(b) follows from the fact that $\src$ and $\tgt$ are strict morphisms of categories with orthogonal factorization systems. Condition (4)(c) follows from thet fact that $i$ is a strict morphism of categories with orthogonal factorization systems. 

Condition (5): This  follows from the lifting property on both orthogonal factorization systems $(\mathcal{L}_1,\mathcal{R}_1)$ and $(\mathcal{L}_0,\mathcal{R}_0)$. 

 Now, a double category $\dD$ with a pair of classes of cells $(\dL,\dR)$ satisfying the conditions of the proposition gives rise to: the classes $\mathcal{L}_0$ and $\mathcal{R}_0$  of morphisms in $\D_0$ that consist of the horizontal morphisms whose associated unit cells lie in $\dL$ and $\dR$, respectively; and the classes $\mathcal{L}_1 := \dL$ and $\mathcal{R}_1 := \dR$, of $\D_1$. 
 
 The pair $(\mathcal{L}_1,\mathcal{R}_1)$ is a factorization system for $\mathcal{D}_1$ because $\dL$ and $\dR$ are both closed under horizontal composition, both have all horizontally invertible cells, condition (4)(a) gives the chosen factorizations and condition (5) gives the lifting property. 
 
 We now show that the pair $(\mathcal{L}_0,\mathcal{R}_0)$ is an orthogonal factorization system in $\D_0$. Indeed, from the fact that horizontal composition of unit cells is again a unit, and $\dL$ and $\dR$ are closed under horizontal composition, it follows that $\mathcal{L}_0$ and  $\mathcal{R}_0$ are closed under composition. They contain all isomorphisms in $\D_0$ because $\dL$ and $\dR$ contain all horizontally invertible cells. Every morphism $f$ in $\D_0$ factors as a morphism in $\mathcal{L}_0$ followed by morphism in $\mathcal{R}_0$ by conditions (4)(a) and (4)(c), and we take these factorizations as the chosen ones. The lifting property in $(\mathcal{L}_0,\mathcal{R}_0)$ follows from the last part of condition (5). 
 
 The functors $\src{}$ and $\tgt$ are strict morphisms of categories with orthogonal factorization systems by condition (3) and condition (4)(b). The unit functor $i$ is a strict morphism of categories with orthogonal factorization systems by definition of the pairs $(\mathcal{L}_0,\mathcal{R}_0)$ and $(\mathcal{L}_1,\mathcal{R}_1)$ and conditions (4)(b) and (4)(c) and our choice of factorizations. The vertical composition is a morphism of categories with orthogonal factorization systems by condition (2).
 The proof of the equivalence for the lax (respectively, colax) DOFS is exactly the same, but noticing that axiom (2') (respectively, axiom (2'')) implies the vertical composition is a lax (respectively, colax) morphism of categories with orthogonal factorization system.  
\end{proof}

We now present a result giving an equivalent formulation of Condition (5).  This result will serve as a supporting tool in examples in Subsection \ref{subsec:2ortFactSys}.

\begin{proposition}\label{prop:ax5eqv}
	Condition (5) in Proposition \ref{prop:explicitDOFS} is equivalent to the following statements:
	\begin{enumerate}
    \item[5'] For any equation 
	\[\begin{tikzcd}
	\bullet & \bullet & \bullet & \bullet & \bullet & \bullet \\
	\bullet & \bullet & \bullet & \bullet & \bullet & \bullet
	\arrow[from=1-1, to=1-2]
	\arrow[""{name=0, anchor=center, inner sep=0}, "\shortmid"{marking}, from=1-1, to=2-1]
	\arrow[from=1-2, to=1-3]
	\arrow[""{name=1, anchor=center, inner sep=0}, "\shortmid"{marking}, from=1-2, to=2-2]
	\arrow[""{name=2, anchor=center, inner sep=0}, "\shortmid"{marking}, from=1-3, to=2-3]
	\arrow[from=1-4, to=1-5]
	\arrow[""{name=3, anchor=center, inner sep=0}, "\shortmid"{marking}, from=1-4, to=2-4]
	\arrow[from=1-5, to=1-6]
	\arrow[""{name=4, anchor=center, inner sep=0}, "\shortmid"{marking}, from=1-5, to=2-5]
	\arrow[""{name=5, anchor=center, inner sep=0}, "\shortmid"{marking}, from=1-6, to=2-6]
	\arrow[from=2-1, to=2-2]
	\arrow[from=2-2, to=2-3]
	\arrow[from=2-4, to=2-5]
	\arrow[from=2-5, to=2-6]
	\arrow["{{\ell}}"{description}, draw=none, from=0, to=1]
	\arrow["{{r}}"{description}, draw=none, from=1, to=2]
	\arrow["{{=}}"{description}, draw=none, from=2, to=3]
	\arrow["{{\ell'}}"{description}, draw=none, from=3, to=4]
	\arrow["{{r'}}"{description}, draw=none, from=4, to=5]
\end{tikzcd}\]
		with $\ell, \ell' \in \mathbb{L}$ and $r$, $r'\in\mathbb{R}$, there is a unique horizontally invertible cell $\theta$ that satisfies the following two equations
		\[\begin{tikzcd}
	\bullet & \bullet & \bullet & \bullet & \bullet && \bullet & \bullet & \bullet & \bullet & \bullet \\
	\bullet & \bullet & \bullet & \bullet & \bullet && \bullet & \bullet & \bullet & \bullet & \bullet
	\arrow[from=1-1, to=1-2]
	\arrow[""{name=0, anchor=center, inner sep=0}, "\shortmid"{marking}, from=1-1, to=2-1]
	\arrow["\cong" {description}, from=1-2, to=1-3]
	\arrow[""{name=1, anchor=center, inner sep=0}, "\shortmid"{marking}, from=1-2, to=2-2]
	\arrow[""{name=2, anchor=center, inner sep=0}, "\shortmid"{marking}, from=1-3, to=2-3]
	\arrow[from=1-4, to=1-5]
	\arrow[""{name=3, anchor=center, inner sep=0}, "\shortmid"{marking}, from=1-4, to=2-4]
	\arrow[""{name=4, anchor=center, inner sep=0}, "\shortmid"{marking}, from=1-5, to=2-5]
	\arrow["\cong" {description}, from=1-7, to=1-8]
	\arrow[""{name=5, anchor=center, inner sep=0}, "\shortmid"{marking}, from=1-7, to=2-7]
	\arrow[from=1-8, to=1-9]
	\arrow[""{name=6, anchor=center, inner sep=0}, "\shortmid"{marking}, from=1-8, to=2-8]
	\arrow[""{name=7, anchor=center, inner sep=0}, "\shortmid"{marking}, from=1-9, to=2-9]
	\arrow[from=1-10, to=1-11]
	\arrow[""{name=8, anchor=center, inner sep=0}, "\shortmid"{marking}, from=1-10, to=2-10]
	\arrow[""{name=9, anchor=center, inner sep=0}, "\shortmid"{marking}, from=1-11, to=2-11]
	\arrow[from=2-1, to=2-2]
	\arrow["\cong" {description}, from=2-2, to=2-3]
	\arrow[from=2-4, to=2-5]
	\arrow["\cong" {description}, from=2-7, to=2-8]
	\arrow[from=2-8, to=2-9]
	\arrow[from=2-10, to=2-11]
	\arrow["{{\ell}}"{description}, draw=none, from=0, to=1]
	\arrow["\theta"{description}, draw=none, from=1, to=2]
	\arrow["{{=}}"{description}, draw=none, from=2, to=3]
	\arrow["{{\ell'}}"{description}, draw=none, from=3, to=4]
	\arrow["{{\text{and}}}"{description}, draw=none, from=4, to=5]
	\arrow["\theta"{description}, draw=none, from=5, to=6]
	\arrow["{{r'}}"{description}, draw=none, from=6, to=7]
	\arrow["{{=}}"{description}, draw=none, from=7, to=8]
	\arrow["{{r}}"{description}, draw=none, from=8, to=9]
\end{tikzcd}\]
When $\ell,r,\ell', r'$ are unit cells, $\theta$ is also a unit cell.
    \end{enumerate} 
\end{proposition}

\begin{proof}
    The equivalence follows from the analogous versions on orthogonal factorization systems in one-dimensional categories.
\end{proof}

The following observation shows that the cell $\theta$ in Proposition \ref{prop:ax5eqv} is a globular cell if the horizontal sources and targets coincide in two different factorizations.

\begin{observation}\label{obs:globularTheta}
    In Proposition \ref{prop:ax5eqv}, consider the particular case
    \[\begin{tikzcd}
	\bullet & \bullet & \bullet & \bullet & \bullet & \bullet \\
	\bullet & \bullet & \bullet & \bullet & \bullet & \bullet
	\arrow["e", from=1-1, to=1-2]
	\arrow[""{name=0, anchor=center, inner sep=0}, "\shortmid"{marking}, from=1-1, to=2-1]
	\arrow["m", from=1-2, to=1-3]
	\arrow[""{name=1, anchor=center, inner sep=0}, "\shortmid"{marking}, from=1-2, to=2-2]
	\arrow[""{name=2, anchor=center, inner sep=0}, "\shortmid"{marking}, from=1-3, to=2-3]
	\arrow["e", from=1-4, to=1-5]
	\arrow[""{name=3, anchor=center, inner sep=0}, "\shortmid"{marking}, from=1-4, to=2-4]
	\arrow["m", from=1-5, to=1-6]
	\arrow[""{name=4, anchor=center, inner sep=0}, "\shortmid"{marking}, from=1-5, to=2-5]
	\arrow[""{name=5, anchor=center, inner sep=0}, "\shortmid"{marking}, from=1-6, to=2-6]
	\arrow["{e'}"', from=2-1, to=2-2]
	\arrow["{m'}"', from=2-2, to=2-3]
	\arrow["{e'}"', from=2-4, to=2-5]
	\arrow["{m'}"', from=2-5, to=2-6]
	\arrow["\ell"{description}, draw=none, from=0, to=1]
	\arrow["r"{description}, draw=none, from=1, to=2]
	\arrow["{=}"{description}, draw=none, from=2, to=3]
	\arrow["{\ell'}"{description}, draw=none, from=3, to=4]
	\arrow["{r'}"{description}, draw=none, from=4, to=5]
\end{tikzcd}\]
where $\ell$ and $\ell'$ have the same source and target, and  $r$ and $r'$ have the same source and target. Now consider the equations $\src(\theta)\circ e = e$ and $m\circ \src(\theta)=m$. The orthogonal property in the orthogonal factorization system in $\mathbb{D}_0$ implies that $src(\theta) = \id{}$. Analogously $\tgt(\theta) = \id{}$. This asserts that the unique cell arising from two different factorizations of the same cell, with the same sources and targets, is horizontally globular.
\end{observation}

\subsection{Examples of DOFS on double categories constructed from 2-categories}\label{subsec:2ortFactSys}

In this subsection, we present an example of a DOFS in the double category of quintets on $\Cat$, together with an example in the double category of paths over the double category of quintets on a 2-category $\mathcal{B}$.  

\begin{example}
(The quintet double category on the 2-category $\Cat$, $\mathbb{Q}(\Cat)$). Our aim is that the class $\mathbb{L}$ that consists of cells with final functors in their horizontal arrows and the class $\mathbb{R}$ that consists of cells that are natural isomorphism with discrete fibrations in their horizontal arrows is a DOFS in $\Cat$. First of all, final functors and discrete fibrations form an OFS in the category of categories \citep{Street1973TheCF}. Thanks to that fact, it is clear that both $\mathbb{L}$ and $\mathbb{R}$ are closed under horizontal composition. Also, by definition, the classes $\mathbb{L}$ and $\mathbb{R}$, contain all the horizontally invertible cells. Now, we want to show that the axioms (1), (2), (3), (4) in Definition \ref{prop:explicitDOFS} and axiom (5') in Proposition \ref{prop:ax5eqv} hold for $\mathbb{L}$ and $\mathbb{R}$. It is clear that axioms (1), (3), and (4) hold. We will focus on axiom (2) and axiom (5'). 

    Axiom (2). We wish to factor an arbitrary cell $\alpha$ into two cells 
    \[\begin{tikzcd}[ampersand replacement=\&]
	A \& B \& A \& {X_F} \& B \\
	C \& D \& D \& {X_G} \& D
	\arrow["F", from=1-1, to=1-2]
	\arrow["U"', from=1-1, to=2-1]
	\arrow[""{name=0, anchor=center, inner sep=0}, "V", from=1-2, to=2-2]
	\arrow["{\ell_1}", from=1-3, to=1-4]
	\arrow[""{name=1, anchor=center, inner sep=0}, "U"', from=1-3, to=2-3]
	\arrow["{r_1}", from=1-4, to=1-5]
	\arrow["T"{description}, from=1-4, to=2-4]
	\arrow["V", from=1-5, to=2-5]
	\arrow["\alpha", shift right, shorten <=6pt, shorten >=6pt, Rightarrow, from=2-1, to=1-2]
	\arrow["G"', from=2-1, to=2-2]
	\arrow["\lambda", shift right, shorten <=6pt, shorten >=6pt, Rightarrow, from=2-3, to=1-4]
	\arrow["{\ell_2}"', from=2-3, to=2-4]
	\arrow["\rho", shift right, shorten <=6pt, shorten >=6pt, Rightarrow, from=2-4, to=1-5]
	\arrow["{r_2}"', from=2-4, to=2-5]
	\arrow["{=}"{description}, draw=none, from=0, to=1]
\end{tikzcd}\]
where $\lambda\in\mathbb{L}$ and $\rho\in\mathbb{R}$. Step 1, we chose $\ell_1,r_1,\ell_2,r_2$ using the comprehensive factorization system as in \citep{nlab:comprehensive_factorization_system} or on page 74 of \citep{KellyBook}. Step 2, we define the functor $T$ as follows: For an object $(b,[h\: Fa\to b])$ in $X_F$ , $T(b,[h]):= (Vb, [Vh\circ \alpha_a])$ in $X_G$, and for a morphism $t$ in $X_F$, $T(t) := V(t)$. Step 3, we define $\lambda$ and $\rho$. Observe that for an object $a\in A$, $\ell_2U(a) = (GUa, [id])$, and $T\ell_1(a) = T(Fa, [id]) = (VFa, [\alpha_a])$. We define $\lambda_a$ by the morphism $\alpha_a : (GUA, [\id{}]) \to (VFa,[\alpha_a])$. The natural isomorphisms $\rho_{(b,[h])}$ are defined by the identity $\id{b}\: r_2T(b,[h])=Vb \to Vr_1(b,[h]) =Vb$. 
    It is clear that $\lambda\in \mathbb{L}$, $\rho\in \mathbb{R}$, and $\alpha = \lambda\ast \rho$.
    
    Axiom (5'). Assume there are two different ways to factor a cell
    \[\begin{tikzcd}[ampersand replacement=\&]
	A \& X \& B \& A \& {X'} \& B \\
	C \& Y \& D \& C \& {Y'} \& D
	\arrow["{\ell_1}", from=1-1, to=1-2]
	\arrow["U"', from=1-1, to=2-1]
	\arrow["{r_1}", from=1-2, to=1-3]
	\arrow["W"{description}, from=1-2, to=2-2]
	\arrow[""{name=0, anchor=center, inner sep=0}, "V", from=1-3, to=2-3]
	\arrow["{\ell_1'}", from=1-4, to=1-5]
	\arrow[""{name=1, anchor=center, inner sep=0}, "U"', from=1-4, to=2-4]
	\arrow["{r_1'}", from=1-5, to=1-6]
	\arrow["{W'}"{description}, from=1-5, to=2-5]
	\arrow["V", from=1-6, to=2-6]
	\arrow["\lambda", shift right, shorten <=6pt, shorten >=6pt, Rightarrow, from=2-1, to=1-2]
	\arrow["{\ell_2}"', from=2-1, to=2-2]
	\arrow["\rho", shift right, shorten <=6pt, shorten >=6pt, Rightarrow, from=2-2, to=1-3]
	\arrow["{r_2}"', from=2-2, to=2-3]
	\arrow["{\lambda'}", shift right, shorten <=6pt, shorten >=6pt, Rightarrow, from=2-4, to=1-5]
	\arrow["{\ell_2'}"', from=2-4, to=2-5]
	\arrow["{\rho'}", shift right, shorten <=6pt, shorten >=6pt, Rightarrow, from=2-5, to=1-6]
	\arrow["{r_2'}"', from=2-5, to=2-6]
	\arrow["{=}"{description}, draw=none, from=0, to=1]
\end{tikzcd}\]

    We wish to define a horizontally invertible cell
    \[\begin{tikzcd}[ampersand replacement=\&]
	X \& {X'} \\
	Y \& {Y'}
	\arrow["{q_1}", from=1-1, to=1-2]
	\arrow["W"', from=1-1, to=2-1]
	\arrow["{W'}", from=1-2, to=2-2]
	\arrow["\theta", shift right, shorten <=6pt, shorten >=6pt, Rightarrow, from=2-1, to=1-2]
	\arrow["{q_2}"', from=2-1, to=2-2]
\end{tikzcd}\]
    such that $\theta\ast \lambda = \lambda'$ and $\rho'\ast \theta = \rho$. By the orthogonal property of an OFS, there exist isomorphisms $q_1\: X\to X'$ and $q_2\: Y\to Y'$ that make the following diagrams commute
    \[\begin{tikzcd}[ampersand replacement=\&]
	A \& X \& B \& C \& Y \& D \\
	\& {X'} \&\&\& {Y'}
	\arrow["{\ell_1}", from=1-1, to=1-2]
	\arrow["{\ell_1'}"', from=1-1, to=2-2]
	\arrow["{r_1}", from=1-2, to=1-3]
	\arrow["{q_1}"{pos=0.3}, dashed, from=1-2, to=2-2]
	\arrow["{,}"', shift right=5, draw=none, from=1-3, to=1-4]
	\arrow["{\ell_2}", from=1-4, to=1-5]
	\arrow["{\ell_2'}"', from=1-4, to=2-5]
	\arrow["{r_2}", from=1-5, to=1-6]
	\arrow["{q_2}"{pos=0.3}, dashed, from=1-5, to=2-5]
	\arrow["{r_1'}"', from=2-2, to=1-3]
	\arrow["{r_2'}"', from=2-5, to=1-6]
\end{tikzcd}\]

    Finally, we wish to find a natural isomorphisms $\theta\: q_2W\to W'q_1$. First, observe that 
    \[\begin{tikzcd}[ampersand replacement=\&]
	{r_2'q_2W=r_2W} \& {Vr_1=Vr_1'q_1} \& {r_2'W'q_1}
	\arrow["\rho", from=1-1, to=1-2]
	\arrow["{\rho'^{-1}}", from=1-2, to=1-3]
\end{tikzcd}\]

    From the fact that $r_2'$ is a discrete fibration, there is a unique cartesian lifting $\theta\: q_2W\to W'q_1$ of $\rho'^{-1}\rho$. It is a natural isomorphism because $\rho'$ and $\rho$ are natural isomorphisms and $r_2'$ is a discrete fibration. Now, by definition, $r_2'(\theta) = \rho'^{-1}\rho$, which implies that $\rho' r_2'(\theta) = \rho$ i.e. $\rho'\ast \theta = \rho$. Finally, we wish to prove that $\theta\ast \lambda = \lambda'$, i.e. $\theta q_2(\lambda) = \lambda'$. We know that $\rho r_2(\lambda) = \rho\ast \lambda =\rho'\ast \lambda' = \rho'r_2'(\lambda')$. Then, $r_2'(\theta)r_2(\lambda)=\rho'^{-1}\rho r_2(\lambda)  = r_2'(\lambda')$. Again, since $r_2'$ is a discrete fibration, there exists a unique morphism $s$ such that $r_2'(s) = r_2'(\lambda)$. We also have that $ \rho \rho'^{-1}r_2'(s) =  r_2'(\lambda')$, and by the cartesian property $\theta s = \lambda'$. Observe that $r_2'q_2(\lambda) = r_2(\lambda)$, by definition of $q_2$; then, by uniqueness of $s$,  $s=q_2(\lambda)$, which concludes the example.     
\end{example}

\begin{example}[The double category of paths in $\mathbb{Q}(\mathcal{B})$] In Example \ref{ex:PathC} we describe a DOFS in the double category of paths in a category $\mathcal{C}$; we refer the reader to that example for a clearer understanding of the notation. Here, we extend the DOFS to a more general setting. For a 2-category $\mathcal{B}$, the double category of paths $\dpath\mathbb{Q}\mathcal{B}$ in the double category of quintets $\mathbb{Q}\mathcal{B}$ has the following structure: its objects are those of $\mathcal{B}$; its horizontal arrows are the 1-morphisms in $\mathcal{B}$; its vertical arrows are paths of 1-morphisms in $\mathcal{B}$; and the cells 
\[\begin{tikzcd}[ampersand replacement=\&]
	{A_0} \& {B_0} \\
	{A_m} \& {B_n}
	\arrow[""{name=0, anchor=center, inner sep=0}, "{h_0}", from=1-1, to=1-2]
	\arrow["{\langle f_i\rangle_{0\leq i\leq m}}"', from=1-1, to=2-1]
	\arrow["{\langle g_j\rangle_{0\leq j\leq n}}", from=1-2, to=2-2]
	\arrow[""{name=1, anchor=center, inner sep=0}, "{h_m}"', from=2-1, to=2-2]
	\arrow["{\langle\alpha_i\rangle_{0\leq i\leq m}}"{description}, draw=none, from=0, to=1]
\end{tikzcd}\] 
are pairs $(\varphi,\langle\alpha_i\rangle_{0\leq i\leq m},)$, where $\varphi:[m+1]\to [n+1]$ is an order preserving function with $\varphi(0)=0$ and $\varphi(m)=n$; and for every $i\in [m+1]$, $\alpha_i$ is a 2-morphism in $\mathbb{Q}\mathcal{B}$ of the form 
\[\begin{tikzcd}[ampersand replacement=\&]
	{A_{i-1}} \& {B_{\varphi(i-1)}} \\
	{A_i} \& {B_{\varphi(i)}}
	\arrow["{h_{i-1}}", from=1-1, to=1-2]
	\arrow["{f_i}"', from=1-1, to=2-1]
	\arrow["{g^{\varphi(i-1)}_{\varphi(i)}}", from=1-2, to=2-2]
	\arrow["\alpha_i", shift right, shorten <=4pt, shorten >=4pt, Rightarrow, from=2-1, to=1-2]
	\arrow["{h_i}"', from=2-1, to=2-2]
\end{tikzcd}\]
The cells $\alpha_i$ are called \textit{fences}. Vertical composition of cells is defined by concatenation, while horizontal composition is induced by composition of 2-morphisms, following the threads of the cross arrows. 

Suppose that the 2-category $\mathcal{B}$ has a 2OFS $(L, R)$ with a choice, and consider a cell $(\varphi, \langle\alpha_i\rangle_{0\leq i\leq m})$. Then, the fences $\alpha_i$ factor as
\[\begin{tikzcd}[ampersand replacement=\&]
	{A_{i-1}} \& {B_{i-1}} \&\& {A_{i-1}} \& {\text{Im}h_{i-1}} \& {B_{i-1}} \\
	{A_i} \& {B_i} \&\& {A_i} \& {\text{Im}h_i} \& {B_i}
	\arrow["{h_{i-1}}", from=1-1, to=1-2]
	\arrow["{f_i}"', from=1-1, to=2-1]
	\arrow[""{name=0, anchor=center, inner sep=0}, "{g^{\varphi(i-1)}_{\varphi(i)}}", from=1-2, to=2-2]
	\arrow["{\ell_{h_{i-1}}}", from=1-4, to=1-5]
	\arrow[""{name=1, anchor=center, inner sep=0}, "{f_i}"', from=1-4, to=2-4]
	\arrow["{r_{h_{i-1}}}", from=1-5, to=1-6]
	\arrow["t_i"{description}, from=1-5, to=2-5]
	\arrow["{g^{\varphi(i-1)}_{\varphi(i)}}", from=1-6, to=2-6]
	\arrow["{ \alpha_i}", shift right, shorten <=4pt, shorten >=4pt, Rightarrow, from=2-1, to=1-2]
	\arrow["{h_i}"', from=2-1, to=2-2]
	\arrow["{\ell_i}", shorten <=4pt, shorten >=4pt, Rightarrow, from=2-4, to=1-5]
	\arrow["{\ell_{h_i}}"', from=2-4, to=2-5]
	\arrow["{r_i}", shorten <=4pt, shorten >=4pt, Rightarrow, from=2-5, to=1-6]
	\arrow["{r_{h_i}}"', from=2-5, to=2-6]
	\arrow["{=}"{description, pos=0.6}, draw=none, from=0, to=1]
\end{tikzcd}\]
where $r_i$ and $\ell_i$ are in the chosen 2-morphisms. This provides a factorization for the cell $(\varphi, \langle\alpha_i\rangle_{0\leq i\leq m})$. Now, we can introduce a double orthogonal factorization system $(\mathbb{L},\mathbb{R})$ in $\dpath\mathbb{Q}\mathcal{B}$. The class $\mathbb{L}$ consists of the cells $(id_{[m+1]},\langle\ell_i\rangle_{0\leq i\leq m} )$, where the horizontal morphisms in the fences $\ell_i$ are in the class $L$. Similarly, the class $\mathbb{R}$ consists of the cells $(\varphi, \langle r_i\rangle_{0\leq i \leq m})$, where the horizontal morphisms in the fences $r_i$ are in the class $R$. Observe that the fences in $\dL$ and the fences in $\dR$ are cells in $\mathbb{Q}\mathcal{B}$, and recall that a 2OFS with a choice in $\B$ corresponds to a DOFS in $\mathbb{Q}\B$. Moreover, every cell in $\dpath\mathbb{Q}\B$ can be expressed as a vertical composition of fences. These observations allow it to be derived that the pair $(\dL,\dR)$ satisfies the axioms in Proposition \ref{prop:explicitDOFS}, and therefore constitutes a DOFS in the double category $\dpath\mathbb{Q}\B$.
\end{example}

\begin{remark}
It may be tempting to compare the DOFS on a double category of quintets with the 2-dimensional OFS introduced by Stefan Milius \citep[Definition~7.3]{StefanMilius}.
However, the goals of these two types of systems are different: our systems provide a notion of factorization and hence an image for double cells, whereas the goal of Milius' work is to provide a weakened version of orthogonal factorization systems for 2-categories. It is not obvious how to construct a left and right class of 2-cells from Milius' concept. Also, the diagonal arrow goes in the wrong direction to form the image of a 2-cell.
\end{remark}

\section{Monadicity of double factorization systems}\label{sectionmonadicity}

In this section, we prove a monadicity result for double orthogonal factorization systems and for lax or colax double orthogonal factorization systems, achieving double categorical generalizations of the fundamental monadicity result of \thex\ref{theorkorostenskitholen}. To reach this, we first prove a conceptually important general monadicity result for categories internal to algebras, \thex\ref{teorgenmonadicity}. This theorem is particularly helpful in our double categorical context, but is more widely applicable.

We then merge the two flavours of DOFS together via $\F$-category theory, also known as enhanced 2-category theory (see \citep{lackshulman_enhancedtwocatlimlaxmor}). The right $\F$-categorical monad to consider to achieve our monadicity theorem is a generalization of the squaring 2-monad on $\Cat$ to the double categorical setting.

The results of this section can be thought of as giving a monadic definition of (lax or colax) double orthogonal factorization systems and showing that this definition is equivalent to the internal one we presented in Section~\ref{sectionDOFS}.

Aiming at the general monadicity theorem, we prove the following lemma. This is a generalization of \prox\ref{propfacthaspullbacks}.

\begin{lemma}\label{lemmapullbacksofalgebras}
    Let $\K$ be a 2-category and let $(T\: \K\to \K, \eta\:\Id{}\aR{}T, \mu\:T^2\aR{} T)$ be a 2-monad whose underlying 2-functor preserves pullbacks. The 2-category $\Alg+[l]{T}$ (respectively, $\Alg+[p]{T}$) of normal pseudo-algebras and lax morphisms (respectively, pseudo morphisms) between them has all pullbacks of strict morphisms along strict morphisms that exist in $\K$ (and these pullback morphisms are again strict).
\end{lemma}

\begin{proof}
    Let $(C_0,Q_0\:T(C_0)\to C_0,\alpha_{\mu,0})$, $(C_1,Q_1\:T(C_1)\to C_1,\alpha_{\mu,1})$ and $(C_2,Q_2\:T(C_2)\to C_2,\alpha_{\mu,2})$ be normal pseudo-algebras for $T$. Furthermore, let $f\:C_1\to C_0$ and $g\:C_2\to C_0$ be strict morphisms of normal pseudo-algebras, and assume that the pullback $C_1\times_{C_0}C_2$ of $f$ and $g$ exists in $\K$,
    \sq[p]{C_1\times_{C_0}C_2}{C_1}{C_2}{C_0}{\pi_1}{\pi_2}{f}{g}
    We prove that $(C_1\times_{C_0}C_2,Q_1\times_{Q_0}Q_2\:T(C_1\times_{C_0}C_2)\to C_1\times_{C_0}C_2,\alpha_{\mu,1}\times_{\alpha_{\mu,0}}\alpha_{\mu,2})$ is the pullback of $f$ and $g$ in $\Alg+[l]{T}$.

    Since $T$ preserves pullbacks, we can choose $T(C_1)\times_{T(C_0)} T(C_2)$ to be $T(C_1\times_{C_0}C_2)$. Then the pullback square of $C_1\times_{C_0}C_2$ with its image under $T$, connected by $Q_0,Q_1,Q_2$ induce a morphism $Q_1\times_{Q_0}Q_2\:T(C_1\times_{C_0}C_2)\to C_1\times_{C_0}C_2$, by the universal property of the pullback. Analogously, by the 2-dimensional universal property of the pullback, we can induce an invertible 2-cell
    \sq[l][6][10][\alpha_{\mu,1}\times_{\alpha_{\mu,0}}\alpha_{\mu,2}]{T^2(C_1\times_{C_0}C_2)}{T(C_1\times_{C_0}C_2)}{T(C_1\times_{C_0}C_2)}{C_1\times_{C_0}C_2}{T(Q_1\times_{Q_0}Q_2)}{\mu_{C_1\times_{C_0}C_2}}{Q_1\times_{Q_0}Q_2}{Q_1\times_{Q_0}Q_2}
    again using that $f$ and $g$ are strict morphisms of normal pseudo-algebras. It is straightforward to show that $(C_1\times_{C_0}C_2,Q_1\times_{Q_0}Q_2\:T(C_1\times_{C_0}C_2)\to C_1\times_{C_0}C_2,\alpha_{\mu,1}\times_{\alpha_{\mu,0}}\alpha_{\mu,2})$ is a normal pseudo-algebra for $T$. We then notice that $\pi_1$ and $\pi_2$ are strict morphisms of normal pseudo-algebras.

    Consider now $(C_3,Q_3\:T(C_3)\to C_3,\alpha_{\mu,3})$ a normal pseudo-algebras for $T$ together with lax morphisms $(h,\phi^h)\:(C_3,Q_3,\alpha_{\mu,3})\to (C_1,Q_1,\alpha_{\mu,1})$ and $(k,\phi^k)\:(C_3,Q_3,\alpha_{\mu,3})\to (C_2,Q_2,\alpha_{\mu,2})$ between normal pseudo-algebras such that $(f,\id{})\c (h,\phi^h)=(g,\id{})\c (k,\phi^k)$. Then by the universal property of the pullback, these data induce both a morphism $\langle h,k\rangle\:C_3\to C_1\times_{C_0}C_2$ and a 2-cell 
    \sq[l][6][6][\langle\phi^h,\phi^k\rangle]{T(C_3)}{T(C_1\times_{C_0}C_2)}{C_3}{C_1\times_{C_0}C_2}{T(\langle h,k\rangle)}{Q_3}{Q_1\times_{Q_0}Q_2}{\langle h,k \rangle}
    Notice that the morphism on top is indeed $T(\langle h,k\rangle)$ by the uniqueness part of the universal property of the pullback. It is straightforward to check that $(\langle h,k \rangle,\langle\phi^h,\phi^k\rangle)$ is a lax morphism of normal pseudo-algebras, and that it is the unique lax morphism that fulfills the 1-dimensional universal property of the pullback in $\Alg+[l]{T}$. Notice that if both $\phi^h$ and $\phi^k$ are iso then also $\langle\phi^h,\phi^k\rangle$ is iso. So pseudo morphisms between normal pseudo-algebras induce a pseudo morphism.

    Taking then 2-cells $\sigma\:h\aR{}h'$ and $\tau\:k\aR{}k'$ between lax morphisms of algebras such that $f\ast \sigma=g\ast \tau$, we can analogously induce a 2-cell $\langle\sigma,\tau\rangle\:\langle h,k\rangle\aR{}\langle h',k'\rangle$ between lax morphisms of algebras, which is straightforwardly shown to be the unique 2-cell that fulfills the 2-dimensional universal property of the pullback in $\Alg+[l]{T}$.
\end{proof}

We now prove a general monadicity theorem which provides a way to extend a 2-monad on a 2-category $\mathcal{K}$ to canonical 2-monads on the 2-categories of categories internal to $\mathcal{K}$ with either lax, oplax or pseudo functors, and shows how the 2-categories of algebras for the latter are isomorphic to the corresponding 2-categories of  internal categories in algebras for the former. This will be particularly useful for us in the context of double categories.

\begin{theorem}\label{teorgenmonadicity}
    Let $\K$ be a 2-category and let $(T\: \K\to \K, \eta\:\Id{}\aR{}T, \mu\:T^2\aR{} T)$ be a 2-monad whose underlying 2-functor preserves pullbacks. Then $T$ induces 2-monads
    $$\overline{T}_{\oplax}\:\Catint+[o]{\K}\to \Catint+[o]{\K}$$
    $$\overline{T}_{\ps}\:\Catint+[p]{\K}\to \Catint+[p]{\K}$$
    such that
    $$\Alg+[l]{\overline{T}_{\oplax}}\iso \Catint+'[o]{\Alg+[l]{T}}$$
    $$\Alg+[p]{\overline{T}_{\ps}}\iso \Catint+'[p]{\Alg+[p]{T}}$$
    where $\Catint+'[o]{\Alg+[l]{T}}$ (respectively, $\Catint+'[p]{\Alg+[p]{T}}$) is the full sub-2-category of $\Catint+[o]{\Alg+[l]{T}}$ (respectively, $\Catint+[p]{\Alg+[p]{T}}$) on those pseudo-categories whose source, target and identity assignment are strict morphisms between normal pseudo-algebras.
\end{theorem}
\begin{proof}
    Given $C$ a pseudo-category
    \[\begin{tikzcd}
	{C_1\times_{C_0}C_1} & {C_1} & {C_0}
	\arrow["\otimes", from=1-1, to=1-2]
	\arrow["{\operatorname{src}}", shift left=2, from=1-2, to=1-3]
	\arrow["{\operatorname{tgt}}"', shift right=2, from=1-2, to=1-3]
	\arrow["i"{description}, from=1-3, to=1-2]
    \end{tikzcd}\]
    in $\K$, we define $\overline{T}_{\oplax}(C)$ to be the pseudo-category
     \[\begin{tikzcd}
	 T({C_1\times_{C_0}C_1}) & T({C_1}) & T({C_0})
	\arrow["T(\otimes)", from=1-1, to=1-2]
	\arrow["{T(\operatorname{src})}", shift left=3, from=1-2, to=1-3]
	\arrow["{T(\operatorname{tgt})}"', shift right=3, from=1-2, to=1-3]
	\arrow["T(i)"{description}, from=1-3, to=1-2]
    \end{tikzcd}\]
    choosing ${T(C_1)\times_{T(C_0)}T(C_1)}$ to be $T({C_1\times_{C_0}C_1})$ (and similarly with iterated pullbacks), thanks to the fact that $T$ preserves pullbacks. Its associator is just $T(\Phi)$, where $\Phi$ is the associator of $C$. It is straightforward to show that this gives indeed a pseudo-category.
    
    \noindent Given then a normal internal oplax functor $(F_0\:C_0\to D_0,F_1\:C_1\to D_1,\phi)$, we send it to the normal internal oplax functor $(T(F_0),T(F_1),T(\phi))\:\overline{T}_{\oplax}(C)\to \overline{T}_{\oplax}(D)$. The required axioms are all images under $T$ of the corresponding axioms for $(F_0,F_1,\phi)$. And given an internal natural transformation $(\alpha_0,\alpha_1)$ between normal internal oplax functors, we send it to the internal natural transformation $(T(\alpha_0),T(\alpha_1))$. Again, all the required axioms are images under $T$ of the corresponding axioms for $(\alpha_0,\alpha_1)$. It is straightforward to check that $\overline{T}_{\oplax}$ is a 2-functor.

    We then define a unit $\overline{\eta}\:\Id{}\aR{}\overline{T}_{\oplax}$ and a multiplication $\overline{\mu}\:\overline{T}_{\oplax}^2\to \overline{T}_{\oplax}$ for $\overline{T}_{\oplax}$ as follows. Given $C\in \Catint+[o]{\K}$, we define $\overline{\mu}_C$ to be the normal internal strict functor $(\mu_{C_0}\:T^2(C_0)\to T(C_0),\mu_{C_1}\:T^2(C_1)\to T(C_1),\id{})$, thanks to the naturality of $\mu$. Notice that $\mu_{C_1}\times_{\mu_{C_0}}\mu_{C_1}$ coincides with $\mu_{C_1\times_{C_0}C_1}$ by the universal property of the pullback, and we can thus choose $\id{}$ as structure 2-cell
    \sq{T^2(C_1\times_{C_0}C_1)}{T^2(C_1)}{T(C_1\times_{C_0}C_1)}{T(C_1)}{T^2(\otimes)}{\mu_{C_1\times_{C_0}C_1}}{\mu_{C_1}}{T(\otimes)}
    by naturality of $\mu$, making $\overline{\mu}_C$ into an internal strict functor between internal pseudo-categories (the required axioms hold thanks to the 2-naturality of $\mu$). 2-naturality of $\mu$ then also guarantees that $\overline{\mu}$ is 2-natural.

    \noindent Given $C\in \Catint+[o]{\K}$, we then analogously define $\overline{\eta}_C$ to be the normal internal strict functor $(\eta_{C_0}\:C_0\to T(C_0),\eta_{C_1}\:C_1\to T(C_1),\id{})$. 2-Naturality of $\eta$ guarantees that $\overline{\eta}$ is well defined and 2-natural.

    It is straightforward to check that $(\overline{T}_{\oplax},\overline{\eta},\overline{\mu})$ satisfies the axioms for a 2-monad on $\Catint+[o]{\K}$, thanks to the corresponding ones for $(T,\eta,\mu)$ and the fact that both $\overline{\eta}$ and $\overline{\mu}$ have internal strict functors as components, with structure 2-cells given by identities.

    Notice then that the 2-monad $\overline{T}_{\oplax}$ on $\Catint+[o]{\K}$ restricts to a 2-monad $\overline{T}_{\ps}$ on $\Catint+[p]{\K}$. Indeed if $(F_0\:C_0\to D_0,F_1\:C_1\to D_1,\phi)$ is a normal internal pseudofunctor, which means that $\phi$ is invertible, then also $T(\phi)$ is invertible and $$(T(F_0),T(F_1),T(\phi))\:\overline{T}_{\oplax}(C)\to \overline{T}_{\oplax}(D)$$ is a normal internal pseudofunctor. Moreover, both $\overline{\eta}$ and $\overline{\mu}$ have as components internal strict functors.

    We now prove that
    $$\Alg+[l]{\overline{T}_{\oplax}}\iso \Catint+'[o]{\Alg+[l]{T}}.$$
    So consider $(C,Q\:\overline{T}_{\oplax}(C)\to C,(\alpha_{\mu,0},\alpha_{\mu,1}))$ a normal pseudo-algebra for $\overline{T}_{\oplax}$. Here, $C=(C_0,C_1,\operatorname{src},\operatorname{tgt},i,\Phi)$, with $\Phi$ the associator, and $(\alpha_{\mu,0},\alpha_{\mu,1})\:Q\circ \overline{T}_{\oplax}Q\Rightarrow Q\circ\mu_C$ is the internal (levelwise) natural transformation which gives the structure 2-cell of a normal pseudo-algebra. We write $Q=(Q_0,Q_1,\phi^Q)$, with $\phi^Q$ the structure 2-cell for the internal oplax functor $Q$. The following diagrams help visualizing the situation:
    $$
\begin{tikzcd}
	{T(C_1) \times_{T(C_0)} T(C_1)} && {T(C_1)} && {T(C_0)} \\
	\\
	{C_1 \times_{C_0} C_1} && {C_1} && {C_0}
	\arrow["{T(\otimes)}"{description}, from=1-1, to=1-3]
	\arrow["{Q_1 \times_{Q_0} Q_1}"{description}, from=1-1, to=3-1]
	\arrow["{T(\src)}"{description}, shift left=4, from=1-3, to=1-5]
	\arrow["{T(\tgt)}"{description}, shift right=4, from=1-3, to=1-5]
	\arrow["{\cong\,\,\varphi^Q}"{description}, draw=none, from=1-3, to=3-1]
	\arrow["{Q_1}"{description}, from=1-3, to=3-3]
	\arrow["{T(i)}"{description}, from=1-5, to=1-3]
	\arrow["{Q_0}"{description}, from=1-5, to=3-5]
	\arrow["\otimes"{description}, from=3-1, to=3-3]
	\arrow["\src"{description}, shift left=4, from=3-3, to=3-5]
	\arrow["\tgt"{description}, shift right=4, from=3-3, to=3-5]
	\arrow["i"{description}, from=3-5, to=3-3]
\end{tikzcd}
$$

\[\begin{tikzcd}[ampersand replacement=\&,row sep=3ex, column sep=5ex]
	{\overline{T}^2_{opl}C} \&\& {\overline{T}_{opl}C} \&\& {T^2C_1} \&\&\& {TC_1} \\
	\&\&\&\&\& {TC_1} \&\&\& {C_1} \\
	{\overline{T}_{opl}C} \&\& C \&\& {T^2C_0} \&\&\& {TC_0} \\
	\&\&\&\&\& {TC_0} \&\&\& {C_0}
	\arrow["{\mu_C}", from=1-1, to=1-3]
	\arrow[""{name=0, anchor=center, inner sep=0}, "{\overline{T}_{opl}Q}"', from=1-1, to=3-1]
	\arrow[""{name=1, anchor=center, inner sep=0}, "Q", from=1-3, to=3-3]
	\arrow["{\mu_{C_1}}", from=1-5, to=1-8]
	\arrow[""{name=2, anchor=center, inner sep=0}, "{TQ_1}"{description}, from=1-5, to=2-6]
	\arrow[shift left, from=1-5, to=3-5]
	\arrow[""{name=3, anchor=center, inner sep=0}, shift right, from=1-5, to=3-5]
	\arrow["{Q_1}", from=1-8, to=2-9]
	\arrow[shift left, from=1-8, to=3-8]
	\arrow[shift right, from=1-8, to=3-8]
	\arrow["{Q_1}"{description}, from=2-6, to=2-9]
	\arrow[shift left, from=2-6, to=4-6]
	\arrow[shift right, from=2-6, to=4-6]
	\arrow[shift right, from=2-9, to=4-9]
	\arrow[shift left, from=2-9, to=4-9]
	\arrow["Q"', from=3-1, to=3-3]
	\arrow[from=3-5, to=1-5]
	\arrow["{\mu_{C_0}}", from=3-5, to=3-8]
	\arrow["{TQ_0}"', from=3-5, to=4-6]
	\arrow[from=3-8, to=1-8]
	\arrow[""{name=4, anchor=center, inner sep=0}, "{Q_0}", from=3-8, to=4-9]
	\arrow[from=4-6, to=2-6]
	\arrow["{Q_0}"', from=4-6, to=4-9]
	\arrow[from=4-9, to=2-9]
	\arrow["{\alpha_\mu}"', shorten <=12pt, Rightarrow, from=0, to=1-3]
	\arrow["{=}"{description}, draw=none, from=1, to=3]
	\arrow["{\alpha_{\mu,1}}"{description}, shorten <=14pt, shorten >=14pt, Rightarrow, from=2, to=1-8]
	\arrow["{\alpha_{\mu,0}}"{description}, shorten >=14pt, Rightarrow, from=4-6, to=4]
\end{tikzcd}\]
    It is straightforward to show that:
    \begin{itemize}
        \item 
        The 2-cells $\alpha_{\mu,0}\: Q_0\circ TQ_0\Rightarrow Q_0\circ \mu_{C_0}$ and $\alpha_{\mu,1}\colon Q_1\circ TQ_1\Rightarrow Q_1\circ \mu_{C_1}$ give 
        $(C_0,Q_0\:T(C_0)\to C_0,\alpha_{\mu,0})$ and $(C_1,Q_1\:T(C_1)\to C_1,\alpha_{\mu,1})$ respectively the structure of  normal pseudo-algebras for $T$;
        \item 
        The commutative diagram
        \[\begin{tikzcd}
	{T(C_1)} && {T(C_0)} \\
	\\
	{C_1} && {C_0}
	\arrow["{T(\operatorname{src})}", shift left=4, from=1-1, to=1-3]
	\arrow["{T(\operatorname{tgt})}"', shift right=4, from=1-1, to=1-3]
	\arrow["{Q_1}"', from=1-1, to=3-1]
	\arrow["{T(i)}"{description}, from=1-3, to=1-1]
	\arrow["{Q_0}", from=1-3, to=3-3]
	\arrow["{\operatorname{src}}", shift left=3, from=3-1, to=3-3]
	\arrow["{\operatorname{tgt}}"', shift right=3, from=3-1, to=3-3]
	\arrow["i"{description}, from=3-3, to=3-1]
\end{tikzcd}\]
    can be viewed as establishing that $\operatorname{src},\operatorname{tgt}\:C_1\to C_0$ and $i\: C_0\to C_1$ are strict morphisms of $T$-algebras.
    \end{itemize}
 By \lemx\ref{lemmapullbacksofalgebras}, $\operatorname{src}$ and $\operatorname{tgt}$ have a pullback in $\Alg+[l]{T}$, given by $$(C_1\times_{C_0} C_1, Q_1\times_{Q_0}Q_1\:T(C_1\times_{C_0} C_1)\to C_1\times_{C_0} C_1, \alpha_{\mu,1}\times_{\alpha_{\mu,0}} \alpha_{\mu,1}).$$ And analogously, $\operatorname{src}$ and $\operatorname{tgt}$ have all iterated pullbacks. We can then use the vertical composition $\otimes$ of $C$ and $\phi^Q$ to produce a vertical composition for $(Q_0,Q_1)$. It is indeed straightforward to show that the 2-cell
\sq[l][6][6][\phi^Q]{T(C_1\times_{C_0} C_1)}{T(C_1)}{C_1\times_{C_0} C_1}{C_1}{T(\otimes)}{Q_1\times_{Q_0} Q_1}{Q_1}{\otimes}
is a lax morphism of normal pseudo-algebras for $T$, and then that
$$(Q_0,Q_1,(\operatorname{src},\id{}),(\operatorname{tgt},\id{}),(i,\id{}),(\otimes,\phi^Q),\Phi)$$
is a pseudo-category internal to $\Alg+[l]{T}$.

If we now start from a pseudo-category internal to $\Alg+[l]{T}$ with source, target and identity being strict morphisms between normal pseudo-algebras, we can produce a normal pseudo-algebra for $\overline{T}_{\oplax}$. Indeed we can first produce a pseudo-category $C$ internal to $\K$. Then the structure 2-cell $\phi^\otimes$ of the vertical composition $\otimes$ extends $Q_0$ and $Q_1$ to an internal oplax functor $Q\:\overline{T}_{\oplax} (C)\to C$, also thanks to the associator $\Phi$ being a 2-cell between algebras. Finally, $\alpha_{\mu,0}$ and $\alpha_{\mu,1}$ form an internal natural transformation thanks to the fact that $\operatorname{src},\operatorname{tgt},i$ are strict morphisms of normal pseudo-algebras. So we produce a normal pseudo-algebra for $\overline{T}_{\oplax}$. It is straightforward to see that the two constructions we have given are inverses of each other, giving a bijection on objects between $\Alg+[l]{\overline{T}_{\oplax}}$ and $\Catint+'[o]{\Alg+[l]{T}}$.

Consider now a lax morphism
$$(C,Q\:\overline{T}_{\oplax}(C)\to C,(\alpha_{\mu,0},\alpha_{\mu,1}))\to (D,R\:\overline{T}_{\oplax}(D)\to D,(\beta_{\mu,0},\beta_{\mu,1}))$$
between normal pseudo-algebras for $\overline{T}_{\oplax}$. This consist of an oplax functor $(F_0,F_1,\lambda)\:C\to D$ internal to $\K$ equipped with a structure internal natural transformation $(\xi_0,\xi_1)$ where
\sq[l][6][6][\xi_i]{T(C_i)}{T(D_i)}{C_i}{D_i}{T(F_i)}{Q_i}{R_i}{F_i}
for $i=1,2$. This gives us a lax morphism of normal pseudo-algebras 
$$(F_0,\xi_0)\: (C_0,Q_0\:T(C_0)\to C_0,\alpha_{\mu,0})\to (D_0,R_0\:T(D_0)\to D_0,\beta_{\mu,0})$$
and analogously $(F_1,\xi_1)\:Q_1\to R_1$. We can thus associate to $((F_0,F_1,\lambda),(\xi_0,\xi_1))$ the oplax functor internal to $\Alg+[l]{T}$ given by
$$((F_0,\xi_0)\:Q_0\to R_0, (F_1,\xi_1)\:Q_1\to R_1,\lambda).$$
Indeed, it is straightforward to show that
\sq[o][6][6][\lambda]{C_1\times_{C_0}C_1}{D_1\times_{D_0}D_1}{C_1}{D_1}{F_1\times_{F_0}F_1}{\otimes}{\otimes}{F_1}
is a 2-cell between lax morphisms of normal pseudo-algebras and that the structure defined above is an oplax functor internal to $\Alg+[l]{T}$. Notice that, for the latter, the required axioms are equalities of 2-cells between normal pseudo-algebras, which can then be checked in $\K$ using the underlying 2-cells. The following diagram helps visualizing the situation:
$$
{\footnotesize  
\begin{tikzcd}
	{T(C_1) \times_{T(C_0)} T(C_1)} && {T(C_1)} && {T(C_0)} \\
	& {T(D_1) \times_{T(D_0)} T(D_1)} && {T(D_1)} && {T(D_0)} \\
	{C_1 \times_{C_0} C_1} && {C_1} && {C_0} \\
	& {D_1 \times_{D_0} D_1} && {D_1} && {D_0}
	\arrow[from=1-1, to=1-3]
	\arrow[from=1-1, to=2-2]
	\arrow[from=1-1, to=3-1]
	\arrow[shift left=4, from=1-3, to=1-5]
	\arrow[shift right=4, from=1-3, to=1-5]
	\arrow["{T(\lambda)}"{description, pos=0.6}, shift right=2, shorten <=8pt, Rightarrow, from=1-3, to=2-2]
	\arrow[from=1-3, to=2-4]
	\arrow["\cong"{description, pos=0.4}, curve={height=6pt}, draw=none, from=1-3, to=3-1]
	\arrow[from=1-3, to=3-3]
	\arrow[from=1-5, to=1-3]
	\arrow[from=1-5, to=2-6]
	\arrow[from=1-5, to=3-5]
	\arrow[from=2-2, to=2-4]
	\arrow["{\xi_1\times_{\xi_0}\xi_1}", shorten <=4pt, shorten >=4pt, Rightarrow, from=2-2, to=3-1]
	\arrow[from=2-2, to=4-2]
	\arrow[shift left=1.5, from=2-4, to=2-6]
	\arrow[shift right=1.5, from=2-4, to=2-6]
	\arrow["\xi_1", shorten <=4pt, shorten >=4pt, Rightarrow, from=2-4, to=3-3]
	\arrow["\cong"{description}, shift right=3, draw=none, from=2-4, to=4-2]
	\arrow[from=2-4, to=4-4]
	\arrow[from=2-6, to=2-4]
	\arrow["\xi_0", shorten <=4pt, shorten >=2pt, Rightarrow, from=2-6, to=3-5]
	\arrow[from=2-6, to=4-6]
	\arrow[from=3-1, to=3-3]
	\arrow[from=3-1, to=4-2]
	\arrow[shift right=1.5, from=3-3, to=3-5]
	\arrow[shift left=1.5, from=3-3, to=3-5]
	\arrow["\lambda"{description, pos=0.6}, shorten <=8pt, shift right=2, Rightarrow, from=3-3, to=4-2]
	\arrow[from=3-3, to=4-4]
	\arrow[from=3-5, to=3-3]
	\arrow[from=3-5, to=4-6]
	\arrow[from=4-2, to=4-4]
	\arrow[shift left=1.5, from=4-4, to=4-6]
	\arrow[shift right=1.5, from=4-4, to=4-6]
	\arrow[from=4-6, to=4-4]
\end{tikzcd} }$$ 

Starting now from an oplax functor $((F_0,\xi_0)\:Q_0\to R_0,(F_1,\xi_1)\:Q_1\to R_1,\lambda)$ internal to $\Alg+[l]{T}$ between pseudo-categories with source, target and identity being strict morphisms between normal pseudo-algebras, we produce a lax morphism between normal pseudo-algebras for $\overline{T}_{\oplax}$. It is straightforward to show that $(F_0,F_1,\lambda)$ is an oplax functor internal to $\K$, thanks to the fact that $((F_0,\xi_0),(F_1,\xi_1),\lambda)$ is an oplax functor and $\lambda$ is a 2-cell between normal pseudo-algebras. Then $(\xi_0,\xi_1)$ is an internal natural transformation that makes 
$((F_0,F_1,\lambda),(\xi_0,\xi_1))$ into a lax morphism between normal pseudo-algebras for $\overline{T}_{\oplax}$. It is straightforward to show that the two assignments we have described are functorial and inverses of each other, giving the 1-dimensional part of 
$$\Alg+[l]{\overline{T}_{\oplax}}\iso \Catint+'[o]{\Alg+[l]{T}}.$$

We show that the 2-dimensional part holds as well. So consider a 2-cell between normal pseudo-algebras for $\overline{T}_{\oplax}$. This is an internal natural transformation $(\gamma_0\:F_0\aR{} G_0,\gamma_1\:F_1\aR{}G_1)$ that satisfies some axioms. It is straightforward to see that $\gamma_0$ is a 2-cell from the lax morphism $(F_0,\xi_0)$ of normal pseudo-algebras to the corresponding $(G_0,\xi'_0)$, and analogously for $\gamma_1$. One can then see that sending $(\gamma_0,\gamma_1)$ to $(\gamma_0,\gamma_1)$ and vice versa gives an isomorphic 2-functor, exhibiting thus the isomorphism of 2-categories 
$$\Alg+[l]{\overline{T}_{\oplax}}\iso \Catint+'[o]{\Alg+[l]{T}}.$$

Finally, it is straightforward to see that this isomorphism of 2-categories restricts to
$$\Alg+[p]{\overline{T}_{\ps}}\iso \Catint+'[p]{\Alg+[p]{T}}.$$
Indeed the structure 2-cell $\phi^Q$ of the internal oplax functor $Q$ and the structure 2-cell $\phi^\otimes$ of the lax morphism of normal pseudo-algebras which gives the vertical composition for $(Q_0,Q_1)$ correspond to each other, and if one is invertible then the other one is invertible. Analogously, for $\xi_0,\xi_1$ and $\lambda$ in the correspondence between morphisms.

So we have established the required isomorphisms of 2-categories.
\end{proof}

\begin{remark}
The intuitive idea behind the proof of \thex\ref{teorgenmonadicity} is that, interestingly, objects, morphisms, and 2-cells of the 2-categories $$\Alg+[l]{\overline{T}_{\oplax}} \quad \text{and} \quad \Catint+'[o]{\Alg+[l]{T}}$$ and $$\Alg+[p]{\overline{T}_{\ps}}\quad \text{and} \quad \Catint+'[p]{\Alg+[p]{T}}$$ respectively, are all presented by the same underlying data, grouped with each other in two different ways.
\end{remark}

We can now apply this general monadicity theorem to generalize the fundamental monadicity result of \ref{theorkorostenskitholen} to the double 
categorical setting.

\begin{construction}\label{conssquaringmonaddouble}
    We construct the needed generalization of the squaring 2-monad on $\Cat$ to the double categorical setting. We use the recipe given by \thex\ref{teorgenmonadicity}.

    The underlying 2-functor of the squaring 2-monad $\sqm\:\Cat\to \Cat$ preserves pullbacks. Indeed taking exponentials by $\to$ in $\Cat$ gives a right adjoint, that then preserves limits. By \thex\ref{teorgenmonadicity}, then $\sqm\:\Cat\to \Cat$ induces a 2-monad
    $$\sqm\loose\:\Catint+[o]{\Cat}\to \Catint+[o]{\Cat}.$$
    (Note that in the notation of \thex\ref{teorgenmonadicity} this would be called $\sqm_{\oplax}$, but for reasons that will become clear in Remark \ref{remFmonad} we choose this new notation.)
    But notice that $\Catint+[o]{\Cat}$ precisely gives the 2-category $\DblCatnoplax$ of pseudo double categories, normal (i.e.\ unitary) oplax double functors and horizontal natural transformations (i.e.\ with arrow components). So we get a 2-monad
    $$\sqm\loose\:\DblCatnoplax\to \DblCatnoplax$$
    which restricts to a 2-monad $\sqm\tight$ on the 2-category $\DblCatnps$ of pseudo double categories, normal pseudo double functors and horizontal natural transformations. By \thex\ref{teorgenmonadicity}, the 2-monad $\sqm\loose$ sends a pseudo double category $\dD$ to the pseudo double category $\dD^\to$ given by
    \[\begin{tikzcd}
	{\D_1^{\to}\times_{\D_0^{\to}}\D_1^{\to}} & {\D_1^{\to}} & {\D_0^{\to}}
	\arrow["\otimes^{\to}", from=1-1, to=1-2]
	\arrow["{\operatorname{src}^{\to}}", shift left=2, from=1-2, to=1-3]
	\arrow["{\operatorname{tgt}^{\to}}"', shift right=2, from=1-2, to=1-3]
	\arrow["i^{\to}"{description}, from=1-3, to=1-2]
    \end{tikzcd}\]
    with associators induced from $\dD$. Equivalently, $\dD^\to$ is given by the exponential (also called cotensor) of $\dD$ to the walking horizontal arrow $\to$ (i.e, the double category which has a single non-trivial horizontal arrow and only has vertical identities both as proarrows and double cells). See \citep[Section 7]{grandispare} for the explicit construction of the exponential as the pseudo double category of (normal lax) double functors, horizontal transformations as arrows, (strong) vertical transformations as proarrows and (strong) modifications as double cells. Explicitly, $\dD^{\to}$ has the following description:
    \begin{description}
        \item[an object] is a (horizontal) arrow $f\:A\to B$ in $\dD$;
        \item[an arrow $f\to g$] is a pair $(u,v)$ of arrows in $\dD$ forming a commutative square in $\D_0$
        \sq[n][5][5]{A}{A'}{B}{B'}{u}{f}{g}{v}
        \item[a proarrow $f\proarrow h$] is a double cell in $\dD$
        \sq[d][5][5][\alpha]{A}{B}{C}{D}{f}{m}{n}{h}
        \item[a double cell {\sq*[d][5][5][\Xi]{f}{g}{h}{k}{(u,v)}{\alpha}{\beta}{(y,z)}}] is a pair $\Xi=(\sigma,\tau)$ making the following cube commute (as a square in $\D_1$)
        \begin{cd}[3.75][3.05]
	\& A' \arrow[dd,proarrow,"{m'}"'{pos=0.6}]\arrow[rr,"{g}"] \&\& B' \arrow[dd,proarrow,"{n'}"'{pos=0.6}]\\[-2.9ex]
	A \arrow[dd,proarrow,"{m}"'{pos=0.4}]\arrow[ru,"{u}"]\arrow[rr,"{f}"{pos=0.6}, crossing over,crossing over clearance=1.2ex]\&\& B \arrow[ru,"{v}"'] \\
	\& C' \arrow[rruu,phantom,"{\beta}"{pos=0.35}]\arrow[rr,"{k}"'{pos=0.4}] \&\& D'\\[-2.9ex]
	C \arrow[ruuu,phantom,"{\sigma}"{pos=0.47}]\arrow[rruu,phantom,"{\alpha}"{pos=0.65}] \arrow[ru,"{y}"] \arrow[rr,"{h}"'{pos=0.4}]\&\& D \arrow[ruuu,phantom,"{\tau}"{pos=0.47}] \arrow[ru,"{z}"']
	\arrow[from=2-3,to=4-3,"{n}"{pos=0.4},crossing over,crossing over clearance=1.2ex,proarrow]
\end{cd}
i.e, $(\alpha|\tau)=(\sigma|\beta)$.
\item[identities and composition] are induced from $\dD$.
    \end{description}
The functor  $\sqm\loose$ then acts on morphisms and 2-cells in the following way:
    \begin{fun}
	   \sqm\loose & \: & \DblCatnoplax & \too & \DblCatnoplax \\[1ex]
       && \begin{cd}*[4][3]
           \dD \arrow[d,bend right=40,"{H}"',""{name=A}]\arrow[d,bend left=40,"{K}",""'{name=B}] \\
           \dE 
        \arrow[from=A,to=B,Rightarrow,"\alpha",shorten <=-0.2ex, shorten >=-0.5ex]
       \end{cd}
       & \mto &
       \begin{cd}*[4][3]
           \dD^{\to} \arrow[d,bend right=40,"{H\c -}"',""{name=A}]\arrow[d,bend left=40,"{K\c -}",""'{name=B}] \\
           \dE^{\to}
        \arrow[from=A,to=B,Rightarrow,"\alpha\ast -",shorten <=-0.2ex, shorten >=-0.5ex]
       \end{cd}
    \end{fun}
\end{construction}

\begin{remark}\label{remFmonad}
    We can merge the two 2-monads $\sqm\loose\:\DblCatnoplax\to \DblCatnoplax$ and $\sqm\tight\:\DblCatnps\to \DblCatnps$ into an $\F$-categorical monad.

    Recall from \citep{lackshulman_enhancedtwocatlimlaxmor} that $\F$-category theory coincides with category theory enriched over $\F$, which is the cartesian closed full subcategory of $\Cat^{\to}$ determined by the functors which are injective on objects and fully faithful (i.e, full embeddings). An $\F$-category $\S$ is given by a collection of objects, a hom-category ${\HomC{\S}{X}{Y}}\tight$ of tight morphisms and a second hom-category ${\HomC{\S}{X}{Y}}\loose$ of loose morphisms that give $2$-category structures (respectively) $\S\tight$ and $\S\loose$ to $\S$, together with an identity on objects, faithful and locally fully faithful $2$-functor $J_\S\:\S\tight\to\S\loose$; this is equivalent to a 2-category $\S\loose$ with a selected subclass of morphisms called tight forming a sub-2-category (with all 2-cells). An $\F$-functor $F\:\S\to \T$ is a $2$-functor $F\loose\:\S\loose\to\T\loose$ that restricts to a $2$-functor $F\tight\:\S\tight\to\T\tight$ (forming a commutative square); this is equivalent to $F\loose$ preserving tightness. And an $\F$-natural transformation is a $2$-natural transformation $\alpha\loose$ between the loose parts that restricts to one between the tight parts; this is equivalent to $\alpha\loose$ having tight components. 

    $\F$-categories, $\F$-functors and $\F$-natural transformations then form a 2-category $\VCat{\F}$. An $\F$-monad is a pseudomonoid in $\VCat{\F}$, which is equivalent to a 2-monad on the loose part whose underlying 2-functor preserves tightness and such that the multiplication and the unit have tight components. 

    Notice that an $\F$-monad $T$ on $\S$ always induces two 2-monads $T\tight$ and         $T\loose$ on $\S\tight$ and $\S\loose$ respectively, together with a 2-monad morphism $(\S\tight,T\tight)\to (\S\loose,T\loose)$. As a consequence, it also induces a 2-functor (of inclusion) $\Alg+[p]{T\tight}\to \Alg+[p]{T\loose}$ (and an analogous one with lax morphisms of pseudo-algebras), by \citep[Lemma 3.1]{gambinolobbia}.
\end{remark}

\begin{proposition}
    The two squaring $2$-monads of \conx\ref{conssquaringmonaddouble} form an $\F$-monad $\sqm$ on the $\F$-category $\DblCatnpsnoplax$ of pseudo double categories, normal pseudo double functors as tight morphisms, normal oplax double functors as loose morphisms and horizontal natural transformations as 2-cells.
\end{proposition}
\begin{proof}
By the proof of \thex\ref{teorgenmonadicity}, the 2-monad $\sqm\loose$ on the loose part $\DblCatnoplax$ of $\DblCatnpsnoplax$ restricts to the 2-functor $\sqm\tight$. Moreover the multiplication and the unit of $\sqm\loose$ have tight components. Indeed, their components are actually internal strict functors.
\end{proof}

We can now deduce the main theorem of this section, which is a monadicity result for both DOFS and lax DOFS, generalizing the fundamental result of \thex\ref{theorkorostenskitholen} to the double categorical setting. This theorem further justifies the internal definition of (lax) double orthogonal factorization system we presented in Section~\ref{sectionDOFS}.

\begin{theorem}\label{teormonadicitydouble}
    Normal pseudo-algebras for the 2-monad $\sqm\tight$ on $\DblCatnps$ are precisely the double orthogonal factorization systems.
    
    Normal pseudo-algebras for the 2-monad $\sqm\loose$ on $\DblCatnoplax$ are precisely the lax double orthogonal factorization systems.
\end{theorem}
\begin{proof}
    By \thex\ref{teorgenmonadicity} and \conx\ref{conssquaringmonaddouble}, the 2-monad $\sqm\loose\:\DblCatnoplax\to \DblCatnoplax$ is such that 
    $$\Alg+[l]{\sqm\loose}\iso \Catint+'[o]{\Alg+[l]{\sqm}}.$$
    But by \thex\ref{theorkorostenskitholen} and Notation \ref{notationFact} (thanks to \prox\ref{propmorphismsofalgebrassqmonad}), $\Alg+[l]{\sqm}$ precisely coincides with $\Factlax$. And then, by \defx\ref{def:internalDOFS}, $\Catint+'[o]{\Alg+[l]{\sqm}}$ is precisely the 2-category of lax double orthogonal factorization systems. We have thus proved that normal pseudo-algebras for the 2-monad $\sqm\loose$ give precisely lax DOFS. The restricted isomorphism of 2-categories
    $$\Alg+[p]{\sqm\tight}\iso \Catint+'[p]{\Alg+[p]{\sqm}}$$
    then shows that normal pseudo-algebras for the 2-monad $\sqm\tight$ give precisely DOFS.
    \end{proof}

\begin{remark}
    By \thex\ref{teormonadicitydouble}, we can rewrite the axioms of a lax DOFS as the axioms of a normal pseudo-algebra for the 2-monad $\sqm\loose$ on $\DblCatnoplax$. The latter is given by a double category $\dD$ equipped with an algebra structure $Q\:\dD^\to\to \dD$. $Q$ has components $Q_0\:\D_0^\to \to\D_0$ and $Q_1\:\D_1^\to\to \D_1$, expressing the orthogonal factorization systems on $\D_0$ and $\D_1$ (i.e, on horizontal arrows and double cells viewed in the horizontal direction), by \thex\ref{theorkorostenskitholen} and \remx\ref{remdetailsmonadicity}. Interestingly, the axioms of compatibility between the two orthogonal factorization systems on $\D_0$ and $\D_1$ precisely translate into making $Q$  a normal oplax functor. For example, the axiom of the oplax functor $Q$ respecting sources 
    \sq{\dD_1^\to}{\dD_0^\to}{\dD_1}{\dD_0}{\operatorname{src}^\to}{Q_1}{Q_0}{\operatorname{src}}
    captures the axiom of $\operatorname{src}$ being a strict morphism of orthogonal factorization systems.
\end{remark}

    Moreover, by essentially the same line of reasoning as in \thex\ref{teormonadicitydouble}, we can describe a lax morphism between lax DOFS's as a lax morphism between normal pseudo-algebras for the 2-monad $\sqm\loose$ on $\DblCatnoplax$. The latter, for a morphism between normal pseudo-algebras $Q^\dD \: \dD^\to_\lambda \to \dD$ and $Q^\dE \: \dE^\to_\lambda \to \dE$, is given by an oplax double functor $F \: \dD \to \dE$ equipped with a transformation $\digamma$ as in the following diagram: 
    $$
    \begin{tikzcd}
	{\mathbb{D}^\to} & {\mathbb{E}^\to} \\
	{\mathbb{D}} & {\mathbb{E}}
	\arrow["{F^\to}", from=1-1, to=1-2]
	\arrow["{Q^\mathbb{D}}"', from=1-1, to=2-1]
	\arrow["\digamma"{description}, shorten <=4pt, shorten >=4pt, Rightarrow, from=1-2, to=2-1]
	\arrow["{Q^\mathbb{E}}", from=1-2, to=2-2]
	\arrow["F"', from=2-1, to=2-2]
\end{tikzcd}
    $$ 
    that commutes appropriately with the data of $Q^\dD$ and $Q^\dE$. In particular, by \prox\ref{propmorphismsofalgebrassqmonad}, the underlying natural transformations $$\digamma_i : Q^\dE_i \circ F^\to_i \Rightarrow F_i \circ Q_i^\dD$$ for $i = 0,1$, equivalently witness that $F_i$ preserves the \emph{right class} of morphisms in the factorization system on $\D_i$ given by $Q_i^\dD$. 
    Furthermore, 2-cells between such morphisms can be given by arrow-direction 2-cells that give rise to commuting cylinders,
\[\begin{tikzcd}
	{\mathbb{D}^\rightarrow} && {\mathbb{E}^\rightarrow} && {\mathbb{D}^\rightarrow} && {\mathbb{E}^\rightarrow} \\
	{\mathbb{D}} && {\mathbb{E}} && {\mathbb{D}} && {\mathbb{E}}
	\arrow["{F^\rightarrow}", from=1-1, to=1-3]
	\arrow[""{name=0, anchor=center, inner sep=0}, "{Q^{\mathbb{D}}}"', from=1-1, to=2-1]
	\arrow[""{name=1, anchor=center, inner sep=0}, "{Q^{\mathbb{E}}}", from=1-3, to=2-3]
	\arrow[""{name=2, anchor=center, inner sep=0}, "{G^{\rightarrow}}"', from=1-5, to=1-7]
	\arrow[""{name=3, anchor=center, inner sep=0}, "{F^\rightarrow}", curve={height=-18pt}, from=1-5, to=1-7]
	\arrow[""{name=4, anchor=center, inner sep=0}, "{Q^{\mathbb{D}}}"', from=1-5, to=2-5]
	\arrow[""{name=5, anchor=center, inner sep=0}, "{Q^{\mathbb{E}}}", from=1-7, to=2-7]
	\arrow[""{name=6, anchor=center, inner sep=0}, "F", from=2-1, to=2-3]
	\arrow[""{name=7, anchor=center, inner sep=0}, "G"', curve={height=18pt}, from=2-1, to=2-3]
	\arrow["G"', from=2-5, to=2-7]
	\arrow["\digamma"{description}, between={0.2}{0.8}, Rightarrow, from=1, to=0]
	\arrow["{=}"{description}, draw=none, from=1, to=4]
	\arrow["{\alpha^\rightarrow}", between={0.2}{0.8}, Rightarrow, from=3, to=2]
	\arrow["\gamma"{description}, between={0.2}{0.8}, Rightarrow, from=5, to=4]
	\arrow["\alpha", between={0.2}{0.8}, Rightarrow, from=6, to=7]
\end{tikzcd}\]
    i.e., pairs of compatible 2-cells between lax morphisms of categories with orthogonal factorization systems.
    So we can conclude: 

\begin{proposition}
    A lax morphism between lax DOFS's is an oplax double functor that preserves the right class -- both of horizontal morphisms, and of double cells. Similarly, a (pseudo) morphism between DOFS's is a (pseudo) double functor that preserves both left and right classes of horizontal arrows and double cells. 2-Cells between these morphisms correspond to arrow direction (horizontal) transformations between the double functors that are compatible with the algebra structure on both the arrows and the double cells.
\end{proposition}

\begin{remark}
    The inclusion of algebras $\Alg+[l]{\sqm\tight}\to \Alg+[l]{\sqm\loose}$ described in \remx\ref{remFmonad}, guaranteed from having an $\F$-monad $\sqm$ on $\DblCatnpsnoplax$, is the inclusion of DOFS's into lax DOFS's.
\end{remark}

\begin{remark}
    The whole theory of this section can be dualized for colax DOFS's. This way we obtain normal pseudo-algebras for the 2-monad $(-)^\to_{\lax}$ on $\DblCatnlax$ are precisely the colax double orthogonal factorization systems. A lax morphism between colax DOFS's is a lax double functor that preserves the left class -- both of horizontal morphisms, and of double cells.
\end{remark}

\begin{remark}\label{morearrows}
We see here that the treatment of double orthogonal factorization systems as algebras for a monad gets us various 2-categories of double orthogonal factorization systems on double categories (one for each flavour: oplax, lax, pseudo, strict, where obviously the last two are sub-2-categories of the former). However, keeping the algebra functors within the larger context of all double categories with various flavours of double functors gives us further options for morphisms and 2-cells between them. In \citep{grandispare-adjoint} there is a double category with lax functors in one direction and oplax functors in the other direction. The cells for this double category would give rise to a notion of morphism between double categories with a lax orthogonal factorization systems that preserves the left class (where the vertical composition preserves the right class). In this paper we will not explore this further. 

Another type of morphism that is afforded us by this approach is to go from a double category with a lax DOFS to one with an oplax DOFS (or vice versa). In this case the morphism $F$ needs to be a pseudofunctor. We will have use for these morphisms in Section \ref{appsandex} when we want to compare different DOFS with a common OFS on the arrows of the double category.
\end{remark}

\begin{remark}
    Double categories also give us the option to turn the 2-category of double categories with a lax DOFS into a category enriched in double categories. We do this by adding vertical transformations to the mix; i.e., transformations in the proarrow direction. Let $\mathbb{D}$ and $\mathbb{E}$ be double categories with a lax DOFS and $F,G\colon \mathbb{D}\to\mathbb{E}$ be oplax double functors which also preserve the right classes.
    A vertical transformation $\theta\colon F\Rightarrow\negthinspace\negthinspace\negthinspace\negthinspace\negthinspace\negthinspace|\,\,\,\, G$ would correspond to a functor $\theta\colon \D_0\to\E_1$ satisfying the usual properties. In this case, we only need to add the requirement that $\theta$ be a lax morphism of algebras; i.e., that it preserves the right class. We can now get the double category structure on the homs by taking the full sub double category on these horizontal and vertical transformations: simply, take all modifications that fit.

    There is even more possible: we can combine all the algebras for the arrow monad, the lax, the colax, the pseudo and the strict, into a single 2-category with the use of 2-cells and arrows as indicated in \citep{grandispare-adjoint}. We will not need this result in full. We will only use it in the next section for DOFS extending a common OFS on the arrows of a double category.
\end{remark}

\section{Extending an arrow OFS to a DOFS}\label{appsandex}
In this section we present some results around the question when a given OFS on the arrows of a double category can be extended to a DOFS and how these extensions are related to each other. We will illustrate this by extending our examples from Section \ref{sectionDOFS} and adding some new ones.

\subsection{Theoretical results}
We present further examples of double orthogonal factorization systems. 
Every double orthogonal factorization system $(\dL,\dR)$ on $\dD$ has an underlying factorization system $(\cL_0,\cR_0)$ on its category of arrows $\D_0$. In this section we will consider the extensions of various well-known orthogonal factorization systems to double factorization systems and we will explore the relationships between these extensions.

For our first general result we want to address the requirement that the source, target and identity structure arrows, $\src$, $\tgt$ and $i$, of $\dD$ need to be strict morphisms of categories with an orthogonal factorization system. We give a condition under which a pair of orthogonal factorization systems for which these structure arrows only preserve the left and the right classes can be adjusted to a pair for which these maps are strict; i.e, preserve the chosen factorizations. 

\begin{proposition}\label{propisoextension}
    Let  $\dD$ be a double category such that both $\D_0$ and $\D_1$ are categories with an OFS and chosen factorizations such that the source, target and unit functors preserve the left and right classes, but not necessarily the chosen factorizations. Then the factorizations of the cells in $\dD$  can be adjusted to obtain a DOFS for $\dD$ if for every diagram of invertible arrows and a proarrow on the left there is a horizontally invertible cell $\theta$ as on the right:
\[\begin{tikzcd}[ampersand replacement=\&]
	\bullet \&\& \bullet \&\& \bullet \&\& \bullet \\
	\bullet \&\& \bullet \&\& \bullet \&\& \bullet
	\arrow["\cong", from=1-1, to=1-3]
	\arrow["u"', "\shortmid"{marking}, from=1-1, to=2-1]
	\arrow["\cong", from=1-5, to=1-7]
	\arrow[""{name=0, anchor=center, inner sep=0}, "u"', "\shortmid"{marking}, from=1-5, to=2-5]
	\arrow[""{name=1, anchor=center, inner sep=0}, "v", "\shortmid"{marking}, from=1-7, to=2-7]
	\arrow["\cong"', from=2-1, to=2-3]
	\arrow["\cong"', from=2-5, to=2-7]
	\arrow["\theta"{description}, draw=none, from=0, to=1]
\end{tikzcd}\]
\end{proposition}

\begin{proof}
Given a choice of factorizations for the arrows, $f=r_f\circ\ell_f$ where $\ell_f\in\cL$ and $r_f\in\cR$, we will show how to adjust the factorization for the double cells so that all of $\src$, $\tgt$ and $i$ preserve the chosen factorizations. 

We make the units functor strict by choosing the corresponding factorizations for the unit cells: $\ell_{i_f}=i_{\ell_f}$ and $r_{i_f}=i_{r_f}$.

  To adjust the factorizations of the other cells so that source and target become strict, let 

\[\begin{tikzcd}[ampersand replacement=\&]
	A \&\& C \& A \&\& {A'} \&\& C \\
	B \&\& D \& B \&\& {B'} \&\& D
	\arrow["f", from=1-1, to=1-3]
	\arrow[""{name=0, anchor=center, inner sep=0}, "u"', "\shortmid"{marking}, from=1-1, to=2-1]
	\arrow[""{name=1, anchor=center, inner sep=0}, "v", "\shortmid"{marking}, from=1-3, to=2-3]
	\arrow["{\src(\ell_\alpha)}", two heads, from=1-4, to=1-6]
	\arrow[""{name=2, anchor=center, inner sep=0}, "u"', "\shortmid"{marking}, from=1-4, to=2-4]
	\arrow["{\src(r_\alpha)}", from=1-6, to=1-8]
	\arrow[""{name=3, anchor=center, inner sep=0}, "{Im(\alpha)}", "\shortmid"{marking}, from=1-6, to=2-6]
	\arrow[""{name=4, anchor=center, inner sep=0}, "v", "\shortmid"{marking}, from=1-8, to=2-8]
	\arrow["g"', from=2-1, to=2-3]
	\arrow["{\tgt(\ell_\alpha)}"', two heads, from=2-4, to=2-6]
	\arrow["{\tgt(r_\alpha)}"', from=2-6, to=2-8]
	\arrow["\alpha"{description}, draw=none, from=0, to=1]
	\arrow["{=}"{description}, draw=none, from=1, to=2]
	\arrow["{\ell_\alpha}"{description}, draw=none, from=2, to=3]
	\arrow["{r_\alpha}"{description}, draw=none, from=3, to=4]
\end{tikzcd}\]
be the chosen factorization for a cell $\alpha$.
By orthogonality of the factorization system on $D_0$, there are unique isomorphisms $d_1$ and $d_2$ making the following diagrams commute:

\[\begin{tikzcd}[ampersand replacement=\&]
	\& {A'} \&\&\& {B'} \\
	A \&\& C \& B \&\& D \\
	\& {\im(f)} \&\&\& {\im(g)}
	\arrow["{\src(r_\alpha)}", hook, from=1-2, to=2-3]
	\arrow["{d_1}"{description}, dashed, from=1-2, to=3-2]
	\arrow["{\tgt(r_\alpha)}", hook, from=1-5, to=2-6]
	\arrow["{d_2}"{description}, dashed, from=1-5, to=3-5]
	\arrow["{\src(\ell_\alpha)}", two heads, from=2-1, to=1-2]
	\arrow["{\ell_f}"', two heads, from=2-1, to=3-2]
	\arrow["{\tgt(\ell_\alpha)}", two heads, from=2-4, to=1-5]
	\arrow["{\ell_g}"', two heads, from=2-4, to=3-5]
	\arrow["{r_f}"', hook, from=3-2, to=2-3]
	\arrow["{r_g}"', hook, from=3-5, to=2-6]
\end{tikzcd}\]
Now choose a horizontally invertible cell $\theta$ according to the condition of the proposition:

\[\begin{tikzcd}[ampersand replacement=\&]
	{A'} \&\& {\im(f)} \\
	{B'} \&\& {\im(g)}
	\arrow["{d_1}", from=1-1, to=1-3]
	\arrow[""{name=0, anchor=center, inner sep=0}, "{\im(\alpha)}"', "\shortmid"{marking}, from=1-1, to=2-1]
	\arrow[""{name=1, anchor=center, inner sep=0}, "{\im(\alpha)'}", "\shortmid"{marking}, from=1-3, to=2-3]
	\arrow["{d_2}"', from=2-1, to=2-3]
	\arrow["\theta"{description}, draw=none, from=0, to=1]
\end{tikzcd}\]
Now let $\ell'_\alpha=\theta\circ\ell_\alpha$ and $r'_\alpha=r_\alpha\circ \theta^{-1}$, where $\theta^{-1}$ is the horizontal inverse of $\theta$.
Since $\theta$ is horizontally invertible, $\ell'_\alpha$ is in the left class on $\D_1$ and $r'_\alpha$ is in the right class on $\D_1$ and with this new choice of factorizations, the source and target functors have become strict.
\end{proof}

In the following theorem we discuss properties on a double category that assure us that there is at least one double orthogonal factorization system that extends a given factorization system on its arrows. We will make use of the notions of extension (or, cartesian cell) and restriction (or, opcartesian cell) in a double category \citep{shulman}.

\begin{definition}
\begin{enumerate}
    \item 
   A double category has an {\em extension}, or {\em opcartesian cell}, for the diagram on the left below if there is a cell $\gamma_{s,s'}^M$ as on the right,

\[\begin{tikzcd}
	A & B && A & B \\
	{A'} & {B'} && {A'} & {B'}
	\arrow["s", from=1-1, to=1-2]
	\arrow["M"'{inner sep=.8ex}, "\shortmid"{marking}, from=1-1, to=2-1]
	\arrow["s", from=1-4, to=1-5]
	\arrow[""{name=0, anchor=center, inner sep=0}, "M"'{inner sep=.8ex}, "\shortmid"{marking}, from=1-4, to=2-4]
	\arrow[""{name=1, anchor=center, inner sep=0}, "{\operatorname{ext}_{s,s'}(M)}"{inner sep=.8ex}, "\shortmid"{marking}, from=1-5, to=2-5]
	\arrow["{s'}"', from=2-1, to=2-2]
	\arrow["{s'}"', from=2-4, to=2-5]
	\arrow["{\gamma_{s,s'}^M}"{description}, draw=none, from=0, to=1]
\end{tikzcd}\]
with the following universal property:
for each double cell 

\[\begin{tikzcd}
	A & B & C \\
	{A'} & {B'} & {C'}
	\arrow["s", from=1-1, to=1-2]
	\arrow[""{name=0, anchor=center, inner sep=0}, "M"'{inner sep=.8ex}, "\shortmid"{marking}, from=1-1, to=2-1]
	\arrow["h", from=1-2, to=1-3]
	\arrow[""{name=1, anchor=center, inner sep=0}, "N"{inner sep=.8ex}, "\shortmid"{marking}, from=1-3, to=2-3]
	\arrow["{s'}"', from=2-1, to=2-2]
	\arrow["{h'}"', from=2-2, to=2-3]
	\arrow["\theta"{marking, allow upside down}, draw=none, from=0, to=1]
\end{tikzcd}\]
there is a unique cell $\theta'$ such that

\[\begin{tikzcd}
	A & B & C && A & B && C \\
	{A'} & {B'} & {C'} && {A'} & {B'} && {C'}
	\arrow["s", from=1-1, to=1-2]
	\arrow[""{name=0, anchor=center, inner sep=0}, "M"'{inner sep=.8ex}, "\shortmid"{marking}, from=1-1, to=2-1]
	\arrow["h", from=1-2, to=1-3]
	\arrow[""{name=1, anchor=center, inner sep=0}, "N"{inner sep=.8ex}, "\shortmid"{marking}, from=1-3, to=2-3]
	\arrow["s", from=1-5, to=1-6]
	\arrow[""{name=2, anchor=center, inner sep=0}, "M"'{inner sep=.8ex}, "\shortmid"{marking}, from=1-5, to=2-5]
	\arrow["h", from=1-6, to=1-8]
	\arrow[""{name=3, anchor=center, inner sep=0}, "{\operatorname{ext}_{s,s'}(M)}"{inner sep=.8ex}, "\shortmid"{marking}, from=1-6, to=2-6]
	\arrow[""{name=4, anchor=center, inner sep=0}, "N"{inner sep=.8ex}, "\shortmid"{marking}, from=1-8, to=2-8]
	\arrow["{s'}"', from=2-1, to=2-2]
	\arrow["{h'}"', from=2-2, to=2-3]
	\arrow["{s'}"', from=2-5, to=2-6]
	\arrow["{h'}"', from=2-6, to=2-8]
	\arrow["\theta"{marking, allow upside down}, draw=none, from=0, to=1]
	\arrow["{=}"{marking, allow upside down}, draw=none, from=1, to=2]
	\arrow["{\gamma_{s,s'}^M}"{marking, allow upside down}, draw=none, from=2, to=3]
	\arrow["{\theta'}"{marking, allow upside down, pos=0.7}, draw=none, from=3, to=4]
\end{tikzcd}\]
\item The definition of {\em restriction}, or {\em cartesian cell},  is dual to that of extension.
\item A double category has all extensions along arrows in a class $\mathcal{S}$ if it has all extensions for diagrams where the arrows are in $\mathcal{S}$.
\item Dually, a double category has all restrictions along arrows in a class $\mathcal{S}$ if it has all restrictions for diagrams where the arrows are in $\mathcal{S}$.
   \end{enumerate}
\end{definition}

\begin{proposition}\label{prop:ofsExtensions}
    Let $\dD$ be a double category with an orthogonal factorization system $(\cL,\cR)$ on $\D_0$.
    \begin{enumerate}
        \item If $\dD$ has all extensions along the arrows in $\cL$, it has a lax DOFS $(\dL,\dR)=(\cL,\cR)_{\operatorname{extn}}$ with $(\cL_0,\cR_0)=(\cL,\cR)$, $\cL_1$ consists of all extension cells with horizontal boundaries in $\cL$, and $\cR_1$ consists of all double cells with horizontal boundaries in $\cR$;
        \item If $\dD$ has all restrictions along arrows in $\cR$, it has a colax DOFS $(\dL,\dR)=(\cL,\cR)_{\operatorname{restr}}$ with $(\cL_0,\cR_0)=(\cL,\cR)$, $\cL_1$ has all double cells with horizontal boundaries in $\cL$, and $\cR_1$ has all restriction cells with horizontal boundaries in $\cR$;
    \end{enumerate} 
\end{proposition}

\begin{proof}
\begin{enumerate}
\item 
It is clear that both the class of extension cells with $\cL$-boundaries and the class of cells with $\cR$-boundaries are closed under composition and contain all horizontally invertible cells. 

To factor an arbitrary double cell,
factor its horizontal boundary arrows in $(\cL,\cR)$, and take the extension of its vertical source arrow $M$ along the arrows in $\cL$. Then the right side of the factorization is the cell $\theta'$ that exists by the universal property of the extension cell.

\[\begin{tikzcd}
	\bullet & \bullet && \bullet & \bullet && \bullet \\
	\bullet & \bullet && \bullet & \bullet && \bullet
	\arrow["f", from=1-1, to=1-2]
	\arrow[""{name=0, anchor=center, inner sep=0}, "M"'{inner sep=.8ex}, "\shortmid"{marking}, from=1-1, to=2-1]
	\arrow[""{name=1, anchor=center, inner sep=0}, "N"{inner sep=.8ex}, "\shortmid"{marking}, from=1-2, to=2-2]
	\arrow["{\ell_f}", from=1-4, to=1-5]
	\arrow[""{name=2, anchor=center, inner sep=0}, "M"'{inner sep=.8ex}, "\shortmid"{marking}, from=1-4, to=2-4]
	\arrow["{r_f}", from=1-5, to=1-7]
	\arrow[""{name=3, anchor=center, inner sep=0}, "{\operatorname{ext}_{\ell_f,\ell_g}^M}"{inner sep=.8ex}, "\shortmid"{marking}, from=1-5, to=2-5]
	\arrow[""{name=4, anchor=center, inner sep=0}, "\shortmid"{marking}, from=1-7, to=2-7]
	\arrow["g"', from=2-1, to=2-2]
	\arrow["{\ell_g}"', from=2-4, to=2-5]
	\arrow["{r_g}"', from=2-5, to=2-7]
	\arrow["\theta"{description}, draw=none, from=0, to=1]
	\arrow["{=}"{description}, draw=none, from=2, to=1]
	\arrow["{\gamma_{\ell_f,\ell_g}^M}"{description}, draw=none, from=2, to=3]
	\arrow["{\theta'}"{description, pos=0.6}, draw=none, from=3, to=4]
\end{tikzcd}\]

Finally, to prove the lifting property, we will denote a double cell $\theta$ as in the previous diagram by $M\xrightarrow[]{(f,\theta,g)}N$.
With this notation, we are considering a commutative square of the form,

\[\begin{tikzcd}
	M && N \\
	E && P
	\arrow["{(h,\theta,h')}", from=1-1, to=1-3]
	\arrow["{(\ell,\gamma,\ell')}"', from=1-1, to=2-1]
	\arrow["{(r,\delta,r')}", from=1-3, to=2-3]
	\arrow["{(k,\tau,k')}"', from=2-1, to=2-3]
\end{tikzcd}\]
where $\gamma$ is an extension cell, $\ell,\ell'\in\cL$ and $r,r'\in\cR$.
This gives us the following two diagrams in the arrow category $\D_0$,

\[\begin{tikzcd}
	\bullet & \bullet & \bullet & \bullet \\
	\bullet & \bullet & \bullet & \bullet
	\arrow["h", from=1-1, to=1-2]
	\arrow["\ell"', from=1-1, to=2-1]
	\arrow["r", from=1-2, to=2-2]
	\arrow["{h'}", from=1-3, to=1-4]
	\arrow["{\ell'}"', from=1-3, to=2-3]
	\arrow["{r'}", from=1-4, to=2-4]
	\arrow["d"{description}, dashed, from=2-1, to=1-2]
	\arrow["k"', from=2-1, to=2-2]
	\arrow["{d'}"{description}, dashed, from=2-3, to=1-4]
	\arrow["{k'}"', from=2-3, to=2-4]
\end{tikzcd}\]
which have unique diagonals $d$ and $d'$ because $(\cL,\cR)$ is an OFS. Hence, $h=d\circ\ell$ and $h'=d'\circ\ell'$.
The universal property of $\gamma$ as extension cell gives rise to a unique double cell $\theta'$ as in the following diagram,

\[\begin{tikzcd}
	\bullet & \bullet & \bullet && \bullet & \bullet & \bullet \\
	\bullet & \bullet & \bullet && \bullet & \bullet & \bullet
	\arrow["\ell", from=1-1, to=1-2]
	\arrow[""{name=0, anchor=center, inner sep=0}, "M"'{inner sep=.8ex}, "\shortmid"{marking}, from=1-1, to=2-1]
	\arrow["d", from=1-2, to=1-3]
	\arrow[""{name=1, anchor=center, inner sep=0}, "N"{inner sep=.8ex}, "\shortmid"{marking}, from=1-3, to=2-3]
	\arrow["\ell", from=1-5, to=1-6]
	\arrow[""{name=2, anchor=center, inner sep=0}, "M"'{inner sep=.8ex}, "\shortmid"{marking}, from=1-5, to=2-5]
	\arrow["d", from=1-6, to=1-7]
	\arrow[""{name=3, anchor=center, inner sep=0}, "E"{inner sep=.8ex}, "\shortmid"{marking}, from=1-6, to=2-6]
	\arrow[""{name=4, anchor=center, inner sep=0}, "N"{inner sep=.8ex}, "\shortmid"{marking}, from=1-7, to=2-7]
	\arrow["{\ell'}"', from=2-1, to=2-2]
	\arrow["{d'}"', from=2-2, to=2-3]
	\arrow["{\ell'}"', from=2-5, to=2-6]
	\arrow["{d'}"', from=2-6, to=2-7]
	\arrow["\theta"{description}, draw=none, from=0, to=1]
	\arrow["{=}"{description}, draw=none, from=1, to=2]
	\arrow["\gamma"{description}, draw=none, from=2, to=3]
	\arrow["{\theta'}"{description}, draw=none, from=3, to=4]
\end{tikzcd}\]
This gives the commutativity of the upper triangle in the following diagram,

\[\begin{tikzcd}
	M && N \\
	E && P
	\arrow["{(h,\theta,h')}", from=1-1, to=1-3]
	\arrow["{(\ell,\gamma,\ell')}"', from=1-1, to=2-1]
	\arrow["{(r,\delta,r')}", from=1-3, to=2-3]
	\arrow["{(d,\theta',d')}"{description}, from=2-1, to=1-3]
	\arrow["{(k,\tau,k')}"', from=2-1, to=2-3]
\end{tikzcd}\]
For the lower triangle, note that
\begin{eqnarray*}
((r,\delta,r')\circ(d,\theta',d')) \circ(\ell,\gamma,\ell')
    &=&(r,\delta,r')\circ((d,\theta',d')\circ(\ell,\gamma,\ell'))\\
    &=&(r,\delta,r')\circ(h\theta,h')\\
    &=&(k,\tau,k')\circ(\ell,\gamma,\ell')
\end{eqnarray*}
The uniqueness in the universal property of $\gamma$ now gives us that $(r,\delta,r')\circ(d,\theta',d')=(k,\tau,k')$. So $(d,\theta',d')$ is the required lifting. Its uniqueness follows again from the universal property of 
$\gamma$.

Finally, a choice of an extension cell for each
diagram,

\[\begin{tikzcd}
	\bullet & \bullet \\
	\bullet & \bullet
	\arrow["\ell", from=1-1, to=1-2]
	\arrow["M"'{inner sep=.8ex}, "\shortmid"{marking}, from=1-1, to=2-1]
	\arrow["{\ell'}"', from=2-1, to=2-2]
\end{tikzcd}\]
gives rise to a choice of factorizations of the double cells such that their sources and targets are the chosen factorizations on the arrows. So we have a DOFS.
\item This follows from the first part by duality.
\end{enumerate}
\end{proof}

Recall that each DOFS on $\dD$ gives rise to a suitable (i.e., normal oplax, lax or pseudo, depending on the type of factorization system) double functor from $\dD^{\rightarrow}$ to $\dD$. 
Now a morphism between two double orthogonal factorization systems on the same double category with the same factorization system on the arrow category corresponds to a diagram
 $$
    \begin{tikzcd}
	{\mathbb{D}^\to} & {\mathbb{D}^\to} \\
	{\mathbb{D}} & {\mathbb{D}}
	\arrow["{F^\to}", from=1-1, to=1-2]
	\arrow["{Q^\mathbb{D}}"', from=1-1, to=2-1]
	\arrow["\digamma"{description}, shorten <=4pt, shorten >=4pt, Rightarrow, from=1-2, to=2-1]
	\arrow["{(Q^\mathbb{D})'}", from=1-2, to=2-2]
	\arrow["F"', from=2-1, to=2-2]
\end{tikzcd}
    $$ 
    where $F^\to$ is the identity on $\mathbb{D}^\to$. Hence, $F$ is the identity on $\mathbb{D}$. So a morphism between such double orthogonal factorization systems is a horizontal identity on objects transformation between the double functors representing the two factorization systems. Note that \citep{grandispare-adjoint}
gives a detailed description of horizontal transformations, where the domain and the codomain may be double functors with distinct types of laxity.

\begin{thm}\label{thm:comparaisonCartDOFS}
  \begin{enumerate}
      \item 
      Let $\dD$ be a double category with an OFS $(\cL_0,\cR_0)$ on its category of arrows and all extensions along the left class $\cL_0$. Then for any lax $(\dL,\dR)$ on $\dD$ with $(\cL_0,\cR_0)$ as the arrow part, there is a unique morphism $(\dL,\dR)\to(\cL_0,\cR_0)_{\operatorname{extn}}$ of DOFS on $\dD$.
      \item 
      Let $\dD$ be a double category with an OFS $(\cL_0,\cR_0)$ on its category of arrows and all restrictions along the right class $\cR_0$. Then for any $(\dL,\dR)$ on $\dD$ with $(\cL_0,\cR_0)$ as the arrow part, 
      there is a unique morphism $(\cL_0,\cR_0)_{\operatorname{restr}}\to(\dL,\dR)$ of DOFS on $\dD$.      
  \end{enumerate}  
\end{thm}

\begin{proof}
\begin{enumerate}
    \item Let $(\dL,\dR)$ be a DOFS on $\dD$, with corresponding oplax double functor $Q\colon \dD^\to\to\dD$. For a double cell $\theta$ in $\dD$ we obtain two factorizations with the same factorizations of the boundary horizontal cells:
    
\[\begin{tikzcd}
	\bullet & \bullet && \bullet && {Q(f)} && \bullet \\
	\bullet & \bullet && \bullet && {Q(g)} && \bullet \\
	\bullet & \bullet && \bullet && {Q(f)} && \bullet \\
	\bullet & \bullet && \bullet && {Q(g)} && \bullet
	\arrow["f", from=1-1, to=1-2]
	\arrow[""{name=0, anchor=center, inner sep=0}, "M"'{inner sep=.8ex}, "\shortmid"{marking}, from=1-1, to=2-1]
	\arrow[""{name=1, anchor=center, inner sep=0}, "N"{inner sep=.8ex}, "\shortmid"{marking}, from=1-2, to=2-2]
	\arrow["{\ell_f}", from=1-4, to=1-6]
	\arrow[""{name=2, anchor=center, inner sep=0}, "M"'{inner sep=.8ex}, "\shortmid"{marking}, from=1-4, to=2-4]
	\arrow["{r_f}", from=1-6, to=1-8]
	\arrow[""{name=3, anchor=center, inner sep=0}, "{\operatorname{ext}_{\ell_f,\ell_g}^M}"{inner sep=.8ex}, "\shortmid"{marking}, from=1-6, to=2-6]
	\arrow[""{name=4, anchor=center, inner sep=0}, "N"{inner sep=.8ex}, "\shortmid"{marking}, from=1-8, to=2-8]
	\arrow["g"', from=2-1, to=2-2]
	\arrow["{\ell_g}"', from=2-4, to=2-6]
	\arrow["{r_g}", from=2-6, to=2-8]
	\arrow["f", from=3-1, to=3-2]
	\arrow[""{name=5, anchor=center, inner sep=0}, "M"'{inner sep=.8ex}, "\shortmid"{marking}, from=3-1, to=4-1]
	\arrow[""{name=6, anchor=center, inner sep=0}, "N"{inner sep=.8ex}, "\shortmid"{marking}, from=3-2, to=4-2]
	\arrow["{\ell_f}", from=3-4, to=3-6]
	\arrow[""{name=7, anchor=center, inner sep=0}, "M"'{inner sep=.8ex}, "\shortmid"{marking}, from=3-4, to=4-4]
	\arrow["{r_f}", from=3-6, to=3-8]
	\arrow[""{name=8, anchor=center, inner sep=0}, "{Q(\theta)}"{inner sep=.8ex}, "\shortmid"{marking}, from=3-6, to=4-6]
	\arrow[""{name=9, anchor=center, inner sep=0}, "N"{inner sep=.8ex}, "\shortmid"{marking}, from=3-8, to=4-8]
	\arrow["g"', from=4-1, to=4-2]
	\arrow["{\ell_g}"', from=4-4, to=4-6]
	\arrow["{r_g}"', from=4-6, to=4-8]
	\arrow["\theta"{description}, draw=none, from=0, to=1]
	\arrow["{=}"{description}, draw=none, from=1, to=2]
	\arrow["{\gamma^M_{\ell_f,\ell_g}}"{description}, draw=none, from=2, to=3]
	\arrow["{\theta'}"{description, pos=0.7}, draw=none, from=3, to=4]
	\arrow["\theta"{description}, draw=none, from=5, to=6]
	\arrow["{=}"{description}, draw=none, from=7, to=6]
	\arrow["{\ell_\theta}"{description, pos=0.4}, draw=none, from=7, to=8]
	\arrow["{r_{\theta}}"{description}, draw=none, from=8, to=9]
\end{tikzcd}\]
    By the universal property of the extension cell $\gamma^M_{\ell_f,\ell_g}$ there is a unique globular double cell 
   
\[\begin{tikzcd}
	{Q(f)} && {Q(f)} \\
	{Q(g)} && {Q(g)}
	\arrow["{\operatorname{id}_{Q(f)}}", from=1-1, to=1-3]
	\arrow[""{name=0, anchor=center, inner sep=0}, "{\operatorname{ext}_{\ell_f,\ell_g}^M}"'{inner sep=.8ex}, "\shortmid"{marking}, from=1-1, to=2-1]
	\arrow[""{name=1, anchor=center, inner sep=0}, "{Q(\theta)}"{inner sep=.8ex}, "\shortmid"{marking}, from=1-3, to=2-3]
	\arrow["{\operatorname{id}_{Q(g)}}"', from=2-1, to=2-3]
	\arrow["{\ell'_\theta}"{description}, draw=none, from=0, to=1]
\end{tikzcd}\]
Let $Q_{\operatorname{extn}}\colon\dD^\to\to\dD$ be the double functor representing the extension DOFS. Then we define the components of a horizontal transformation

\[\begin{tikzcd}
	{\mathbb{D}^\to} & {\mathbb{D}^\to} \\
	{\mathbb{D}} & {\mathbb{D}}
	\arrow["{\operatorname{id}}", from=1-1, to=1-2]
	\arrow[""{name=0, anchor=center, inner sep=0}, "Q"', from=1-1, to=2-1]
	\arrow[""{name=1, anchor=center, inner sep=0}, "{Q_{\operatorname{extn}}}", from=1-2, to=2-2]
	\arrow["{\operatorname{id}}"', from=2-1, to=2-2]
	\arrow["\zeta", between={0.2}{0.8}, Rightarrow, from=1, to=0]
\end{tikzcd}\]
by $\zeta_f=\operatorname{id}_{Q(f)}$ for all objects $f$ of $\dD^{\to}$ (i.e., arrows of $\dD$) and $\zeta_\theta=\ell'_\theta$ for each proarrow $\theta$ of $\dD^\rightarrow$.
Using the universal property of the extension cells it is easy to check that this family does indeed define a horizontal transformation, and hence a morphism  of DOFS on $\dD$, $(\dL,\dR)\to(\cL_0,\cR_0)_{\operatorname{extn}}$. For each component of such a morphism there is only one cell that could fulfill the role, so this morphism is unique. 
    \item 
    This follows from the first part by duality.
\end{enumerate}
\end{proof}

Recall that a fibrant category \citep{aleiferi}, also called a framed bicategory \citep{shulman}, is a double category with all extensions and restrictions.

\begin{corollary}
    For any fibrant double category $\dD$ with an OFS $(\cL,\cR)$ on its category of arrows, 
    the DOFS $(\cL,\cR)_{\operatorname{restr}}$ is initial and 
    $(\cL,\cR)_{\operatorname{extn}}$ is terminal
    among the DOFS on $\dD$ with $(\cL,\cR)$ as arrow part.
\end{corollary}

\subsection{Examples}

We will now present examples of double categories with multiple DOFS that share the same OFS on the arrows of the double category.

\begin{example}\label{ex:Mod} (The double category $\mathbb{M}\mathrm{od}$.)
In the category $\mathcal{R}ing$ of associative unital rings, the class $\mathcal{L}$ of surjective morphisms together with the class $\mathcal{R}$ of injective morphisms form an orthogonal factorization system. This follows from the fact that surjective and injective functions form an orthogonal factorization system in the category of sets, and the image of a ring morphism is itself a ring. Moreover, the bijection between two distinct factorizations into a surjective and an injective morphism is itself a ring isomorphism. The OFS $(\mathcal{L},\mathcal{R})$ of $\mathcal{R}ing$ appears as the factorization system for the horizontal arrows in a colax DOFS on the double category $\mathbb{M}\mathrm{od}$ of associative unital rings, bimodules, and equivariant morphisms. The left class $\cL_1$ of the factorization system on the cells consists of surjective cells whose horizontal morphisms are surjective, while the right class $\cR_1$ consists of injective cells whose horizontal morphisms are injective. Every cell $\varphi$ in $\mathbb{M}\mathrm{od}$ admits a factorization of the form:
    \[\begin{tikzcd}[ampersand replacement=\&]
	A \& C \& A \& {\text{Im}f} \& C \\
	B \& D \& B \& {\text{Im}g} \& C
	\arrow["f", from=1-1, to=1-2]
	\arrow[""{name=0, anchor=center, inner sep=0}, "M"', "\shortmid"{marking}, from=1-1, to=2-1]
	\arrow[""{name=1, anchor=center, inner sep=0}, "N", "\shortmid"{marking}, from=1-2, to=2-2]
	\arrow[two heads, from=1-3, to=1-4]
	\arrow[""{name=2, anchor=center, inner sep=0}, "M"', "\shortmid"{marking}, from=1-3, to=2-3]
	\arrow[hook, from=1-4, to=1-5]
	\arrow[""{name=3, anchor=center, inner sep=0}, "{\text{Im}\varphi}"{description}, from=1-4, to=2-4]
	\arrow[""{name=4, anchor=center, inner sep=0}, "N", "\shortmid"{marking}, from=1-5, to=2-5]
	\arrow["g"', from=2-1, to=2-2]
	\arrow[two heads, from=2-3, to=2-4]
	\arrow[hook, from=2-4, to=2-5]
	\arrow["\varphi"{description}, draw=none, from=0, to=1]
	\arrow["{=}"{description}, draw=none, from=1, to=2]
	\arrow["{\overline\varphi}"{description}, draw=none, from=2, to=3]
	\arrow["i"{description}, draw=none, from=3, to=4]
\end{tikzcd}\] where $\bar \varphi$ denotes the restriction of $\varphi$ to its image, and $i$ is the inclusion cell.
Furthermore, $\cL_1$ is closed under vertical composition, since the tensor product of surjective morphisms is again surjective. However, the right class $\cR_1$ is not closed under vertical composition. For instance, the vertical composition of the cell 
\[\begin{tikzcd}[ampersand replacement=\&]
	{\mathbb{R}} \& {\mathbb{C}} \\
	{\mathbb{R}} \& {\mathbb{C}}
	\arrow[hook, from=1-1, to=1-2]
	\arrow[""{name=0, anchor=center, inner sep=0}, "{\mathbb{C}}"', "\shortmid"{marking}, from=1-1, to=2-1]
	\arrow[""{name=1, anchor=center, inner sep=0}, "{\mathbb{C}}", "\shortmid"{marking}, from=1-2, to=2-2]
	\arrow[hook, from=2-1, to=2-2]
	\arrow["\id{}"{description}, draw=none, from=0, to=1]
\end{tikzcd}\] with itself does not belong to the right class $\cR$. This is because $i\otimes_\mathbb{R} 1 \neq 1\otimes_\mathbb{R}i$, while over $\mathbb{C}$ one has $i\otimes_\mathbb{C}1 = 1\otimes_\mathbb{C} i$. Here, the left vertices labeled $\mathbb{R}$ in the cell $\mathrm{id}$, as well as the subscript in the tensor operator, refer to the field of real numbers.

This is not the only DOFS on $\mathbb{M}\mathrm{od}$, because the double category $\mathbb{M}\mathrm{od}$ is fibrant \citep{shulman}. The restriction and extension cells are of the form 
\[\begin{tikzcd}[ampersand replacement=\&]
	A \& C \&\& A \& C \\
	B \& D \&\& B \& D
	\arrow["f", from=1-1, to=1-2]
	\arrow[""{name=0, anchor=center, inner sep=0}, "N"', "\shortmid"{marking}, from=1-1, to=2-1]
	\arrow[""{name=1, anchor=center, inner sep=0}, "N", "\shortmid"{marking}, from=1-2, to=2-2]
	\arrow["f", from=1-4, to=1-5]
	\arrow[""{name=2, anchor=center, inner sep=0}, "M"', "\shortmid"{marking}, from=1-4, to=2-4]
	\arrow[""{name=3, anchor=center, inner sep=0}, "{C\otimes_A M\otimes_B D}", from=1-5, to=2-5]
	\arrow["g"', from=2-1, to=2-2]
	\arrow["g"', from=2-4, to=2-5]
	\arrow["id"{description}, draw=none, from=0, to=1]
	\arrow["{,}"', shift right=3, draw=none, from=1, to=2]
	\arrow["{1\otimes id\otimes 1}"{description}, draw=none, from=2, to=3]
\end{tikzcd}\]
respectively. Here the left proarrow $N$ in the restriction is the $A\text{-}B$-bimodule induced by $f$ and $g$, and the right proarrow $C\otimes_A M\otimes_B D$ in the extension is the $C\text{-}D$-bimodule obtained by tensoring the $C\text{-}A$-bimodule $C$, the $A\text{-}B$-bimodule $M$, and the $B\text{-}D$-bimodule $D$. The morphism $1\otimes id\otimes 1$ maps an element $m\in M$ to $1\otimes m\otimes 1$.

By Proposition \ref{prop:ofsExtensions}, there exist a DOFS $(\mathcal{L}, \mathcal{R})_\text{restr}$ and a DOFS $(\mathcal{L},\mathcal{R})_\text{extn}$  in $\mathbb{M}\mathrm{od}$, induced by restrictions and extensions cells. Then, every cell $\varphi$ admits two factorizations of the form:
\[\begin{tikzcd}[ampersand replacement=\&]
	A \& C \& A \& {\text{Im}f} \& C \\
	B \& D \& B \& {\text{Im}g} \& D
	\arrow["f", from=1-1, to=1-2]
	\arrow[""{name=0, anchor=center, inner sep=0}, "M"', "\shortmid"{marking}, from=1-1, to=2-1]
	\arrow[""{name=1, anchor=center, inner sep=0}, "\shortmid"{marking}, from=1-2, to=2-2]
	\arrow[two heads, from=1-3, to=1-4]
	\arrow[""{name=2, anchor=center, inner sep=0}, "M"', "\shortmid"{marking}, from=1-3, to=2-3]
	\arrow[hook, from=1-4, to=1-5]
	\arrow[""{name=3, anchor=center, inner sep=0}, "N"{description}, from=1-4, to=2-4]
	\arrow[""{name=4, anchor=center, inner sep=0}, "N", "\shortmid"{marking}, from=1-5, to=2-5]
	\arrow["g"', from=2-1, to=2-2]
	\arrow[two heads, from=2-3, to=2-4]
	\arrow[hook, from=2-4, to=2-5]
	\arrow["\varphi"{description}, draw=none, from=0, to=1]
	\arrow["{=}"{description}, draw=none, from=1, to=2]
	\arrow["\varphi"{description}, draw=none, from=2, to=3]
	\arrow["id"{description}, draw=none, from=3, to=4]
\end{tikzcd}\]
and
\[\begin{tikzcd}[ampersand replacement=\&]
	A \& C \& A \&\& {\text{Im}f} \&\& C \\
	B \& D \& B \&\& {\text{Im}g} \&\& D
	\arrow["f", from=1-1, to=1-2]
	\arrow[""{name=0, anchor=center, inner sep=0}, "M"', "\shortmid"{marking}, from=1-1, to=2-1]
	\arrow[""{name=1, anchor=center, inner sep=0}, "\shortmid"{marking}, from=1-2, to=2-2]
	\arrow[two heads, from=1-3, to=1-5]
	\arrow[""{name=2, anchor=center, inner sep=0}, "M"', "\shortmid"{marking}, from=1-3, to=2-3]
	\arrow[hook, from=1-5, to=1-7]
	\arrow[""{name=3, anchor=center, inner sep=0}, "{\text{Im}f\otimes_A M \otimes_B \text{Im}g}"{description}, from=1-5, to=2-5]
	\arrow[""{name=4, anchor=center, inner sep=0}, "N", "\shortmid"{marking}, from=1-7, to=2-7]
	\arrow["g"', from=2-1, to=2-2]
	\arrow[two heads, from=2-3, to=2-5]
	\arrow[hook, from=2-5, to=2-7]
	\arrow["\varphi"{description}, draw=none, from=0, to=1]
	\arrow["{=}"{description}, draw=none, from=1, to=2]
	\arrow["{1\otimes id\otimes 1}"{description, pos=0.35}, draw=none, from=2, to=3]
	\arrow["{\_ \varphi \_}"{description, description, pos=0.65}, draw=none, from=3, to=4]
\end{tikzcd}\]
where the cell $\_\varphi\_$ sends an element $f(a)\otimes_A m\otimes_B g(b)$ to $f(a)\varphi(m)g(b)$. 

The DOFS $(\mathcal{L},\mathcal{R})_\text{restr}$ and $(\mathcal{L},\mathcal{R})_\text{extn}$ differ from the DOFS $(\mathbb{L},\mathbb{R})$. By Theorem \ref{thm:comparaisonCartDOFS}, there exist a morphism of DOFS $(\mathbb{L},\mathbb{R})\to (\mathcal{L},\mathcal{R})_\text{restr}$ and $(\mathcal{L},\mathcal{R})_\text{extn}\to (\mathbb{L},\mathbb{R})$. Explicitly, the globular double cells that define these morphisms are
\[\begin{tikzcd}[ampersand replacement=\&]
	{\text{Im}f} \& {\text{Im}f} \&\&\& {\text{Im}f} \& {\text{Im}f} \\
	{\text{Im}g} \& {\text{Im}g} \&\&\& {\text{Im}g} \& {\text{Im}g}
	\arrow[Rightarrow, no head, from=1-1, to=1-2]
	\arrow[""{name=0, anchor=center, inner sep=0}, "{\text{Im}\varphi}"', "\shortmid"{marking}, from=1-1, to=2-1]
	\arrow[""{name=1, anchor=center, inner sep=0}, "N", "\shortmid"{marking}, from=1-2, to=2-2]
	\arrow[Rightarrow, no head, from=1-5, to=1-6]
	\arrow[""{name=2, anchor=center, inner sep=0}, "{\text{Im}f\otimes_A M \otimes_B \text{Im}g}"', "\shortmid"{marking}, from=1-5, to=2-5]
	\arrow[""{name=3, anchor=center, inner sep=0}, "{\text{Im}\varphi}", "\shortmid"{marking}, from=1-6, to=2-6]
	\arrow[Rightarrow, no head, from=2-1, to=2-2]
	\arrow[Rightarrow, no head, from=2-5, to=2-6]
	\arrow["i"{description}, draw=none, from=0, to=1]
	\arrow["{,}"'{pos=0.3}, draw=none, from=1, to=2]
	\arrow["{\_\overline{\varphi}\_}"{description}, draw=none, from=2, to=3]
\end{tikzcd}\]  
respectively, where the cell $\_\overline{\varphi}\_$ maps an element $f(a)\otimes_A m\otimes_B g(b)$ to $f(a)\overline{\varphi}(m)g(b)$. 
\end{example}

Our next example will be the double category $\Prof$ of categories, profunctors and functors. This is double category is also known to be fibrant, so any orthogonal factorization system on $\Cat$ will give rise to one or more DOFS on $\Prof$.
Here we want to consider the factorization systems by strong epimorphisms and functors that are faithful and injective on objects. We first remind the reader of the definitions.

\begin{definition}\label{def:almosFull}
    A functor is \dfn{compositionally surjective} if its codomain coincides with the subcategory generated by its pointwise image.
\end{definition}

\begin{observation}\label{obs:ofsInCat}
Observe that the class $\mathcal{L}$ of compositionally surjective functors and the class $\mathcal{R}$ of  faithful functors which are injective on objects form an orthogonal factorization system in $\Cat$. For a functor $F\colon \mathcal{A}\to\mathcal{B}$ we write $\mathcal{I}m(F)$ for the subcategory of $\mathcal{B}$ on the objects of the form $F(A)$ with $A$ in $\mathcal{A}$ and arrows generated under composition in $\mathcal{B}$ by the arrows of the form $F(f)$ with $f$ in $\mathcal A$. We then factor $F$ as 
\[\begin{tikzcd}[ampersand replacement=\&]
	{\mathcal{A}} \&\& {\mathcal{I}m(F)} \&\& {\mathcal{B}}
	\arrow["{\overline{F}}", from=1-1, to=1-3]
	\arrow["{J_F}", from=1-3, to=1-5]
\end{tikzcd}\]
where $\overline{F}$ is the restriction of $F$ and $J_F$ is the inclusion of the subcategory into $\mathcal{B}$.
Note that this makes the compositionally surjective functors the strong epimorphisms in $\Cat$.
We will use this factorization system in the next example. 
\end{observation}

\begin{example}
    The OFS in $\Cat$ given in the above observation and the surjective-injective OFS in $\Set$ induce the following DOFS in the double category of profunctors $\Prof$. The left class $\cL_1$ of double cells consists of surjective natural transformations with compositionally surjective functors as horizontal arrows. The right class $\mathbb{R}$ consists of injective natural transformations in which horizontal functors are faithful functors that are injective on objects. It is clear that both $\cL_1$ and $\cR_1$ are closed under horizontal composition and contain all horizontally invertible cells. Before showing that any cell 
    \[\begin{tikzcd}[ampersand replacement=\&]
	{\mathcal{A}} \& {\mathcal{C}} \\
	{\mathcal{B}} \& {\mathcal{D}}
	\arrow["F", from=1-1, to=1-2]
	\arrow[""{name=0, anchor=center, inner sep=0}, "P"', "\shortmid"{marking}, from=1-1, to=2-1]
	\arrow[""{name=1, anchor=center, inner sep=0}, "Q", "\shortmid"{marking}, from=1-2, to=2-2]
	\arrow["G"', from=2-1, to=2-2]
	\arrow["\varphi"{description}, draw=none, from=0, to=1]
\end{tikzcd}\]
     has an $(\mathbb{L}, \mathbb{R})$-factorization, it is necessary to define the profunctor 
    $\mathcal{I}m(\varphi)\: \mathcal{I}m(G)^{\op}\times \mathcal{I}m(F) \to Set.$
    For any objects $x\in \mathcal{I}m(F)$, $y\in\mathcal{I}m(G)$, with $a\in \mathcal{A}$ and $b\in\mathcal{B}$ that satisfy $F(a)=x$ and $G(b) =y$, consider the surjective-injective factorization for the function $\varphi_{b,a}$:
    \[\begin{tikzcd}[ampersand replacement=\&,  row sep=normal, column sep=small]
	{P(b,a)} \&\& {Q(y,x)} \\
	\& {\im\varphi_{b,a}}
	\arrow["{\varphi_{b,a}}", from=1-1, to=1-3]
	\arrow["{\overline{\varphi}_{b,a}}"', from=1-1, to=2-2]
	\arrow["{j_{b,a}}"', hook, from=2-2, to=1-3]
\end{tikzcd}\]
Taking the coproduct over every pair of elements in the fibers of $x$ and $y$, we get the commutative diagram for each pair $(b,a)$ in the fiber over $(y,x)$:
\[\begin{tikzcd}[ampersand replacement=\&]
	{P(b,a)} \& {Q(y,x)} \\
	{\im\varphi_{b,a}} \& \begin{array}{c} \bigsqcup_{\substack{F(a)=x \\ G(b)=y}}\im\varphi_{b,a} \end{array}
	\arrow["{\varphi_{b,a}}", from=1-1, to=1-2]
	\arrow["{\overline\varphi_{b,a}}"', from=1-1, to=2-1]
	\arrow["{j_{b,a}}",  hook, from=2-1, to=1-2]
	\arrow["{k_{b,a}}"', from=2-1, to=2-2]
	\arrow["{q_{y,x}}"', from=2-2, to=1-2]
\end{tikzcd}\]

     The functor $\mathcal{I}m(\varphi)$ is defined on objects as $$\mathcal{I}m(\varphi)(y,x):= \bigsqcup_{\substack{F(a)=x \\ G(b)=y}} \im\varphi_{b,a}$$
    On morphisms it is defined as follows: if $(s,t)$ is a morphism in $\mathcal{I}m(G)^{\op}\times \mathcal{I}m(F)$, then there are morphisms $s_1,\ldots,s_m$ in $\mathcal{B}^{op}$ and $t_1,\ldots,t_n$ in $\mathcal{A}$ such that  $s = G(s_m)\cdots G(s_1)$ and $t=F(t_n)\cdots F(t_1).$ Without loss of generality, $n=m$ (otherwise add identity arrows as needed). For every $i=1,\ldots,n$ there is a unique arrow $\im\varphi(Gs_i,Ft_i)$ given by the functorial property applied to the diagram:
\[\begin{tikzcd}[ampersand replacement=\&, row sep = small, column sep= tiny]
	{P(b_i,a_i)} \&\& {Q(Gb_i,Fa_i)} \\
	\& {\im\varphi_{b_i,a_i}} \\
	\& {\im\varphi_{b_{i+1},a_{i+1}}} \\
	{P(b_{i+1},a_{i+1})} \&\& {Q(Gb_{i+1},Fa_{i+1})}
	\arrow["\varphi_{b_i,a_i}", from=1-1, to=1-3]
	\arrow["{\overline{\varphi}}"', from=1-1, to=2-2]
	\arrow["{P(s_i,t_i)}"', from=1-1, to=4-1]
	\arrow["{Q(Gs_i,Ft_i)}", from=1-3, to=4-3]
	\arrow["j_{b_i,a_i}"', hook, from=2-2, to=1-3]
	\arrow["{\im\varphi(Gs_i,Ft_i)}", dashed, from=2-2, to=3-2]
	\arrow["j_{b_{i+1},a_{i+1}}", hook, from=3-2, to=4-3]
	\arrow["{\overline{\varphi}}", from=4-1, to=3-2]
	\arrow["\varphi_{b_{i+1},a_{i+1}}"', from=4-1, to=4-3]
\end{tikzcd}\]

The composition $\im\varphi(Gs_n,Ft_n)\cdots \im\varphi(Gs_1,Ft_1)$ does not depend on the selection of the $s_i$ and the $t_i$ for $i=1,\ldots,n$ because $j_{b_{n+1},a_{n+1}}$ 
is injective (it is left cancelable) and for each choice it holds that
    \begin{eqnarray*}j_{b_{n+1},a_{n+1}} \im\varphi(Gs_n,Ft_n)\cdots \im\varphi(Gs_1,Ft_1) &=& Q(Gs_n,Ft_n)\cdots Q(Gs_1, Ft_1)j_{b_1,a_1}\\ &=& Q(s,t)j_{b_1,a_1}.\end{eqnarray*}
    
    The functor $\mathcal{I}m\varphi$ is defined on a morphism $(s,t)$ as 
    $$\mathcal{I}m(\varphi)(s,t) := \bigsqcup \im\varphi(Gs_n,Ft_n)\cdots\im\varphi(Gs_1,Ft_1).$$ If $f\:a\to c$ is a morphism in $\mathcal{A}$ and $h\: d\to b$ is a morphism in $\mathcal{B}$, set $\alpha = \im\varphi(Gh, Ff)$. This verifies that the next diagram commutes:
\[\begin{tikzcd}[ampersand replacement=\&]
	{P(b,a)} \& {Q(y,x)} \\
	{\im\varphi_{b,a}} \& \begin{array}{c} \bigsqcup_{\substack{F(a)=x \\ G(b)=y}}\im\varphi_{b,a} \end{array} \\
	{\im\varphi_{d,c}} \& \begin{array}{c} \bigsqcup_{\substack{F(c)=w \\ G(d)=z}}\im\varphi_{d,c} \end{array} \\
	{P(d,c)} \& {Q(z,w)}
	\arrow["\varphi", from=1-1, to=1-2]
	\arrow["{\overline\varphi}"', from=1-1, to=2-1]
	\arrow["{P(h,f)}"', shift right=5, curve={height=30pt}, from=1-1, to=4-1]
	\arrow["{Q(Gh,Ff)}", shift left=5, curve={height=-30pt}, from=1-2, to=4-2]
	\arrow["j", hook, from=2-1, to=1-2]
	\arrow["k"', from=2-1, to=2-2]
	\arrow["\alpha"', from=2-1, to=3-1]
	\arrow["q"', from=2-2, to=1-2]
	\arrow["\bigsqcup\alpha", from=2-2, to=3-2]
	\arrow["k", from=3-1, to=3-2]
	\arrow["{\overline\varphi}"', from=4-1, to=3-1]
	\arrow["j"', hook', from=3-1, to=4-2]
	\arrow["q", from=3-2, to=4-2]
	\arrow["\varphi"', from=4-1, to=4-2]
\end{tikzcd}\]
Thus, $k\overline{\varphi}$ and $q$ are natural transformations and therefore the next equation holds:
    \[\begin{tikzcd}[ampersand replacement=\&]
	{\mathcal{A}} \& {\mathcal{C}} \& {\mathcal{A}} \& {\mathcal{I}m(F)} \& {\mathcal{C}} \\
	{\mathcal{B}} \& {\mathcal{D}} \& {\mathcal{B}} \& {\mathcal{I}m(G)} \& {\mathcal{D}}
	\arrow["F", from=1-1, to=1-2]
	\arrow[""{name=0, anchor=center, inner sep=0}, "P"', "\shortmid"{marking}, from=1-1, to=2-1]
	\arrow[""{name=1, anchor=center, inner sep=0}, "Q", "\shortmid"{marking}, from=1-2, to=2-2]
	\arrow["{\overline{F}}", from=1-3, to=1-4]
	\arrow[""{name=2, anchor=center, inner sep=0}, "P"', "\shortmid"{marking}, from=1-3, to=2-3]
	\arrow["{J_F}", hook, from=1-4, to=1-5]
	\arrow[""{name=3, anchor=center, inner sep=0}, "{\mathcal{I}m(\varphi)}"{description}, from=1-4, to=2-4]
	\arrow[""{name=4, anchor=center, inner sep=0}, "Q", "\shortmid"{marking}, from=1-5, to=2-5]
	\arrow["G"', from=2-1, to=2-2]
	\arrow["{\overline{G}}"', from=2-3, to=2-4]
	\arrow["{J_G}"', hook, from=2-4, to=2-5]
	\arrow["\varphi"{description}, draw=none, from=0, to=1]
	\arrow["{=}"{description}, draw=none, from=1, to=2]
	\arrow["{k\overline\varphi}"{description}, draw=none, from=2, to=3]
	\arrow["q"{description}, draw=none, from=3, to=4]
\end{tikzcd}\]
 The left class $\cL_1$ is closed under vertical composition because the tensor product of surjective morphisms of profunctors is surjective. The counterexample in Example \ref{ex:Mod} shows that the right class $\mathbb{R}$ is not closed under vertical composition.

Now, let us explicitly describe the vertical composition. Given two profunctors $P'\:\mathcal{E}^{op}\times \mathcal{B} \to \Set$ and $P\:\mathcal{B}^{op}\times \mathcal{A} \to \Set$, their vertical composition $P'\otimes_{\mathcal{B}} P$ is given by the coend $$P'\otimes_\mathcal{B} P := \int^{b\in\mathcal{B}}P'(e,b)\times P(b,a).$$
The double category of profunctors is fibrant. The restriction and extension cells are of the form
\[\begin{tikzcd}[ampersand replacement=\&]
	{\mathcal{A}} \& {\mathcal{C}} \&\& {\mathcal{A}} \&\& {\mathcal{C}} \\
	{\mathcal{B}} \& {\mathcal{D}} \&\& {\mathcal{B}} \&\& {\mathcal{D}}
	\arrow["F", from=1-1, to=1-2]
	\arrow[""{name=0, anchor=center, inner sep=0}, "{Q\circ(G\times F)}"', "\shortmid"{marking}, from=1-1, to=2-1]
	\arrow[""{name=1, anchor=center, inner sep=0}, "Q", "\shortmid"{marking}, from=1-2, to=2-2]
	\arrow["F", from=1-4, to=1-6]
	\arrow[""{name=2, anchor=center, inner sep=0}, "P"', "\shortmid"{marking}, from=1-4, to=2-4]
	\arrow[""{name=3, anchor=center, inner sep=0}, "{\text{Hom}_\mathcal{D}(\_,G)\otimes_\mathcal{B} P\otimes_\mathcal{A}\text{Hom}_\mathcal{C}(F_,\_)}", "\shortmid"{marking}, from=1-6, to=2-6]
	\arrow["G"', from=2-1, to=2-2]
	\arrow["G"', from=2-4, to=2-6]
	\arrow["id"{description}, draw=none, from=0, to=1]
	\arrow["{,}"', shift right=2, draw=none, from=1, to=2]
	\arrow["{\text{id}_{G}\otimes id_P\otimes \text{id}_F}"{description}, shorten <=13pt, shorten >=13pt, Rightarrow, from=2, to=3]
\end{tikzcd}\]
respectively, where for each pair $(b,a)$ the natural transformation $\text{id}_G\otimes id_P\otimes\text{id}_F$ is the function $P(b,a) \to \text{Hom}_\mathcal{D}(Gb,G)\otimes_\mathcal{B} P\otimes_\mathcal{A}\text{Hom}(F_,Fa)$ sending $x\in P(b,a)$ to $\text{id}_{Gb}\otimes x\otimes \text{id}_{Fa}$. Then, using the OFS in $\Cat$ described in Observation \ref{obs:ofsInCat}, Proposition \ref{prop:ofsExtensions} yields two DOFS on the double category $\Prof$, namely $(\mathcal{L},\mathcal{R})_\text{restr}$ and $(\mathcal{L},\mathcal{R})_\text{extn}$. In these, each cell $\varphi$ factors

\[\begin{tikzcd}[ampersand replacement=\&]
	{\mathcal{A}} \& {\mathcal{C}} \& {\mathcal{A}} \&\& {\mathcal{I}mF} \&\& {\mathcal{C}} \\
	{\mathcal{B}} \& {\mathcal{D}} \& {\mathcal{B}} \&\& {\mathcal{I}mG} \&\& {\mathcal{D}}
	\arrow["F", from=1-1, to=1-2]
	\arrow[""{name=0, anchor=center, inner sep=0}, "P"', "\shortmid"{marking}, from=1-1, to=2-1]
	\arrow[""{name=1, anchor=center, inner sep=0}, "Q", "\shortmid"{marking}, from=1-2, to=2-2]
	\arrow["{\overline{F}}", from=1-3, to=1-5]
	\arrow[""{name=2, anchor=center, inner sep=0}, "P"', "\shortmid"{marking}, from=1-3, to=2-3]
	\arrow["{i_F}", from=1-5, to=1-7]
	\arrow[""{name=3, anchor=center, inner sep=0}, "{Q\circ(i_G\times i_F)}"{description}, from=1-5, to=2-5]
	\arrow[""{name=4, anchor=center, inner sep=0}, "Q", "\shortmid"{marking}, from=1-7, to=2-7]
	\arrow["G"', from=2-1, to=2-2]
	\arrow["{\overline{G}}"', from=2-3, to=2-5]
	\arrow["{i_G}"', from=2-5, to=2-7]
	\arrow["\varphi"{description}, draw=none, from=0, to=1]
	\arrow["{=}"{description}, draw=none, from=1, to=2]
	\arrow["\varphi"{description, pos=0.4}, draw=none, from=2, to=3]
	\arrow["id"{description, pos=0.6}, draw=none, from=3, to=4]
\end{tikzcd}\]
and
\[\begin{tikzcd}[ampersand replacement=\&]
	{\mathcal{A}} \& {\mathcal{C}} \& {\mathcal{A}} \&\&\&\& {\mathcal{I}mF} \&\&\& {\mathcal{C}} \\
	{\mathcal{B}} \& {\mathcal{D}} \& {\mathcal{B}} \&\&\&\& {\mathcal{I}mG} \&\&\& {\mathcal{D}}
	\arrow["F", from=1-1, to=1-2]
	\arrow[""{name=0, anchor=center, inner sep=0}, "P"', "\shortmid"{marking}, from=1-1, to=2-1]
	\arrow[""{name=1, anchor=center, inner sep=0}, "Q", "\shortmid"{marking}, from=1-2, to=2-2]
	\arrow["{\overline{F}}", from=1-3, to=1-7]
	\arrow[""{name=2, anchor=center, inner sep=0}, "P"', "\shortmid"{marking}, from=1-3, to=2-3]
	\arrow["{i_F}", from=1-7, to=1-10]
	\arrow[""{name=3, anchor=center, inner sep=0}, "{\text{Hom}_{\mathcal{I}mG}(\_,G)\otimes_\mathcal{B} P\otimes_\mathcal{A}\text{Hom}_{\mathcal{I}mF}(F_,\_)}"{description}, from=1-7, to=2-7]
	\arrow[""{name=4, anchor=center, inner sep=0}, "Q", "\shortmid"{marking}, from=1-10, to=2-10]
	\arrow["G"', from=2-1, to=2-2]
	\arrow["{\overline{G}}"', from=2-3, to=2-7]
	\arrow["{i_G}"{description}, from=2-7, to=2-10]
	\arrow["\varphi"{description}, draw=none, from=0, to=1]
	\arrow["{=}"{description}, draw=none, from=1, to=2]
	\arrow["{\text{id}_{\overline{G}}\otimes id_P\otimes \text{id}_{\overline{F}}}"{description, pos=0.3}, draw=none, from=2, to=3]
	\arrow["{Q(\_,\_)\varphi}"{description, pos=0.8}, draw=none, from=3, to=4]
\end{tikzcd}\]
respectively. Here, the function $$(Q(\_,\_)\varphi)_{(Gb,Fa)}\:\text{Hom}(Gb,G)\otimes P \otimes \text{Hom}(F,Fa) \to Q(Gb,Fa)$$ is defined on an element $g\otimes x\otimes f$ by $$(Q(\_,\_)\varphi)_{(Gb,Fa)} (g\otimes x\otimes g)= Q(g,f)(\varphi (x)).$$ 

The DOFS $(\mathcal{L}, \mathcal{R})_\text{restr}$ and $(\mathcal{L}, \mathcal{R})_\text{extn}$ differ from the DOFS $(\mathbb{L}, \mathbb{R})$. By Theorem \ref{thm:comparaisonCartDOFS}
there exist a morphism of DOFS $(\mathbb{L}, \mathbb{R}) \to (\mathcal{L}, \mathcal{R})_\text{restr}$ and $(\mathcal{L}, \mathcal{R})_\text{extn} \to (\mathbb{L}, \mathbb{R})$. Explicitly,
the globular double cells that define these morphisms are

\[\begin{tikzcd}[ampersand replacement=\&]
	{\mathcal{I}mF} \& {\mathcal{I}mF} \\
	{\mathcal{I}mG} \& {\mathcal{I}mG}
	\arrow[Rightarrow, no head, from=1-1, to=1-2]
	\arrow[""{name=0, anchor=center, inner sep=0}, "{\mathcal{I}m(\varphi)}"', "\shortmid"{marking}, from=1-1, to=2-1]
	\arrow[""{name=1, anchor=center, inner sep=0}, "{Q\circ(i_G\times i_F)}", "\shortmid"{marking}, from=1-2, to=2-2]
	\arrow[Rightarrow, no head, from=2-1, to=2-2]
	\arrow["q"{description}, draw=none, from=0, to=1]
\end{tikzcd}\]
and
\[\begin{tikzcd}[ampersand replacement=\&]
	{\mathcal{I}mF} \&\& {\mathcal{I}mF} \\
	{\mathcal{I}mG} \&\& {\mathcal{I}mG}
	\arrow[Rightarrow, no head, from=1-1, to=1-3]
	\arrow[""{name=0, anchor=center, inner sep=0}, "{\text{Hom}_{\mathcal{I}mG}(\_,G)\otimes_\mathcal{B} P\otimes_\mathcal{A}\text{Hom}_{\mathcal{I}mF}(F_,\_)}"', "\shortmid"{marking}, from=1-1, to=2-1]
	\arrow[""{name=1, anchor=center, inner sep=0}, "{\mathcal{I}m(\varphi)}", "\shortmid"{marking}, from=1-3, to=2-3]
	\arrow[Rightarrow, no head, from=2-1, to=2-3]
	\arrow["{Q(\_,\_)k\overline{\varphi}}"{description}, draw=none, from=0, to=1]
\end{tikzcd}\]
respectively. Here, natural trasnformation $Q(\_,\_)k\overline{\varphi}$ is defined analogously to  $Q(\_,\_)\varphi$.
\end{example}

\begin{example}{(Quantale-valued relations, $Q$-$\mathbb{R}\mathrm{el}$)}
    Let $Q$ be a unital quantale.  A {\em $Q$-valued relation} between sets $X, Y$ is a function $R\: X \times Y \to Q$.  Such relations may be composed; given $R\: X \proarrow Y, S\: Y \proarrow Z$, there exists a $Q$-valued relation $R \otimes S\: X \proarrow Z$, given by \[(R \otimes S)(x,z) := \bigvee_{y \in Y} R(x,y) \odot S(y,z)\]
    
    where $\odot$ denotes multiplication in $Q$. There is a double category $Q$-$\mathbb{R}\mathrm{el}$ whose objects are sets, arrows are functions, and proarrows are $Q$-relations, with vertical composition given as above.  A cell 
\[\begin{tikzcd}
	X & {X'} \\
	Y & {Y'}
	\arrow["f", from=1-1, to=1-2]
	\arrow["R"', "\shortmid"{marking}, from=1-1, to=2-1]
	\arrow["\leq"{description}, draw=none, from=1-1, to=2-2]
	\arrow["R'", "\shortmid"{marking}, from=1-2, to=2-2]
	\arrow["g"', from=2-1, to=2-2]
\end{tikzcd}\]
    exists when $R(x,y) \leq R'(f(x), g(y))$ for all $x \in X, y \in Y$. If the inequality is in fact an equality for all $x\in X$ and $y\in Y$, we indicate it by writing an equality symbol inside the cell. The unit proarrow associated to a set $X$ is the function $1:X\times X\to Q$ that sends the diagonal to the top element $1\in Q$ and all other pairs to the bottom element $0\in Q$. The unit cell associated to an arrow $f:X\to X'$ is the cell
    \[\begin{tikzcd}
	X & X' \\
	X & X'
	\arrow["f", from=1-1, to=1-2]
	\arrow[""{name=0, anchor=center, inner sep=0}, "1"', "\shortmid"{marking}, from=1-1, to=2-1]
	\arrow[""{name=1, anchor=center, inner sep=0}, "1", "\shortmid"{marking}, from=1-2, to=2-2]
	\arrow["f"', from=2-1, to=2-2]
	\arrow["{=}"{description}, draw=none, from=0, to=1]
\end{tikzcd}\]

This double category $Q$-$\mathbb{R}\mathrm{el}$ is fibrant. The restriction and extension cells are of the form
\[\begin{tikzcd}[ampersand replacement=\&]
	X \& {X'} \&\& X \& {X'} \\
	Y \& {Y'} \&\& Y \& {Y'}
	\arrow["f", from=1-1, to=1-2]
	\arrow[""{name=0, anchor=center, inner sep=0}, "{\overline{R'}}"', "\shortmid"{marking}, from=1-1, to=2-1]
	\arrow[""{name=1, anchor=center, inner sep=0}, "{R'}", "\shortmid"{marking}, from=1-2, to=2-2]
	\arrow["f", from=1-4, to=1-5]
	\arrow[""{name=2, anchor=center, inner sep=0}, "R"', "\shortmid"{marking}, from=1-4, to=2-4]
	\arrow[""{name=3, anchor=center, inner sep=0}, "{\underline{R}}", from=1-5, to=2-5]
	\arrow["g"', from=2-1, to=2-2]
	\arrow["g"', from=2-4, to=2-5]
	\arrow["{=}"{description}, draw=none, from=0, to=1]
	\arrow["{,}"', shift right=3, draw=none, from=1, to=2]
	\arrow["\leq"{description}, draw=none, from=2, to=3]
\end{tikzcd}\]
respectively. Here, $\overline{R'}:=R'\circ(f\times g)$ and $\underline{R}(x',y') := \bigvee_{\substack{f(x)=x' \\ g(y)=y'}} R(x,y)$. 

Consider the surjective-injective orthogonal factorization system in the category of sets. Let $\mathcal{L}$ denote the class of surjective functions and $\mathcal{R}$ the class of injective functions. By Proposition \ref{prop:ofsExtensions}, there exist two DOFS in $Q$-$\mathbb{R}\mathrm{el}$, namely $(\mathcal{L},\mathcal{R})_\text{restr}$ and $(\mathcal{L},\mathcal{R})_\text{extn}$ . These two DOFS provide distinct ways to factor a cell:
\[\begin{tikzcd}[ampersand replacement=\&]
	X \& {X'} \& X \& {\text{Im}f} \& {X'} \\
	Y \& {Y'} \& Y \& {\text{Im}g} \& {Y'}
	\arrow["f", from=1-1, to=1-2]
	\arrow[""{name=0, anchor=center, inner sep=0}, "R"', "\shortmid"{marking}, from=1-1, to=2-1]
	\arrow[""{name=1, anchor=center, inner sep=0}, "R'", "\shortmid"{marking}, from=1-2, to=2-2]
	\arrow[two heads, from=1-3, to=1-4]
	\arrow[""{name=2, anchor=center, inner sep=0}, "R"', "\shortmid"{marking}, from=1-3, to=2-3]
	\arrow[hook, from=1-4, to=1-5]
	\arrow[""{name=3, anchor=center, inner sep=0}, "{\overline{R'}}"{description}, from=1-4, to=2-4]
	\arrow[""{name=4, anchor=center, inner sep=0}, "R'", "\shortmid"{marking}, from=1-5, to=2-5]
	\arrow["g"', from=2-1, to=2-2]
	\arrow[two heads, from=2-3, to=2-4]
	\arrow[hook, from=2-4, to=2-5]
	\arrow["\leq"{description}, draw=none, from=0, to=1]
	\arrow["{=}"{description}, draw=none, from=1, to=2]
	\arrow["\leq"{description}, draw=none, from=2, to=3]
	\arrow["{=}"{description}, draw=none, from=3, to=4]
\end{tikzcd}\]
and
\[\begin{tikzcd}[ampersand replacement=\&]
	X \& {X'} \& X \& {\text{Im}f} \& {X'} \\
	Y \& {Y'} \& Y \& {\text{Im}g} \& {Y'}
	\arrow["f", from=1-1, to=1-2]
	\arrow[""{name=0, anchor=center, inner sep=0}, "R"', "\shortmid"{marking}, from=1-1, to=2-1]
	\arrow[""{name=1, anchor=center, inner sep=0}, "R'", "\shortmid"{marking}, from=1-2, to=2-2]
	\arrow[two heads, from=1-3, to=1-4]
	\arrow[""{name=2, anchor=center, inner sep=0}, "R"', "\shortmid"{marking}, from=1-3, to=2-3]
	\arrow[hook, from=1-4, to=1-5]
	\arrow[""{name=3, anchor=center, inner sep=0}, "{\underline{R}}"{description}, from=1-4, to=2-4]
	\arrow[""{name=4, anchor=center, inner sep=0}, "R'", "\shortmid"{marking}, from=1-5, to=2-5]
	\arrow["g"', from=2-1, to=2-2]
	\arrow[two heads, from=2-3, to=2-4]
	\arrow[hook, from=2-4, to=2-5]
	\arrow["\leq"{description}, draw=none, from=0, to=1]
	\arrow["{=}"{description}, draw=none, from=1, to=2]
	\arrow["\leq"{description}, draw=none, from=2, to=3]
	\arrow["\leq"{description}, draw=none, from=3, to=4]
\end{tikzcd}\]
respectively.
We verify how these two factorization systems are related.
For $a\in \text{Im}f$ and $b\in\text{Im} g$, we calculate:
\begin{eqnarray*}
\underline{R}(a,b)&=&\bigvee_{\substack{f(x)=a \\ g(y)=b}}R(x,y)\\
&\le&\bigvee_{\substack{f(x)=a \\ g(y)=b}}R'(f(x),g(y))\\
&=&R'(a,b)\\
&=&\overline{R'}(a,b).
\end{eqnarray*}
So we obtain a morphism of DOFS extending the epi-mono factorization system $(\cL,\cR)$ on $\Set$, $(\mathcal{L},\mathcal{R})_{\operatorname{restr}}\to (\mathcal{L},\mathcal{R})_{\operatorname{extn}}$.

 In the particular case $Q = \{0, 1\}$ we recover the double category $\mathbb{R}\mathrm{el}(\mathrm{Set})$ from Example \ref{ex:rel}, and the DOFS $(\mathcal{L},\mathcal{R})_\text{extn}$ coincides with the DOFS  given in that example.
\end{example}

\begin{example} (The double category $\mathbb{S}\mathrm{pan}(\mathcal{C})$)\label{examplespan2}
Let $\C$ be a category with an OFS $(\cL,\cR)$. In this example we want to expand on the work in Example \ref{examplespan}.
Let $Q\colon \C^\to\to\C$ be the functor that represents the OFS  $(\cL,\cR)$. And write $f=r_f\circ\ell_f$ for the chosen factorizations of the arrows in $\C$.
Then the factorization of a double cell in the DOFS of Example \ref{examplespan} is given by:

\[\begin{tikzcd}
	\bullet & \bullet && \bullet & \bullet & \bullet \\
	\bullet & \bullet && \bullet & \bullet & \bullet \\
	\bullet & \bullet && \bullet & \bullet & \bullet
	\arrow["f", from=1-1, to=1-2]
	\arrow["{\ell_f}", from=1-4, to=1-5]
	\arrow["{r_f}", from=1-5, to=1-6]
	\arrow["u", from=2-1, to=1-1]
	\arrow["m"', from=2-1, to=2-2]
	\arrow["v"', from=2-1, to=3-1]
	\arrow["{u'}"', from=2-2, to=1-2]
	\arrow["{=}"{description}, draw=none, from=2-2, to=2-4]
	\arrow["{v'}", from=2-2, to=3-2]
	\arrow["u", from=2-4, to=1-4]
	\arrow["{\ell_m}"', from=2-4, to=2-5]
	\arrow["v"', from=2-4, to=3-4]
	\arrow["{Q(u,u')}"', from=2-5, to=1-5]
	\arrow["{r_m}"', from=2-5, to=2-6]
	\arrow["{Q(v,v')}", from=2-5, to=3-5]
	\arrow["{u'}"', from=2-6, to=1-6]
	\arrow["{v'}", from=2-6, to=3-6]
	\arrow["g"', from=3-1, to=3-2]
	\arrow["{\ell_g}"', from=3-4, to=3-5]
	\arrow["{r_g}"', from=3-5, to=3-6]
\end{tikzcd}\]
Extensions in $\mathbb{S}\mathrm{pan}(\mathcal{C})$ are double cells of the form:

\[\begin{tikzcd}
	\bullet & \bullet \\
	\bullet & \bullet \\
	\bullet & \bullet
	\arrow["f", from=1-1, to=1-2]
	\arrow["u", from=2-1, to=1-1]
	\arrow["{\operatorname{id}}", from=2-1, to=2-2]
	\arrow["v"', from=2-1, to=3-1]
	\arrow["fu"', from=2-2, to=1-2]
	\arrow["gv", from=2-2, to=3-2]
	\arrow["g"', from=3-1, to=3-2]
\end{tikzcd}\]
Hence, the factorization of cells in $(\cL,\cR)_{\operatorname{extn}}$
is given by:

\[\begin{tikzcd}
	\bullet & \bullet && \bullet & \bullet & \bullet \\
	\bullet & \bullet && \bullet & \bullet & \bullet \\
	\bullet & \bullet && \bullet & \bullet & \bullet
	\arrow["f", from=1-1, to=1-2]
	\arrow["{\ell_f}", from=1-4, to=1-5]
	\arrow["{r_f}", from=1-5, to=1-6]
	\arrow["u", from=2-1, to=1-1]
	\arrow["m"', from=2-1, to=2-2]
	\arrow["v"', from=2-1, to=3-1]
	\arrow["{u'}"', from=2-2, to=1-2]
	\arrow["{=}"{description}, draw=none, from=2-2, to=2-4]
	\arrow["{v'}", from=2-2, to=3-2]
	\arrow["u", from=2-4, to=1-4]
	\arrow["{\operatorname{id}}"', from=2-4, to=2-5]
	\arrow["v"', from=2-4, to=3-4]
	\arrow["{\ell_fu}"', from=2-5, to=1-5]
	\arrow["m"', from=2-5, to=2-6]
	\arrow["{\ell_gv}", from=2-5, to=3-5]
	\arrow["{u'}"', from=2-6, to=1-6]
	\arrow["{v'}", from=2-6, to=3-6]
	\arrow["g"', from=3-1, to=3-2]
	\arrow["{\ell_g}"', from=3-4, to=3-5]
	\arrow["{r_g}"', from=3-5, to=3-6]
\end{tikzcd}\]
If $\C$ has pullbacks, $\mathbb{S}\mathrm{pan}(\mathcal{C})$ also has restrictions, defined as follows:

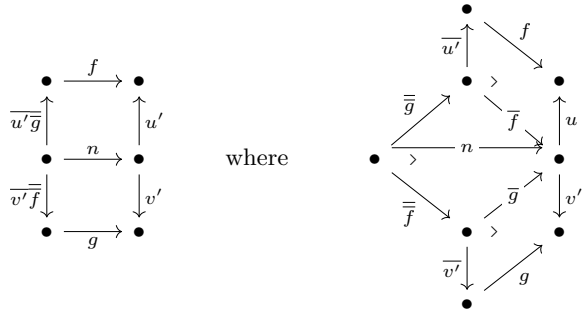
\[\begin{tikzcd}
	&&&& \bullet \\
	\bullet & \bullet &&& \bullet & \bullet \\
	\bullet & \bullet & {\mathrm{where}} & \bullet && \bullet \\
	\bullet & \bullet &&& \bullet & \bullet \\
	&&&& \bullet
	\arrow["f", from=1-5, to=2-6]
	\arrow["f", from=2-1, to=2-2]
	\arrow["{\overline{u'}}", from=2-5, to=1-5]
	\arrow["\lrcorner"{anchor=center, pos=0.125, rotate=45}, draw=none, from=2-5, to=2-6]
	\arrow["{\overline{f}}"{description}, from=2-5, to=3-6]
	\arrow["{\overline{u'}\overline{\overline{g}}}", from=3-1, to=2-1]
	\arrow["n", from=3-1, to=3-2]
	\arrow["{\overline{v'}\overline{\overline{f}}}"', from=3-1, to=4-1]
	\arrow["{u'}"', from=3-2, to=2-2]
	\arrow["{v'}", from=3-2, to=4-2]
	\arrow["{\overline{\overline{g}}}", from=3-4, to=2-5]
	\arrow["\lrcorner"{anchor=center, pos=0.125, rotate=45}, draw=none, from=3-4, to=3-6]
	\arrow["n"{description}, shift left=2, from=3-4, to=3-6]
	\arrow["{\overline{\overline{f}}}"', from=3-4, to=4-5]
	\arrow["u"', from=3-6, to=2-6]
	\arrow["{v'}", from=3-6, to=4-6]
	\arrow["g"', from=4-1, to=4-2]
	\arrow["{\overline{g}}"{description}, from=4-5, to=3-6]
	\arrow["\lrcorner"{anchor=center, pos=0.125, rotate=45}, draw=none, from=4-5, to=4-6]
	\arrow["{\overline{v'}}"', from=4-5, to=5-5]
	\arrow["g"', from=5-5, to=4-6]
\end{tikzcd}\]
and the corresponding factorization of double cells is given by:

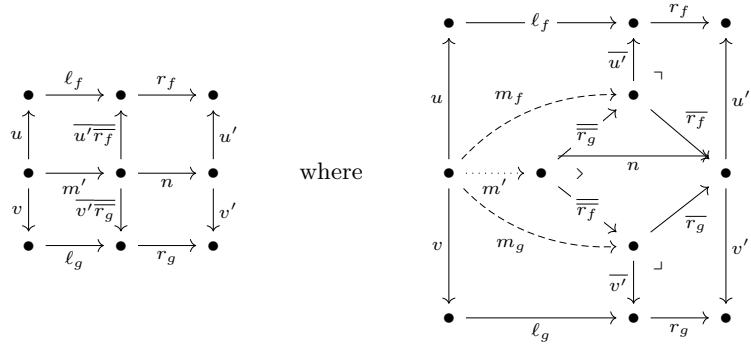
\[\begin{tikzcd}
	&&&& \bullet && \bullet & \bullet \\
	\bullet & \bullet & \bullet &&&& \bullet \\
	\bullet & \bullet & \bullet & {\mathrm{where}} & \bullet & \bullet && \bullet \\
	\bullet & \bullet & \bullet &&&& \bullet \\
	&&&& \bullet && \bullet & \bullet
	\arrow["{\ell_f}"{description}, from=1-5, to=1-7]
	\arrow["{r_f}", from=1-7, to=1-8]
	\arrow["{\ell_f}", from=2-1, to=2-2]
	\arrow["{r_f}", from=2-2, to=2-3]
	\arrow["{\overline{u'}}", from=2-7, to=1-7]
	\arrow["\lrcorner"{anchor=center, pos=0.125, rotate=90}, draw=none, from=2-7, to=1-8]
	\arrow["{\overline{r_f}}", from=2-7, to=3-8]
	\arrow["u", from=3-1, to=2-1]
	\arrow["{m'}"', from=3-1, to=3-2]
	\arrow["v"', from=3-1, to=4-1]
	\arrow["{\overline{u'}\overline{\overline{r_f}}}", from=3-2, to=2-2]
	\arrow["n"', from=3-2, to=3-3]
	\arrow["{\overline{v'}\overline{\overline{r_g}}}"', from=3-2, to=4-2]
	\arrow["{u'}"', from=3-3, to=2-3]
	\arrow["{v'}", from=3-3, to=4-3]
	\arrow["u", from=3-5, to=1-5]
	\arrow["{m_f}", curve={height=-12pt}, dashed, from=3-5, to=2-7]
	\arrow["{m'}"', dotted, from=3-5, to=3-6]
	\arrow["{m_g}"', curve={height=12pt}, dashed, from=3-5, to=4-7]
	\arrow["v"', from=3-5, to=5-5]
	\arrow["{\overline{\overline{r_g}}}"{description}, from=3-6, to=2-7]
	\arrow["\lrcorner"{anchor=center, pos=0.125, rotate=45}, draw=none, from=3-6, to=3-8]
	\arrow["n"', shift left=3, from=3-6, to=3-8]
	\arrow["{\overline{\overline{r_f}}}"{description}, from=3-6, to=4-7]
	\arrow["{u'}"', from=3-8, to=1-8]
	\arrow["{v'}", from=3-8, to=5-8]
	\arrow["{\ell_g}"', from=4-1, to=4-2]
	\arrow["{r_g}"', from=4-2, to=4-3]
	\arrow["{\overline{r_g}}"', from=4-7, to=3-8]
	\arrow["{\overline{v'}}"', from=4-7, to=5-7]
	\arrow["\lrcorner"{anchor=center, pos=0.125}, draw=none, from=4-7, to=5-8]
	\arrow["{\ell_g}"', from=5-5, to=5-7]
	\arrow["{r_g}"', from=5-7, to=5-8]
\end{tikzcd}\]
where the dashed and dotted arrows are induced by the universal properties of the pullbacks.

In the following diagram, we see how our original DOFS sits in between the extension and the restriction systems:

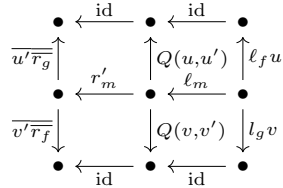
\[\begin{tikzcd}
	\bullet & \bullet & \bullet \\
	\bullet & \bullet & \bullet \\
	\bullet & \bullet & \bullet
	\arrow["{\operatorname{id}}"', from=1-2, to=1-1]
	\arrow["{\operatorname{id}}"', from=1-3, to=1-2]
	\arrow["{\overline{u'}\overline{\overline{r_g}}}", from=2-1, to=1-1]
	\arrow["{\overline{v'}\overline{\overline{r_f}}}"', from=2-1, to=3-1]
	\arrow["{Q(u,u')}"', from=2-2, to=1-2]
	\arrow["{r_m'}"', from=2-2, to=2-1]
	\arrow["{Q(v,v')}", from=2-2, to=3-2]
	\arrow["{\ell_fu}"', from=2-3, to=1-3]
	\arrow["{\ell_m}"', from=2-3, to=2-2]
	\arrow["{l_gv}", from=2-3, to=3-3]
	\arrow["{\operatorname{id}}", from=3-2, to=3-1]
	\arrow["{\operatorname{id}}", from=3-3, to=3-2]
\end{tikzcd}\]
where $r_m'$ is the unique arrow into the pullback induced by $r_m$ (analogous to how $m'$ was defined in the previous diagram).

We will continue this example in Example \ref{imagedofs} of the next section where we will add one more factorization system on $\mathbb{S}\mathrm{pan}(\mathcal{C})$.
\end{example}

\section{Double factorization systems and double fibrations}\label{doublefibs}
We study the interaction between double fibrations, introduced in \citep{df}, and our double orthogonal factorization systems. We generalize some results that are well-known in the ordinary case to the double categorical context and we present some examples. In particular, we show that both DOFS and lax DOFS can be lifted along double fibrations.

We start recalling the following folklore result for which we could not find an appropriate reference. The construction in the proof will be used in proving the results for DOFS and double fibrations that follow. 

\begin{proposition}\label{liftalongfib}
Let $P\: \mathcal{E} \to \mathcal{B}$ be an (ordinary) cloven Grothendieck fibration and let $(\mathcal{L}, \mathcal{R})$ be an orthogonal factorization system on $\mathcal{B}$. Then $(\cL, \cR)$ lifts to an orthogonal factorization system $(\mathcal{L}^{P}, \mathcal{R}^{P})$ on $\mathcal{E}$, taking $\mathcal{L}^{P}$ to be the collection of morphisms over morphisms in $\mathcal{L}$ and $\mathcal{R}^{P}$ to be the collection of cartesian morphisms over morphisms in $\mathcal{R}$.
\end{proposition}
\begin{proof}
Let $f\: x \to y$ be a morphism in $\mathcal{E}$ and let $P(f)=r_{P(f)}\circ\ell_{P(f)}$ be the chosen factorization in $\mathcal{B}$. Then $f=\operatorname{cart}(r_{P(f)}, y) \c g$, where $\operatorname{cart}(r_{P(f)}, y)$ is the chosen cartesian lifting of $r_{P(f)}$ with codomain $y$ and $g$ is the unique morphism induced by the cartesian property of this lifting as in the following diagram; i.e., $P(g)=\ell_{P(f)}$.
\begin{cd}[6][8]
{x} \arrow[rr,"{f}", bend left=30]\arrow[r,"{\exists !g}"', dashed] \arrow[d,"{}",mapsto]\& {{(r_{P(f)})}^{\star} y} \arrow[r,"{\operatorname{cart}(r_{P(f)}, y)}"'] \arrow[d,"{}",mapsto]\& {y}\arrow[d,"{}",mapsto]\\
{P(x)} \arrow[r,"{\ell_{P(f)}}"', two heads]\& {Q(P(f))} \arrow[r,"{r_{P(f)}}"', hook]\& {P(y).}
\end{cd}
So $f$ can be factorized as a composite of a morphism over a morphism in $\mathcal{L}$ followed by a morphism that is cartesian over a morphism in $\mathcal{R}$.

Moreover, using the fact that the morphisms in $\mathcal{R}^P$ are cartesian, it is straightforward to prove that the required lifting properties hold and hence $(\mathcal{L}^{P}, \mathcal{R}^{P})$ is an orthogonal factorization system on $\mathcal{E}$.
\end{proof}

\begin{rmk}\label{strictfib}
    Note that in the proof we provided a lifting of the factorization system that makes the fibration $P$ a strict morphism between categories with an OFS.
\end{rmk}

We want to generalize the previous result of lifting of OFS along fibrations to the double categorical context. 
Double fibrations have been introduced and studied in \citep{df}. 

We recall that a double fibration $P\: \dE \to  \dB$ between pseudo double categories is a pseudo-category in the category $\Fib$ of fibrations. This amounts to a double functor $P\: \dE \to \dB$ between pseudo double categories:
\begin{eqD}{doublefib}
\begin{cd}*[7][8]
{\dE_1 \times_{\dE_0} \arrow[d,"{P_1 \times_{P_0} P_1}"'] \arrow[r,"{\otimes_{\dE}}"]\dE_1}\& {\dE_1} \arrow[d,"{P_1}"']\arrow[r,"s_{\dE}", shift left=1.6ex] \arrow[r,"t_{\dE}"', shift right=1.8ex] \& {\dE_0} \arrow[d,"{P_0}"]\arrow[l,"{y_{\dE}}"{description}]\\
{\dB_1 \times_{\dB_0} \arrow[r,"{\otimes_{\dB}}"]\dB_1}\& {\dB_1} \arrow[r,"s_{\dB}", shift left=1.6ex] \arrow[r,"t_{\dB}"', shift right=1.8ex] \& {\dB_0} \arrow[l,"{y_{\dB}}"{description}]
\end{cd}
\end{eqD}
such that $P_0$ and $P_1$ are cloven fibrations, $s_{\dE}$ and  $t_{\dE}$ are cleavage preserving, and $y_{\dE}$ and $\otimes_{\dE}$ preserve cartesian morphisms.

\begin{theorem}\label{liftalongdblfib}
Let $P\: \dE \to \dB$ be a double fibration and let $\dofs{\cL}{\cR}$ be a double orthogonal factorization system (resp., a lax one) on $\dB$. Then $\dofs{\cL}{\cR}$ lifts to a double orthogonal factorization system (resp., a lax one) $((\cL_0^P,\cR_0^P),(\cL_1^P,\cR_1^P))$ on $\dE$, taking  $(\cL^P_0, \cR^P_0):=(\cL^{P_0},\cR^{P_0})$ and $(\cL^P_1, \cR^P_1):=(\cL^{P_1},\cR^{P_1})$ (in the notation of \prox\ref{liftalongfib}).
\end{theorem}

\begin{proof}
Assume that $\dofs{\cL}{\cR}$ is a lax double orthogonal factorization system. We need to prove that $((\dE_0, (\cL^{P_0},\cR^{P_0})), (\dE_1, (\cL^{P_1},\cR^{P_1})))$ is a pseudo-category in $\Factlax$ with strict source, target and identity arrows. 
We see immediately that \prox\ref{liftalongfib} and Remark \ref{strictfib} imply that $(\cL^{P_0},\cR^{P_0})$ and $(\cL^{P_1},\cR^{P_1})$ are orthogonal factorization systems on $\dE_0$ and $\dE_1$ respectively such that $P_0$ and $P_1$ are strict morphisms between categories with orthogonal factorization systems. 
We will also show below that the factorization systems on $\dE$ combine into a lax double orthogonal factorization system in the sense that $s_{\dE}$, $t_{\dE}$ and $y_{\dE}$ are strict morphisms between categories with factorization systems with the chosen factorizations from the proof of \prox\ref{liftalongfib}.
It follows then, thanks to \prox\ref{propfacthaspullbacks}, that the required pullbacks in $\Cat$ lift automatically to pullbacks in $\Factlax$.

We now check that $s_{\dE}$ and $t_{\dE}$ preserve the chosen factorizations. 
Let $\alpha$ be a cell in $\dE$,
\[\begin{tikzcd}[ampersand replacement=\&]
	{s_{\dE}(u)} \&\& {s_{\dE}(v)} \\
	{t_{\dE}(u)} \&\& {t_{\dE}(v)}
	\arrow["{s_{\dE}(\alpha)}", from=1-1, to=1-3]
	\arrow["u"', "\shortmid"{marking}, from=1-1, to=2-1]
	\arrow["\alpha"{description}, draw=none, from=1-1, to=2-3]
	\arrow["v", "\shortmid"{marking}, from=1-3, to=2-3]
	\arrow["{t_{\dE}(\alpha)}"', from=2-1, to=2-3]
\end{tikzcd}\]
The factorization of $P_1(\alpha)$ is 
\begin{equation}\label{factn}
    \begin{tikzcd}[ampersand replacement=\&]
	{s_{\dB}P_1(u)} \&\& \bullet \&\& {s_{\dB}P_1v} \\
	{t_{\dB}P_1(u)} \&\& \bullet \&\& {t_{\dB}P_1(v)}
	\arrow["{s_{\dB}\ell_{P_1(\alpha)}}", from=1-1, to=1-3]
	\arrow[""{name=0, anchor=center, inner sep=0}, "{P_1(u)}"', "\shortmid"{marking}, from=1-1, to=2-1]
	\arrow["s_{\dB}{r_{P_1(\alpha)}}", from=1-3, to=1-5]
	\arrow[""{name=1, anchor=center, inner sep=0}, "\shortmid"{marking}, from=1-3, to=2-3]
	\arrow[""{name=2, anchor=center, inner sep=0}, "{P_1(v)}", "\shortmid"{marking}, from=1-5, to=2-5]
	\arrow["{t_{\dB}\ell_{P_1(\alpha)}}"', from=2-1, to=2-3]
	\arrow["{t_{\dB}r_{P_1(\alpha)}}"', from=2-3, to=2-5]
	\arrow["{\ell_{P_1(\alpha)}}"{description}, draw=none, from=0, to=1]
	\arrow["{r_{P_1(\alpha)}}"{description}, draw=none, from=1, to=2]
    \end{tikzcd}
\end{equation}
So the factorization of $\alpha$ is 

\[\begin{tikzcd}[ampersand replacement=\&]
	{s_{\dE}(u)} \&\& \bullet \&\& {s_{\dE}v} \\
	{t_{\dE}(u)} \&\& \bullet \&\& {t_{\dE}v}
	\arrow["{s_{\dE}(\beta)}", from=1-1, to=1-3]
	\arrow[""{name=0, anchor=center, inner sep=0}, "u"', "\shortmid"{marking}, from=1-1, to=2-1]
	\arrow["{s_{\dE}(\operatorname{cart}(r_{P_1(\alpha)},v))}", from=1-3, to=1-5]
	\arrow[""{name=1, anchor=center, inner sep=0}, "\shortmid"{marking}, from=1-3, to=2-3]
	\arrow[""{name=2, anchor=center, inner sep=0}, "v", "\shortmid"{marking}, from=1-5, to=2-5]
	\arrow["{t_{\dE}(\beta)}"', from=2-1, to=2-3]
	\arrow["{t_{E}(\operatorname{cart}(r_{P_1(\alpha)},v))}"', from=2-3, to=2-5]
	\arrow["\beta"{description}, draw=none, from=0, to=1]
	\arrow["{\operatorname{cart}(r_{P_1(\alpha)},v)}"{description}, draw=none, from=1, to=2]
\end{tikzcd}\]
where $\operatorname{cart}(r_{P_1(\alpha)},v)$ is the cartesian lifting of $r_{P_1(\alpha)}$ and $\beta$ is the unique arrow over $\ell_{P_1(\alpha)}$ so that the composition gives $\alpha$.
We want to show that ${s_{\dE}(\operatorname{cart}(r_{P_1(\alpha)},v))}\circ{s_{\dE}(\beta)}$ is the chosen factorization in $\dE_0$ of $s_{\dE}(\alpha)$ and similarly for $t_{\dE}(\alpha)$.

Since $s_{\dB}$ and $t_{\dB}$ are strict, (\ref{factn}) is equal to the diagram
\[\begin{tikzcd}[ampersand replacement=\&]
	{s_{\dB}P_1(u)} \&\& \bullet \&\& {s_{\dB}P_1v} \\
	{t_{\dB}P_1(u)} \&\& \bullet \&\& {t_{\dB}P_1(v)}
	\arrow["{\ell_{s_{\dB}P_1(\alpha)}}", from=1-1, to=1-3]
	\arrow[""{name=0, anchor=center, inner sep=0}, "{P_1(u)}"', "\shortmid"{marking}, from=1-1, to=2-1]
	\arrow["{r_{s_{\dB}P_1(\alpha)}}", from=1-3, to=1-5]
	\arrow[""{name=1, anchor=center, inner sep=0}, "\shortmid"{marking}, from=1-3, to=2-3]
	\arrow[""{name=2, anchor=center, inner sep=0}, "{P_1(v)}", "\shortmid"{marking}, from=1-5, to=2-5]
	\arrow["{\ell_{t_{\dB}P_1(\alpha)}}"', from=2-1, to=2-3]
	\arrow["{r_{t_{\dB}P_1(\alpha)}}"', from=2-3, to=2-5]
	\arrow["{\ell_{P_1(\alpha)}}"{description}, draw=none, from=0, to=1]
	\arrow["{r_{P_1(\alpha)}}"{description}, draw=none, from=1, to=2]
\end{tikzcd}\]

Since $P_0s_{\dE}(\alpha)=s_{\dE}P_1(\alpha)$ and $P_0t_{\dE}(\alpha)=t_{\dE}P_1(\alpha)$ and the $s_{\dE}$ is cleavage preserving as map between fibrations, we have that 
\begin{eqnarray*}
r_{s_{\dE}(\alpha)}&=&\operatorname{cart}(r_{P_0s_{\dE}(\alpha)},s_{\dE}(v))\\
&=& \operatorname{cart}(r_{s_{\dB}P_1(\alpha)},s_{\dE}(v))\\
&=& s_{\dE}\operatorname{cart}(r_{P_1(\alpha)},v)
\end{eqnarray*}
By uniqueness of factorizations through cartesian arrows and the fact that $s_{\dE}(\beta)$ fits we find that $$\ell_{s_{\dE}(\alpha)}=s_{\dE}(\beta)$$
We conclude that $s_{\dE}$ is a strict map between categories with a DOF. The proofs for $t_{\dE}$ and $y_{\dE}$ go similarly, using that $t$ and $y$ are cleavage preserving and that $t_{\dB}$ and $y_{\dB}$ are strict maps of categories with a DOF. 

We observe that $s_{\dE}, t_{\dE}$ and $y_{\dE}$ preserve both the left and the right class of arrows. Indeed, they preserve cartesian morphisms by definition of double fibration and the fact that the images of morphisms in the left and right class of the total category are over morphisms in the left and right class in the base category  is guaranteed by the fact that $s_{\dB}, t_{\dB}$ and $y_{\dB}$ preserve morphisms in the left and right class and by the commutativity of diagram (\ref{doublefib}). For example, given a double cell $\alpha\in \cL^{P_1}$, we have $P_0(s_{\dE}(\alpha))= s_{\dB}(P_1(\alpha))$ and $s_{\dB}(P_1(\alpha))\in \cL_0$ since $P_1(\alpha)\in \cL_1$ and $s_{\dB}$ preserves the left class.

Furthermore, since $\otimes_{\dB}$ preserves the right class, and $\otimes_{\dE}$ preserves cartesian arrows it follows that $\otimes_{\dE}$ preserves the right class. This concludes the proof that $((\cL_0^P,\cR_0^P),(\cL_1^P,\cR_1^P))$ is a lax double orthogonal factorization system.

Let now $\dofs{\cL}{\cR}$ be a strict double orthogonal factorization system. Everything we proved in the lax case still holds in the same way. Moreover, since $\otimes_{\dB}$ also preserves the left class, we obtain that $\otimes_{\dE}$ preserves the left class as  because of uniqueness of the factorizations through a cartesian arrow. Hence, $((\cL_0^P,\cR_0^P),(\cL_1^P,\cR_1^P))$ is a strict double orthogonal factorization system.
\end{proof}

\begin{corollary}\label{pseudovertcartdofs}
Let $P\: \dE \to \dB$ be a double fibration. Then $\dE$ has a double orthogonal factorization system $((\text{pseudo-vertical}, \text{cartesian}), (\text{pseudo-vertical}, \text{cartesian}))$, where  an arrow is pseudo-vertical if $P$ sends it to an isomorphism.
\end{corollary}

\begin{proof}
By \prox\ref{liftalongdblfib} we can lift the trivial double orthogonal factorization system $((\text{isomorphisms}, \text{all morphisms}), (\text{isomorphisms}, \text{all morphisms}))$ on $\dB$ to a double orthogonal factorization system on $\dE$. By construction, the double orthogonal factorization system obtained this way has left classes given by pseudo-vertical morphisms and right classes given by cartesian morphisms. 
\end{proof}

We can now easily prove the analogous results for opfibrations.

\begin{proposition}\label{liftalongopfib}
Let $P\: \mathcal{E} \to \mathcal{B}$ be an (ordinary) cloven Grothendieck opfibration and let $(\mathcal{L}, \mathcal{R})$ be an orthogonal factorization system on $\mathcal{B}$. Then $(\cL, \cR)$ lifts to an orthogonal factorization system $(\mathcal{L}^{P}, \mathcal{R}^{P})$ on $\mathcal{E}$, taking $\mathcal{L}^{P}$ to be the collection of cocartesian morphisms over morphisms in $\mathcal{L}$ and $\mathcal{R}^{P}$ to be the collection of morphisms over morphisms in $\mathcal{R}$.
\end{proposition}

\begin{proof}
Analogous to the proof of \prox\ref{liftalongfib}. We now induce the right part of the factorization of a given morphism in $\mathcal{E}$ using the cocartesian property of the chosen cocartesian lifting of the left part of the image of the given morphism under $P$. Moreover, the required lifting properties can be proved using the cocartesian property of the morphisms in $\cL^{P}$.
\end{proof}

\begin{theorem}\label{liftalongdblopfib}
Let $P: \dE \to \dB$ be a double opfibration and let $\dofs{\cL}{\cR}$ be a double orthogonal factorization system (respectively, a lax one) on $\dB$. Then $\dofs{\cL}{\cR}$ lifts to a double orthogonal factorization system (respectively, a lax one) \linebreak $\dofs{\cL^P}{\cR^P}$  on $\dE$, taking  $(\cL^P_0, \cR^P_0):=(\cL^{P_0},\cR^{P_0})$ and $(\cL^P_1, \cR^P_1):=(\cL^{P_1},\cR^{P_1})$ (in the notation of \prox\ref{liftalongopfib}).
\end{theorem}

\begin{proof}
Analogous to the one of \thex\ref{liftalongdblfib}.
\end{proof}

\begin{corollary}\label{cocartpseudovertdofs}
Let $P\: \dE \to \dB$ be a double opfibration. Then $\dE$ has a double orthogonal factorization system $(( \text{cocartesian},\text{pseudovertical}), ( \text{cocartesian},\text{pseudovertical}))$.
\end{corollary}

\begin{proof}
By \prox\ref{liftalongdblopfib} we can lift the trivial double orthogonal factorization system $((\text{isomorphisms}, \text{all morphisms}), (\text{isomorphisms}, \text{all morphisms}))$ on $\dB$ to a double orthogonal factorization system on $\dE$. By construction, the obtained double orthogonal factorization system has left classes given by cocartesian morphisms and right classes given by pseudovertical morphisms. 
\end{proof}

We now present some applications of the results in this section.

\begin{example}[Double Elements Construction]
Following \citep{df}, we can associate to any lax double pseudofunctor $F\: \dD\op\to \Span(\Cat)$ (with $\dD$ a pseudo double category) its double elements construction $\Pi\: \dEl{F} \to \dD$. Recall that $\Pi_0$ is the (ordinary) elements construction of $F_0\: \dD_0\op\to \Cat$, while $\Pi_1$ is the elements construction of the composite
$$\dD_1\op\ar{F_1} \Span(\Cat)_1 \ar{\operatorname{apx}} \Cat,$$
where $\operatorname{apx}$ is the apex functor that sends any span to its middle object. Explicitly, by Theorem 3.46 of \citep{df}, the pseudo double category $\dEl{F}$ is given as follows:
\begin{itemize}
\item[-] objects: pairs $(C,X)$ with $C\in \dD_0$ and $X\in F(C)$;
\item[-] arrows: pairs $(f,\overline{f})\: (C,X) \to (D,Y)$ with $f\: C\to D$ in $\dD$ and $\overline{f}\:X \to f^{\star}Y$ in $F(C)$;
\item[-] proarrows: pairs $(m,\overline{m})\: (C,X) \proarrow (D,Y)$ with $m\: C \proarrow D$ in $\dD$ and $\overline{m} \in F(m)$, such that $D_m (\overline{m})=X$ and $U_m (\overline{m})=Y$ where the span
\begin{cd}[1.5][1.5]
{} \& {F(B)} \\
{F(m)} \arrow[ru,"{U_m}"] \arrow[rd,"{D_m}"'] \& {} \\
{} \& {F(A)}
\end{cd}
is the image of the proarrow $m$ under $F$;
\item[-] double cells: pairs $(\theta, \overline{\theta})$ displayed 
\sq[d][5][5][{(\theta, \overline{\theta})}]{(A,X)}{(C,Z)}{(B,Y)}{(D,W)}{(f,\overline{f})}{(m, \overline{m})}{(n,\overline{n})}{(g,\overline{g})}
with
\sq[d][5.8][5.8][\theta]{A}{C}{B}{D}{f}{m}{n}{g}
a double cell in $\dD$ and $\overline{\theta}\: \overline{m} \to \theta^{\star}\overline{n}$ an arrow of $F(m)$ such that $D_m (\overline{\theta}) =\overline{f}$ and $U_m(\overline{\theta}) =\overline{g}$.
\end{itemize}

\corx\ref{pseudovertcartdofs} yields a ((pseudo-vertical,cartesian),(pseudo-vertical,cartesian)) DOFS on $\dEl{F}$, by lifting the trivial DOFS on $\dD$ along $\Pi$. In terms of the explicit description of the double elements construction, we obtain the following factorizations. By the proof of \prox\ref{liftalongfib}, an arrow $(f,\overline{f})\: (C,X) \to (D,Y)$  in $\dEl{F}$  factorizes as the composite
$$(C,X) \ar{(\id{}, \overline{f})} (C, f^{\star}Y) \ar{(f, \id{})} (D,Y),$$
since $(f,\id{})$ is the chosen cartesian lifting of $f$ to $Y$ along the fibration $\Pi_0\:\dEl{F}_0 \to \dD_0$ and the arrow $(\id{}, \overline{f})$ is the vertical arrow induced by cartesianity of $(f,\id{})$. A double cell 
\sq[d][5][5][{(\theta, \overline{\theta})}]{(A,X)}{(C,Z)}{(B,Y)}{(D,W)}{(f,\overline{f})}{(m, \overline{m})}{(n,\overline{n})}{(g,\overline{g})}
then factorizes as the composite 
\begin{cd}[5][5]
{(A,X)} \arrow[r,"{(\id{}, \overline{\theta})}", phantom, shift right=5.5ex] \arrow[r,"{(\id{},f)}"] \arrow[d,"{(m,\overline{m})}"',proarrow ]\& {(A,f^{\star} Z)} \arrow[r,"{(\theta, \id{})}", phantom, shift right=5.5ex] \arrow[r,"{(f, \id{})}"] \arrow[d,"{(m, \theta^{\star} \overline{n})}"{pos=0.3}, proarrow]  \& {(C, Z)} \arrow[d,"{(n, \overline{n})}", proarrow] \\
{(B,Y)} \arrow[r,"{(\id{},\overline{g})}"']\& {(B,g ^{\star}W)} \arrow[r,"{(g, \id{})}"'] \& {(D,W)} 
\end{cd}
where the double cell $(\theta, \id{})$ is the chosen cartesian lifting of $\theta$ along $\Pi_1\:\dEl{F}_1 \to \dD_1$ and the vertical double cell $(\id{}, \overline{\theta})$ is the one induced by cartesian property of $(\theta, \id{})$.
\end{example}

\begin{example}[Domain Double Fibration] \label{dom}
Consider the domain double fibration
$$\dom\: \dD^{\to} \to \dD$$
that sends:
\begin{itemize}
\item an arrow $A\ar{f}B$ to $A$;
\item a commutative square 
\begin{tikzcd}
	A & {A'} \\
	B & {B'}
	\arrow["u"{pos=0.6}, from=1-1, to=1-2]
	\arrow["f"', from=1-1, to=2-1]
	\arrow["g", from=1-2, to=2-2]
	\arrow["v"'{pos=0.6}, from=2-1, to=2-2]
\end{tikzcd}
to $A\ar{u} A'$;
\item a double cell \begin{tikzcd}
	A & B \\
	C & D
	\arrow[""{name=0, anchor=center, inner sep=0}, "f", from=1-1, to=1-2]
	\arrow["m"', "\shortmid"{marking}, from=1-1, to=2-1]
	\arrow["n", "\shortmid"{marking}, from=1-2, to=2-2]
	\arrow[""{name=1, anchor=center, inner sep=0}, "h"', from=2-1, to=2-2]
	\arrow["\alpha"{description}, draw=none, from=0, to=1]
\end{tikzcd} 
to \begin{tikzcd}
	A \\
	C
	\arrow["m", "\shortmid"{marking}, from=1-1, to=2-1]
\end{tikzcd}
\item a cube 
\begin{tikzcd} [row sep=2.5ex, column sep=2.5ex]
	& {A'} && {B'} \\
	A && B \\
	& {C'} && {D'} \\[0.7ex]
	C && D
	\arrow[""{name=0, anchor=center, inner sep=0}, "g", from=1-2, to=1-4]
	\arrow["{m'}"'{pos=0.4}, "\shortmid"{marking}, from=1-2, to=3-2]
	\arrow[""{name=1, anchor=center, inner sep=0}, "{n'}"{pos=0.4}, "\shortmid"{marking}, from=1-4, to=3-4]
	\arrow[""{name=2, anchor=center, inner sep=0}, "u"{description}, from=2-1, to=1-2]
	\arrow[""{name=3, anchor=center, inner sep=0}, "f"{pos=0.7}, shift right, from=2-1, to=2-3]
	\arrow["m"', "\shortmid"{marking}, from=2-1, to=4-1]
	\arrow["y"{description}, from=2-3, to=1-4]
	\arrow[""{name=4, anchor=center, inner sep=0}, "n", "\shortmid"{marking}, from=2-3, to=4-3]
	\arrow[""{name=5, anchor=center, inner sep=0}, "k"{pos=0.4}, shift left, from=3-2, to=3-4]
	\arrow[""{name=6, anchor=center, inner sep=0}, "v"{description}, from=4-1, to=3-2]
	\arrow[""{name=7, anchor=center, inner sep=0}, "h"', from=4-1, to=4-3]
	\arrow["z"{description}, from=4-3, to=3-4]
	\arrow["\beta"{description, pos=0.3}, draw=none, from=0, to=5]
	\arrow["\sigma"{description}, draw=none, from=2, to=6]
	\arrow["\alpha"{description, pos=0.6}, shift left=4, draw=none, from=3, to=7]
	\arrow["\tau"{description}, draw=none, from=4, to=1]
\end{tikzcd} to \begin{tikzcd}
	A & {A'} \\
	C & {C'}
	\arrow[""{name=0, anchor=center, inner sep=0}, "u", from=1-1, to=1-2]
	\arrow["m"', "\shortmid"{marking}, from=1-1, to=2-1]
	\arrow["{m'}", "\shortmid"{marking}, from=1-2, to=2-2]
	\arrow[""{name=1, anchor=center, inner sep=0}, "v"', from=2-1, to=2-2]
	\arrow["\sigma"{description}, draw=none, from=0, to=1]
\end{tikzcd}
\end{itemize}

Given a double orthogonal factorization system $\dofs{\cL}{\cR}$ on $\dD$, \thex\ref{liftalongdblfib} lets us lift it to a double orthogonal factorization system 
$((\cL^{\dom}_0,\cR^{\dom}_0),(\cL^{\dom}_1,\cR^{\dom}_1))$. 

In particular, since a square
\begin{eqD} {square}
\begin{tikzcd}
	A & {A'} \\
	B & {B'}
	\arrow["u"{pos=0.6}, from=1-1, to=1-2]
	\arrow["f"', from=1-1, to=2-1]
	\arrow["g", from=1-2, to=2-2]
	\arrow["v"'{pos=0.6}, from=2-1, to=2-2]
\end{tikzcd}
\end{eqD}
is cartesian if and only if $v$ is an isomorphism, we obtain that $\cR^{\dom}_0$ is the collection of commutative squares with top morphism in $\cR_0$ and bottom morphism an  isomorphism while $\cL^{\dom}_0$ is the collection of commutative squares with top morphism in $\cL_0$. Moreover, the square in diagram (\ref{square}) factorizes as the composite
\begin{eqD*}
\begin{tikzcd}[row sep=2.5ex]
	A & {Q(u)} & {A'} \\
	& {A'} \\
	B & {B'} & {B'}
	\arrow["{\ell_u}", two heads, from=1-1, to=1-2]
	\arrow["f"', from=1-1, to=3-1]
	\arrow["{r_u}", hook, from=1-2, to=1-3]
	\arrow["{r_u}", hook, from=1-2, to=2-2]
	\arrow["g", from=1-3, to=3-3]
	\arrow["g", from=2-2, to=3-2]
	\arrow[from=3-1, to=3-2]
	\arrow[no head, from=3-2, to=3-3]
	\arrow[shift right, no head, from=3-2, to=3-3]
\end{tikzcd}
\end{eqD*}
Analogously, a cube 
\begin{eqD}{cube}
\begin{tikzcd}[row sep=2.5ex, column sep=2.5ex]
	& {A'} && {B'} \\
	A && B \\
	& {C'} && {D'} \\[0.7ex]
	C && D
	\arrow[""{name=0, anchor=center, inner sep=0}, "g", from=1-2, to=1-4]
	\arrow["{m'}"'{pos=0.4}, "\shortmid"{marking}, from=1-2, to=3-2]
	\arrow[""{name=1, anchor=center, inner sep=0}, "{n'}"{pos=0.4}, "\shortmid"{marking}, from=1-4, to=3-4]
	\arrow[""{name=2, anchor=center, inner sep=0}, "u"{description}, from=2-1, to=1-2]
	\arrow[""{name=3, anchor=center, inner sep=0}, "f"{pos=0.7}, shift right, from=2-1, to=2-3]
	\arrow["m"', "\shortmid"{marking}, from=2-1, to=4-1]
	\arrow["y"{description}, from=2-3, to=1-4]
	\arrow[""{name=4, anchor=center, inner sep=0}, "n", "\shortmid"{marking}, from=2-3, to=4-3]
	\arrow[""{name=5, anchor=center, inner sep=0}, "k"{pos=0.4}, shift left, from=3-2, to=3-4]
	\arrow[""{name=6, anchor=center, inner sep=0}, "v"{description}, from=4-1, to=3-2]
	\arrow[""{name=7, anchor=center, inner sep=0}, "h"', from=4-1, to=4-3]
	\arrow["z"{description}, from=4-3, to=3-4]
	\arrow["\beta"{description, pos=0.3}, draw=none, from=0, to=5]
	\arrow["\sigma"{description}, draw=none, from=2, to=6]
	\arrow["\alpha"{description, pos=0.6}, shift left=4, draw=none, from=3, to=7]
	\arrow["\tau"{description}, draw=none, from=4, to=1]
\end{tikzcd}
\end{eqD}
is cartesian if and only if the cell $\tau$ is horizontally invertible and thus $\cR^{\dom}_1$ is the collection of cubes  with left face in $\cR_1$ and right face horizontally invertible, while $\cL^{\dom}_1$ is the collection of cubes with left face in $\cL_1$. Moreover, a factorization of the cube of diagram (\ref{cube}) is given by the composite:
% Matt: I was able to clean up the long equality arrows as we talked about in the meeting
\[\begin{tikzcd}[row sep=3.9ex,column sep=4ex]
	&&&[-4.1ex]&&&[-4.1ex]&& {} \\
	\\
	&&&&&& {A'} && {B'} \\
	&&& {Q(u)} && {B'} \\
	A && B &&&& {C'} && {D'} \\
	&&& {Q(v)} && {D'} \\
	C && D
	\arrow[""{name=0, anchor=center, inner sep=0}, "g", from=3-7, to=3-9]
	\arrow["{{m'}}", "\shortmid"{marking}, from=3-7, to=5-7]
	\arrow["{{n'}}"', "\shortmid"{marking}, from=3-9, to=5-9]
	\arrow[""{name=1, anchor=center, inner sep=0}, "{{r_u}}"{description}, hook, from=4-4, to=3-7]
	\arrow["{{r_u}}"{description}, hook, from=4-4, to=4-6]
	\arrow["\shortmid"{marking}, from=4-4, to=6-4]
	\arrow[shift right, equals, from=4-6, to=3-9]
	\arrow[""{name=2, anchor=center, inner sep=0}, "{{n'}}"{pos=0.4}, "\shortmid"{marking}, from=4-6, to=6-6]
	\arrow[""{name=3, anchor=center, inner sep=0}, "{{\ell_u}}"{description}, two heads, from=5-1, to=4-4]
	\arrow[""{name=4, anchor=center, inner sep=0}, "f"{pos=0.7}, shift right, from=5-1, to=5-3]
	\arrow["m"', "\shortmid"{marking}, from=5-1, to=7-1]
	\arrow["y"{description}, from=5-3, to=4-6]
	\arrow[""{name=5, anchor=center, inner sep=0}, "n"', "\shortmid"{marking}, from=5-3, to=7-3]
	\arrow[""{name=6, anchor=center, inner sep=0}, "k", shift left, from=5-7, to=5-9]
	\arrow[""{name=7, anchor=center, inner sep=0}, "{{r_v}}"{description}, hook, from=6-4, to=5-7]
	\arrow["{{k\circ r_v}}"{description}, from=6-4, to=6-6]
	\arrow[equals, from=6-6, to=5-9]
	\arrow[""{name=8, anchor=center, inner sep=0}, "{{\ell_v}}"{description}, two heads, from=7-1, to=6-4]
	\arrow[""{name=9, anchor=center, inner sep=0}, "h"', from=7-1, to=7-3]
	\arrow["z"{description}, from=7-3, to=6-6]
	\arrow["\beta"{description}, draw=none, from=0, to=6]
	\arrow["{{r_{\sigma}}}"{description}, draw=none, from=1, to=7]
	\arrow["{{\ell_{\sigma}}}"{description}, draw=none, from=3, to=8]
	\arrow["\alpha"{description}, draw=none, from=4, to=9]
	\arrow["\tau"{description}, draw=none, from=5, to=2]
\end{tikzcd}\]

Furthermore, we notice that if $\dofs{\cL}{\cR}$ is the trivial double orthogonal factorization system, the obtained factorizations coincide with the factorization system of the free algebra for the squaring monad (see Construction \ref{conssquaringmonaddouble}).
\end{example}

\begin{example}[Codomain Double Opfibration]\label{codop}
Dual to Example \ref{dom}, we can lift a double orthogonal factorization system $\dofs{\cL}{\cR}$ on $\dD$ along the double opfibration
$$\cod\: \dD^{\to} \to \dD$$
defined dually to $\dom\: \dD^{\to} \to \dD$. A square is cocartesian exactly when its top morphism is an isomorphism. So, by \thex\ref{liftalongdblopfib}, $\cL^{\cod}_0$ is the collection of commutative squares with bottom morphism in $\cL_0$ and top morphism an isomorphic while $\cR^{\cod}_0$ is the collection of commutative squares with bottom morphism in $\cR_0$. And analogously for double cells. The obtained factorizations of squares and cubes in $\dD^{\to}$ are dual to the ones described in Example \ref{dom}.
\end{example}

\begin{example}[Codomain Double Fibration]\label{codfib}
Let $\dD$ be a pseudo-category in the 2-category of finitely-complete categories, limit-preserving functors and natural transformations. Assume also that source and target (strictly) preserve finite limits and identity and composition preserve finite limits up to isomorphism.

Under these hypotheses, the codomain double functor $\cod\: \dD^{\to} \to \dD$ has a structure of double fibration as well, with cartesian liftings given by pullbacks (see Proposition 2.29 of \citep{df}). Cartesian squares are exactly pullback squares and cartesian cubes as pullback squares of double cells in $\dD_1$.

Given a double orthogonal factorization system $\dofs{\cL}{\cR}$ on $\dD$, its lifting along the double fibration $\cod$ produces a double orthogonal factorization system on $\dD^{\to}$ as follows. The left class of squares consists of the squares with bottom morphism in $\cL_0$, while the right class is the collection of pullback squares with bottom morphism in $\cR_0$. And analogously for cubes, considering pullbacks in $\dD_1$.
The square 
\begin{eqD*} 
\begin{tikzcd}
	A & {A'} \\
	B & {B'}
	\arrow["u"{pos=0.6}, from=1-1, to=1-2]
	\arrow["f"', from=1-1, to=2-1]
	\arrow["g", from=1-2, to=2-2]
	\arrow["v"'{pos=0.6}, from=2-1, to=2-2]
\end{tikzcd}
\end{eqD*}
factorizes as the composite 
\begin{eqD*}
\begin{tikzcd}
	A & {A'\times_{B'} Q(v)} & {A'} \\
	B & {Q(v)} & {B',}
	\arrow["a", from=1-1, to=1-2]
	\arrow["f"', from=1-1, to=2-1]
	\arrow[from=1-2, to=1-3]
	\arrow[from=1-2, to=2-2]
	\arrow["\lrcorner"{anchor=center, pos=0.125}, draw=none, from=1-2, to=2-3]
	\arrow["g", from=1-3, to=2-3]
	\arrow["{\ell_v}"', two heads, from=2-1, to=2-2]
	\arrow["{r_v}"', hook, from=2-2, to=2-3]
\end{tikzcd}
\end{eqD*}
where $a$ is the unique morphism induced by the universal property of the pullback $A'\times_{B'} Q(v)$ starting from the pair $(u,\ell_v \c f)$. Factorizations of cubes are given analogously.
\end{example}

\begin{example}[Image Double Fibration]\label{imagedofs}
Consider the image double functor $\operatorname{Im}\: \Span \to \Rel$ such that $\operatorname{Im}_0\: \Set \to \Set$ is the identity and $\operatorname{Im}_1\: \Span_1 \to \Rel_1$ sends a span $A \xleftarrow{u} S \ar{v} B$ to the image $\operatorname{Im}(u,v)$ of $S \ar{(u,v)} A \times B$ (seen as a relation) and a morphism of spans

\begin{equation}\label{m.spans}\begin{tikzcd}
	A & C \\
	S & R \\
	B & D
	\arrow["f", from=1-1, to=1-2]
	\arrow["u", from=2-1, to=1-1]
	\arrow["m", from=2-1, to=2-2]
	\arrow["v"', from=2-1, to=3-1]
	\arrow["u'"', from=2-2, to=1-2]
	\arrow["v'", from=2-2, to=3-2]
	\arrow["g"', from=3-1, to=3-2]
\end{tikzcd}\end{equation}
to the morphism of relations 

\[\begin{tikzcd}
	{\operatorname{Im}(u,v)} && {\operatorname{Im}(u',v')} \\
	{A\times B} && {C\times D}
	\arrow["{\overline{m}}", from=1-1, to=1-3]
	\arrow[hook, from=1-1, to=2-1]
	\arrow[hook, from=1-3, to=2-3]
	\arrow["{f \times g}"', from=2-1, to=2-3]
\end{tikzcd}\]
where $\overline{m}\: \operatorname{Im}(u,v) \to \operatorname{Im}(u',v')$ is the function defined by,
$$(u(s), v(s))\mapsto(u'(m(s)), v'(m(s))).$$ 
The double functor $\operatorname{Im}$ is a double fibration (Example 2.25, \citep{df}). A morphism of spans $(f,m,g)$ as in (\ref{m.spans}) is cartesian precisely when the square

\[\begin{tikzcd}
	S  & R \\
	{\operatorname{Im}(u,v)} & {\operatorname{Im}(u',v')}
	\arrow["m", from=1-1, to=1-2]
	\arrow["{(u,v)}"', from=1-1, to=2-1]
	\arrow["\lrcorner"{anchor=center, pos=0.125}, draw=none, from=1-1, to=2-2]
	\arrow["{(u',v')}", from=1-2, to=2-2]
	\arrow["{\overline{m}}"', from=2-1, to=2-2]
\end{tikzcd}\]
is a pullback. 
We can now consider two distinct double orthogonal factorization systems on $\Rel$: the one defined by extensions and the one defined by restrictions. The one defined by extensions was also described in detail in Example \ref{examplerel}.
Lifting this DOFS along the image double fibration, we obtain the following DOFS on $\Span$. The left and right classes of arrows are given by surjective and injective functions as expected. The left class of double cells $\cL_1^{\operatorname{Im}}$ is given by morphisms of spans $(f,m,g)$ such that $f$ and $g$ are surjective and $\overline{m}$ is surjective, while the right class $\cR_1^{\operatorname{Im}}$ is given by those morphisms of spans $(f,m,g)$ such that $f$ and $g$ are injective, $\overline{m}$ is injective and $(f,m,g)$ is cartesian. We can now compare this DOFS on $\Span$ with the one described in Example \ref{examplespan}. Since the surjectivity of $m$ implies the surjectivity of $\overline{m}$ (but the two are not equivalent) we have that $\cL_1^{\operatorname{Im}}$ strictly contains the left class of double cells of Example \ref{examplespan}. On the other hand, as expected, since the injectivity of $\overline{m}$ implies the injectivity of $m$, we have that $\cR_1^{\operatorname{Im}}$ is strictly contained in the right class of double cells described in Example \ref{examplespan}. To obtain the morphisms between this DOFS on $\Span$ and the system described in Example \ref{examplespan} as double factorization systems over a common factorization system on the arrows, we spell out the chosen factorizations of the cells in $\Span$ with this system.
We will denote this factorization by

\[\begin{tikzcd}
	A & {A'} && A & {\operatorname{Im}(f)} & {A'} \\
	S & R & {=} & S & {\mathcal{I}m(S)} & R \\
	B & {B'} && B & {\operatorname{Im}(g)} & {B'}
	\arrow["f", from=1-1, to=1-2]
	\arrow["{\ell_f}", from=1-4, to=1-5]
	\arrow["{r_f}", from=1-5, to=1-6]
	\arrow["u", from=2-1, to=1-1]
	\arrow["m", from=2-1, to=2-2]
	\arrow["v"', from=2-1, to=3-1]
	\arrow["{u'}"', from=2-2, to=1-2]
	\arrow["{v'}", from=2-2, to=3-2]
	\arrow["u", from=2-4, to=1-4]
	\arrow["{\lambda_m}", from=2-4, to=2-5]
	\arrow["v"', from=2-4, to=3-4]
	\arrow["{q_1}"', from=2-5, to=1-5]
	\arrow["{\rho_m}", from=2-5, to=2-6]
	\arrow["{q_2}", from=2-5, to=3-5]
	\arrow["{u'}"', from=2-6, to=1-6]
	\arrow["{v'}", from=2-6, to=3-6]
	\arrow["g"', from=3-1, to=3-2]
	\arrow["{\ell_g}"', from=3-4, to=3-5]
	\arrow["{r_g}"', from=3-5, to=3-6]
\end{tikzcd}\]
To determine the central object of this diagram and the arrows in and out of it, 
let $\operatorname{Im}(u,v)\ar{\ell_{\overline{m}}}X\ar{r_{\overline{m}}}\operatorname{Im}(u',v')$ be the epi-mono factorization of $\overline{m}$. Then define
$\mathcal{I}m(S)$ as the pullback in the following diagram,

\[\begin{tikzcd}
	S \\
	& {\mathcal{I}m(S)} & T \\
	{\operatorname{Im}(u,v)} & X & {\operatorname{Im}(u',v')} \\
	{A\times B} & {\operatorname{Im}(f)\times\operatorname{Im}(g)} & {A'\times B'}
	\arrow["{\lambda_m}", dashed, from=1-1, to=2-2]
	\arrow["m", curve={height=-12pt}, from=1-1, to=2-3]
	\arrow["{(u,v)}"', from=1-1, to=3-1]
	\arrow["{\rho_m}", hook, from=2-2, to=2-3]
	\arrow["{(q_1,q_2)}"', from=2-2, to=3-2]
	\arrow["\lrcorner"{anchor=center, pos=0.125}, draw=none, from=2-2, to=3-3]
	\arrow["{(u',v')}", from=2-3, to=3-3]
	\arrow["{\ell_{\overline{m}}}", two heads, from=3-1, to=3-2]
	\arrow[hook, from=3-1, to=4-1]
	\arrow["{r_{\overline{m}}}", hook, from=3-2, to=3-3]
	\arrow[hook, from=3-2, to=4-2]
	\arrow[hook, from=3-3, to=4-3]
	\arrow["{\ell_f\times\ell_g}"', two heads, from=4-1, to=4-2]
	\arrow["{r_f\times r_g}"', hook, from=4-2, to=4-3]
\end{tikzcd}\]
Note that the arrow $\mathcal{I}m(S)\to X$ is of the form $(q_1,q_2)$ for unique arrows $q_1\colon \mathcal{I}m(S)\to \im(f)$ and $q_2\colon \mathcal{I}m(S)\to \im(g)$ because $X\hookrightarrow \im(f)\times\im(g)$.
To compare the factorization $m=r_m\circ \ell_m$ (the chosen epi-mono factorization) with $m=\rho_m\circ\lambda_m$, there is an epi-mono factorization of $\lambda_m$ of the form $\lambda_m=\lambda'_m\circ \ell_m.$
Then there is a morphism  from this lifted system to the system from Example \ref{examplespan} given by the cells

\[\begin{tikzcd}
	{\operatorname{Im}(f)} & {\operatorname{Im}(f)} \\
	{\mathcal{I}m(S)} & {\operatorname{Im}(m)} \\
	{\operatorname{Im}(g)} & {\operatorname{Im}(g)}
	\arrow["{\operatorname{Id}}"', from=1-2, to=1-1]
	\arrow["{q_1}", from=2-1, to=1-1]
	\arrow["{q_2}"', from=2-1, to=3-1]
	\arrow["{Q(u,u')}"', from=2-2, to=1-2]
	\arrow["{\lambda'_m}"', from=2-2, to=2-1]
	\arrow["{Q(v,v')}", from=2-2, to=3-2]
	\arrow["{\operatorname{Id}}", from=3-2, to=3-1]
\end{tikzcd}\]

\end{example}

\bmhead{Acknowledgements}

This paper is the result of research that started as a project for the Adjoint School 2024 at Oxford University. We thank the sponsors of the Adjoint School for their support. Furthermore, travel expenses for the second, fifth and sixth authors were covered by the NSERC Discovery Grant of the sixth author. The fourth author also acknowledges financial support from the SECIHTI program 001458.

\section*{Declarations}

\bmhead{Competing Interests:} We declare that there is no conflict of interest for this manuscript.

\bibliography{dof}% common bib file

\end{document}